\documentclass[11pt, reqno]{amsart}
\usepackage{amsmath}
\usepackage{amsthm}
\usepackage{amssymb}
\usepackage{amscd,mathrsfs}
\usepackage{amsbsy}
\usepackage{epsfig}
\usepackage{graphicx}
\usepackage{psfrag}
\usepackage{changebar}
\usepackage{marginnote}
\usepackage{bbm}
\usepackage{hyperref}

\usepackage{geometry}

\theoremstyle{plain}
\newtheorem{theorem}{Theorem}[section]
\newtheorem{remark}[theorem]{Remark}
\newtheorem{proposition}[theorem]{Proposition}
\newtheorem{lemma}[theorem]{Lemma}
\newtheorem{corollary}[theorem]{Corollary}
\newtheorem{sublemma}[theorem]{Sublemma}
\newtheorem*{claim}{Claim}

\theoremstyle{definition}

\newfont\bbf{msbm10 at 12pt}
\def\eps{\varepsilon}
\def\R{{\mathbb R}}

\def\N{{\mathbb N}}

\def\B{{\mathcal B}}

\def\P{{\mathcal P}}

\def\D{{\mathcal D}}

\def\Q{{\mathcal Q}}

\def\sm{\setminus}

\def\diam{\mbox{\rm diam} }

\def\le{\leqslant}
\def\ge{\geqslant}

\newcommand{\I}{\mathring{I}}

\newcommand{\hI}{\mathring{I}}
\newcommand{\f}{\mathring{f}}
\newcommand{\hDelta}{\mathring{\Delta}}
\newcommand{\tH}{\tilde{H}}

\newcommand{\Lp}{\mathcal{L}}
\newcommand{\Np}{\mathcal{N}}
\newcommand{\bm}{\overline{m}}

\newcommand{\blambda}{\overline{\lambda}}
\newcommand{\vf}{\varphi}
\newcommand{\ve}{\varepsilon}
\newcommand{\tm}{\tilde{m}}

\newcommand{\tg}{\tilde{g}}

\newcommand{\tpsi}{\tilde{\psi}}
\newcommand{\hLp}{\mathring{\Lp}}
\newcommand{\hNp}{\mathring{\Np}}
\newcommand{\hF}{\mathring{F}}
\newcommand{\hf}{\mathring{f}}
\newcommand{\hY}{\mathring{Y}}
\newcommand{\hX}{\mathring{X}}

\newcommand{\dlj}{\Delta_{\ell,j}}

\newcommand{\beq}{\begin{equation}}
\newcommand{\eeq}{\end{equation}}

\def\M{\mathcal{M}}

\numberwithin{equation}{section}

\begin{document}

\title[Slow and fast escape for open intermittent maps]{Slow and fast escape for open intermittent maps}
\author[M.F. Demers]{Mark F. Demers}
\address{Mark F. Demers\\ Department of Mathematics \\
Fairfield University\\
Fairfield, CT 06824 \\
USA}\email{\href{mailto:mdemers@fairfield.edu}{mdemers@fairfield.edu}}
\urladdr{\url{http://faculty.fairfield.edu/mdemers}}

\author[M. Todd]{Mike Todd}
\address{Mike Todd\\ Mathematical Institute\\
University of St Andrews\\
North Haugh\\
St Andrews\\
KY16 9SS\\
Scotland} \email{\href{mailto:m.todd@st-andrews.ac.uk}{m.todd@st-andrews.ac.uk}}
\urladdr{\url{http://www.mcs.st-and.ac.uk/~miket/}}

\date{\today}

\begin{abstract} 
If a system mixes too slowly, putting a hole in it can completely destroy the richness of the dynamics.  Here we study this instability for a class of intermittent maps with a family of slowly mixing measures.  We show that there are three regimes: 1) standard hyperbolic-like behavior where the rate of mixing is faster than the rate of escape through the hole, there is a unique limiting absolutely continuous conditionally invariant measure (accim) and there is a complete thermodynamic description of the dynamics on the survivor set; 2) an intermediate regime, where the rate of mixing and escape through the hole coincide, limiting accims exist, but much of the thermodynamic picture breaks down; 3) a subexponentially mixing regime where the slow mixing means that mass simply accumulates on the parabolic fixed point.  We give a complete picture of the transitions and stability properties (in 
the size of the hole and as we move through the family) in this class of open systems.
In particular we are able to recover a form of stability in the third regime above 
via the dynamics on the survivor set, even when
no limiting accim exists.
\end{abstract}

\thanks{MD was partially supported by NSF grant DMS 1362420.  This project was started 
as part of an RiGs grant through ICMS, Scotland.  The authors would like to thank ICMS
for its generous hospitality.  They would also like to thank AIM (workshop on Stochastic Methods for Non-Equilibrium Dynamical Systems) and the ICERM Semester Program on Dimension 
and Dynamics where some of this work was carried out.  We thank the referees for useful suggestions.}

\maketitle

\section{Introduction and Statement of Results}
\label{sec:intro}

Dynamical systems with holes are examples of systems in which the domain is not
invariant under the dynamics.  Such systems arise in a variety of contexts:   For example, in the
study of non-attracting invariant sets, as well as in
non-equilibrium dynamical systems, in which mass or energy is allowed to enter or escape. 
In 
this latter context, a system with a hole can be viewed as a component of a much larger system
of interacting components.  
Examples of such studies include metastable states
\cite{KL09, GHW11, BV12, DW12}, coherent sets in nonautonomous systems \cite{FrP},
and diffusion in extended systems \cite{DGKK}.

To date, systems with holes have been studied principally
in situations in which the rate of mixing of the closed system (before the
introduction of the hole) is exponential and therefore the rate of escape from the system 
is also exponential.  Such systems include expanding maps 
\cite{pianigiani yorke, collet ms1, chernov bedem, liverani maume}, Smale horseshoes 
\cite{cencova},
Anosov diffeomorphisms \cite{chernov mark1, chernov mt1}, certain unimodal maps \cite{BruDemMel10},
and dispersing billiards \cite{DWY1, demers billiards, demers infinite}, to name but a few.

In all these papers, the main focus is the existence and physical properties
of \emph{conditionally invariant measures}, which describe the limiting distribution of mass
conditioned on non-escape.  
Given a dynamical system, $(T,X, \mathscr{B})$,
one identifies a measurable set $H \in \mathscr{B}$ and studies the open system,
$\mathring{T}: \mathring{X} \to X$, where $\mathring{X} = X \setminus H$.
The $n$-step survivor sets are defined by $\mathring{X}^n = \bigcap_{i=0}^n T^{-i}(\mathring{X})$,
which correspond to the non-invariant domains of
the iterates of the map, $\mathring{T}^n = T^n|_{\mathring{X}^n}$.

A measure $\mu$ on $X$ is called conditionally invariant if 
\[
\frac{\mathring{T}_*\mu(A)}{\mathring{T}_*\mu(X)} := \frac{\mu(\mathring{X}^1 \cap T^{-1}(A))}{\mu(\mathring{X}^1)} = \mu(A) \qquad \mbox{for all $A \in \mathscr{B}$.}
\] 
If we set $\mu(\mathring{X}^1) = \lambda$, the relation above can be iterated to obtain
$\mathring{T}^n_*\mu(A) = \lambda^n \mu(A)$, so that a conditionally invariant measure
necessarily predicts an exponential rate of decay of mass from the open system.
Unfortunately, under quite general conditions, uncountably many such measures exist for 
any eigenvalue $\lambda \in (0,1)$, even if one restricts to measures absolutely continuous
with respect to a given reference measure \cite{DemYoung06}, so existence questions
are meaningless.  

In order to obtain a physically relevant measure, one fixes a reference measure $m$
and focuses on the existence and properties of limiting
distributions obtained by pushing forward $m$ and conditioning
on non-escape, i.e. studying limit points of the sequence
$\frac{\mathring{T}^n_*m}{m(\mathring{X}^n)}$.  For systems with exponential
rates of escape, such limiting distributions are often conditionally invariant measures which
describe the limiting dynamics with respect to a large class of reference measures, and enjoy 
many of the properties that equilibrium measures enjoy in closed systems.  Moreover,
in many cases the eigenvalue $\lambda$ associated with such measures 
describes the exponential {\em rate of escape} from the open system with respect to $m$,
\begin{equation}
\label{eq:esc def}
-\log \lambda = -\lim_{n \to \infty} \frac 1n \log m(\mathring{X}^n).
\end{equation}
Such limiting distributions have been constructed for all the specific systems listed above, 
under some assumptions on the size or geometry of the holes.

Recently, there has been interest in open systems exhibiting subexponential rates of escape 
\cite{dett1, altmann, dett2, FMS, KM16}, 
and in particular their relation to slowly mixing systems 
from non-equilibrium statistical mechanics \cite{yarmola}.  Such open systems exhibit
qualitatively different behavior from systems with exponential escape rates.  
For example, conditionally invariant measures no longer have a physical interpretation
as limiting distributions (although arbitrarily many still exist) and limit points
of $\frac{\mathring{T}^n_*m}{m(\mathring{X}^n)}$ are typically singular (with respect
to $m$) invariant measures
supported on $\mathring{X}^\infty := \bigcap_{i=0}^\infty T^{-i}(X \setminus H)$, the set
of points which never enter the hole \cite{DemFer13}.  From the point of view
of limiting distributions, systems with subexponential rates of escape are {\em unstable}
with respect to leaks in the system.

In the present paper, we introduce holes into a class of Manneville-Pomeau maps $f=f_\gamma$ 
of the unit interval with intermittent behavior.
We consider the dynamics of the open system from the point of view of
the family of geometric potentials, $t\phi = - t \log |Df|$, $t \in [0,1]$.  
When $t=1$, the conformal measure with respect to $\phi = -\log |Df|$
is Lebesgue measure, with respect to which these maps have polynomial rates of mixing. 
As such, the open system has no physically relevant conditionally invariant measure absolutely
continuous with respect to Lebesgue \cite{DemFer13}.
But for $t<1$, the maps admit conformal measures $m_t$ with respect to $t\phi$
that are exponentially mixing and
so have exponential rates of escape, where the mixing rate converges to zero as $t\to 1$, 
yielding an excellent test bed for the study of slow mixing with holes.  

Fixing a hole $H$ as described in Section~\ref{intro hole}, we are able to precisely characterize
the dynamics of the open system in terms of the parameter
$t$ in 3 distinct regimes:  $t \in [0,t^H)$, $t \in [t^H, 1)$, and $t=1$, where $t^H$ is the Hausdorff
dimension of the survivor set.  
\begin{itemize}
  \item  When $t \in [0, t^H)$, the escape rate \eqref{eq:esc def}
  with respect to $m_t$ is slower than the rate of 
  mixing of the closed system, and so the transfer operator associated with the open system
  has a spectral gap.  In this setting, the classical results proved for strongly hyperbolic open systems
  mentioned above are recovered (Theorem~\ref{thm:rate gap}).
  \item When $t \in [t^H,1)$, the escape rate with respect to $m_t$ equals the rate of mixing of the
  closed system and the associated transfer operator for the open system has no spectral gap; 
  however, averaged limit points of the form $\frac{\hf_*^n m_t}{|\hf_*^n m_t|}$,
  yield conditionally invariant measures, absolutely continuous with respect to $m_t$
  (Theorem~\ref{thm:limit point}).
  \item When $t=1$, the rate of escape with respect to Lebesgue is polynomial and 
  the sequence $\frac{\hf_*^n m_1}{|\hf_*^n m_1|}$ converges to the point
  mass at the neutral fixed point \cite{DemFer13}.
\end{itemize}

In order to recover a form of stability for the open system when $t=1$, we use an induced map
to construct an invariant measure on the survivor set.  When $t \in [0, t^H]$, these measures
maximize the pressure on the survivor set and satisfy an {\em escape rate formula}
(see Theorem~\ref{thm:rate gap}).  When $t \in (t^H,1]$, 
these measures do not maximize the pressure on the survivor set, but they do  
converge to the absolutely continuous equilibrium state for the closed system as the hole
shrinks to a point (Theorem~\ref{thm:stable}).  
Thus, although the system is unstable with respect to 
leaks from the point of view of the physical limit $\displaystyle \lim_{n \to \infty} \tfrac{\hf_*^n m_1}{|\hf_*^n m_1|}$
when $t=1$,
we are able to recover a type of stability from the point of view of these invariant measures
supported on the survivor set.

One of the principal tools we use is a Young tower,
which is a type of Markov extension for the open system.  It is of independent interest
that in order to obtain the sharp division between the regimes listed above, we significantly
strengthen previous results on Young towers with holes.  Specifically, we prove the
existence of a spectral gap for the associated transfer operator under a weak
{\em asymptotic} condition:  The escape rate from the tower is strictly less than the rate of decay in the
levels of the tower.  
Previous results \cite{demers tower, BruDemMel10, DWY1} assumed strong control on the amount of
mass lost at each step, while we are able to prove comparable results under this much weaker
and more natural condition.  Our results are in some sense optimal:  When the escape rate
equals the rate of decay in the levels of the tower, the essential spectral radius
and spectral radius of the associated transfer operator on the relevant function space coincide.  
This optimality suggests that these results provide a new paradigm for open non-uniformly 
hyperbolic systems in general.

The paper is organized as follows.  In the remainder of Section~\ref{sec:intro}, we state
our assumptions precisely, define the relevant terminology and state our main results.
In Section~\ref{sec:basic press} we provide some background and initial results on pressure,
while in Section~\ref{sec:volume} we prove an essential inequality relating the escape rate
to the difference in pressures between the open and closed systems. 
In Sections~\ref{sec:rate gap}, \ref{sec:large t} and \ref{stable proof}, 
we prove our main theorems in the three regimes
outlined above.  Section~\ref{sec:swallowing} contains some examples of large holes
that do not satisfy our conditions and some analysis of the dynamics in such cases.


\subsection{Class of maps}
\label{ssec:basic}

For $\gamma \in (0,1)$, we will study the class of Manneville-Pomeau maps defined by

\begin{equation*}
f=f_\gamma:x\mapsto \begin{cases} x(1+2^\gamma x^\gamma) & \text{ if } x\in [0, 1/2),\\
2x-1 & \text{ if } x\in [1/2, 1].\end{cases}
\end{equation*}
Such maps exhibit intermittent behavior due to the neutral fixed point 0 and
have been well-studied, most commonly from the point of view of Lebesgue
measure \cite{young poly, LSV1}, which is the conformal measure with respect to
the potential $\phi := - \log |Df|$.   We will be interested in the related family of 
geometric potentials $t\phi$, $t \in [0,1]$, and their associated pressures.

For a dynamical system $(T, X, \mathscr{B})$ with some measurable and metric structure and a measurable potential  $\psi:X\to [-\infty, \infty]$, we define the \emph{pressure} of this system to be
$$
P(\psi)=P_T(\psi):=\sup\left\{h_T(\mu) + \int \psi ~d\mu : \mu \in \M_T \text{ and } -\int \psi ~d\mu < \infty		\right\}.
$$
Here, $\M_T$ is the set of ergodic, $T$-invariant probability measures on $X$ and
$h_\mu(T)$ is the metric entropy of $\mu$.  Note that the restriction on the integral is to deal with cases where the system has infinite entropy, so that the sum defining the pressure may not make sense.

For our class of maps and potentials, we set $p(t):=P(t\phi)$.
Note that $p(t)=0$ for $t\ge 1$.  It is well-known that for $t\le 1$ there exists a unique $(t\phi-p(t))$-conformal measure $m_t$ (note that $m_1$ is Lebesgue measure).  Moreover,

\begin{itemize}
\item For $\gamma \in (0,1)$ and $t=1$, there is an equilibrium state $\mu_1$ for $\phi$ that
is absolutely continuous with respect to $m_1$.  The system is 
subexponentially mixing with respect to $\mu_1$.
\item For $t<1$, there exists a unique equilibrium state $\mu_t$ for $t\phi$, which is exponentially mixing and furthermore equivalent to $m_t$.
\end{itemize}

(These facts can be derived from \cite[Proposition 1]{Sar01a}; for an alternative perspective on parabolic systems see \cite[Chapter 8]{MauUrb03}.)
We will study the dynamics of the related open system with respect to this family of potentials,
taking their associated conformal measures $m_t$ as our reference measures.


\subsection{Introduction of holes}
\label{intro hole}

We next introduce a hole $H$ into the system, which in this paper will be 
a finite union of intervals. 
The sets $\hI^n = \cap_{i=0}^n f^{-i}(I \setminus H)$, $n \ge 0$, denote the set of points
that have not entered $H$ by time $n$. 
Define $\hf = f |_{\hI^1}$ to be the map with the hole and its iterates, $\hf^n := f^n|_{\hI^n}$.
The dynamics of this map define the open system. 

A particularly convenient form of hole is defined as follows. Let $\P_1$ be the standard renewal partition,  
 i.e., $Z\in \P_1$ implies $Z$ is an interval for which either $f(Z)\in \P_1$, or $Z=[1/2,1)$.  
We then let $\P_n:=\P_1 \vee\left(\bigvee_{k=0}^{n-1}f^{-k}\{[0,1/2), [1/2, 1)\}\right)$.  
We fix $N_0 > 0$ and then define a hole to be some collection of elements of $\P_{N_0}$: 
we call such a set a \emph{Markov hole}.  

Before formulating a condition on the hole, we introduce an induced map $\hF$, defined
as the first return map to $Y = [1/2,1]$ under $\hf$.  Let $\tau : Y \to \N$
denote the inducing time, so that either $\hF(x) = \hf^{\tau(x)}(x) \in Y$ or 
$\hF(x) = f^{\tau(x)}(x) \in H$.  
In the absence of a hole,
$F:Y \to Y$ would be a full-branched map; however, once a hole has been introduced,
$\hF$ is no longer full branched.  
Let $\mathcal{Q}$ denote the coarsest partition of $Y$ by
images of first returns of elements of $\P_{N_0}$.  Note that $\mathcal{Q}$ is a finite partition of 
$Y\setminus H$
due to our definition of $\P_{N_0}$.

The classical definitions of transitivity via open sets no longer make sense for the open
system\footnote{The open system can be decomposed into a disjoint union
of intervals $\hI = \bigcup_{n=1}^\infty \hI^{n-1} \setminus \hI^n$ (mod 0) such that
$f(\hI^{n-1} \setminus \hI^n) = \hI^{n-2} \setminus \hI^{n-1}$ and $f^n(\hI^{n-1} \setminus \hI^n)
\subset H$.}
 (since everything except the hole would be transient), so we adopt the
following combinatorial definition in terms of the Markov partition.

We say that $H$ is \emph{non-swallowing} if: 
\begin{enumerate}
\item  $\mathcal{Q}$ is {\em  transitive on elements}:  For each pair
$Q_1, Q_2 \in \Q$, $\exists n\in \N$ such that 
$\mathring{f}^n(Q_1) \cap Q_2 \neq \emptyset$;
\item  For all $\delta>0$,  $m_1(\hF(1/2, 1/2+\delta))>0$.
\end{enumerate}
Otherwise, we call the hole \emph{swallowing}. 
Define $X := Y \setminus H = \cup_{Q \in \Q} Q$.  Note that 
our definition of non-swallowing implies that  $\hF|_{X}$ is transitive on elements, but not
necessarily aperiodic; however, in Lemma~\ref{lem:nonswallow} we will show that $\hf$ is 
aperiodic on elements with this definition. 

Condition (2) ensures that the system has
repeated passes through a neighborhood of the neutral fixed point: 
Were this condition violated, much of the non-expansive behavior present in the 
closed system would be lost or trivialized.
Examples of swallowing holes and a brief description of the dynamics in these 
nontransitive cases are given in Section~\ref{sec:swallowing}.  For our main results, we will assume
that our hole is non-swallowing.

Below we record two important facts about the open system that follow from the
definition of non-swallowing.

\begin{lemma}
Let $X_i$, $i \in \N$, denote the maximal intervals in $[1/2, 1]$
on which $\hF$ is smooth and injective.
Then there exists $K_1>0$ so that 
$|\{\tau=n\}|\le K_1 n^{-(1+\frac1\gamma)}$ and $\#\{i : \tau(X_i)=n\} \le C_{N_0}$ for some 
$C_{N_0} \ge 1$ depending only on $N_0$.  Moreover, if $H$ is non-swallowing then there exists $K_2>0$ and $N\in \N$ such that for all $n\ge N$, $|\{\tau=n\}|\ge K_2 n^{-(1+\frac1\gamma)}$. 
\label{lem:basic tail}
\end{lemma}

\begin{proof}
 The first fact follows immediately from standard constructions.  See for example,
\cite{young poly, LSV1}.  The second uses the same constructions, with the added information that there are inducing domains arbitrarily close to $1/2$.
\end{proof}

\begin{lemma}
If $H$ is a non-swallowing Markov hole then the survivor set
$\hI^\infty := \cap_{n=0}^\infty \hI^n$ has positive entropy.
\label{lem:non-swal}
\end{lemma}

\begin{proof}
Let $A$ and $B$ be 1-cylinders for $\hF$.  By definition of non-swallowing, there exist 
$n_A,\ n_B\in \N$ be such that $\hf^{n_A}(A)=\hf^{n_B}(B)=X$.  Then 
our system contains a horseshoe with entropy at least $\log 2/\max\{n_A, n_B\}>0$.
\end{proof}

Given the potential $\phi=-\log|Df|$, set $\phi^H$ to be the punctured potential, 
\begin{equation*}
\phi^H(x) =\begin{cases} \phi(x) &\text{ if } x\in I \setminus H,\\
 -\infty &\text{ if } x\in H.\end{cases}
 \end{equation*}  Then let $p(t) = P(t\phi)$ and
$p^H(t):=P(t\phi^H)$.
Due to the neutral fixed point, it is clear that $p(t)\ge p^H(t) \ge 0$, for $t \in [0,1]$.
Also, since $\phi$
is bounded, the condition $- \int \phi^H \, d\nu< \infty$  is equivalent to $\nu(H) = 0$,
which implies the supremum in $P(t\phi^H)$ is over invariant measures 
supported on the survivor set, $\hI^\infty$.

\subsection{Transfer Operator}
\label{transfer}

We will study the evolution of measures from the point of view of the transfer
operators associated with our family of potentials.  Given the potential $t\phi$, we define
the associated transfer operator acting on $L^1(m_t)$ by
\[
\Lp_t \psi(x) := \Lp_{t\phi - p(t)} \psi(x)  = \sum_{y \in f^{-1}x} \psi(y) e^{t\phi(y) - p(t)} .  
\]
When we introduce a hole, the transfer operator for the open system corresponds to the
transfer operator for the punctured potential,
\[
\hLp_t \psi(x) := \Lp_{t\phi^H - p(t)} \psi(x) = \Lp_{t\phi - p(t)} (1_{\hI^1} \cdot \psi)(x) = \sum_{y \in \hf^{-1}x}
\psi(y) e^{t\phi(y) - p(t)} .
\]
Since $m_t$ is conformal with respect to $t\phi - p(t)$, we have
\[
\int \hLp_t^n  \psi \, dm_t = \int_{\hI^n} \psi \, dm_t,
\]
so that the spectral properties of $\hLp_t$ are tied to the rate of escape of the open system
with respect to $m_t$.

\subsection{Main results} 
\label{main results}

The standing assumptions of this section are that $f = f_\gamma$ is a map as described above and
$H$ is a non-swallowing hole.

One of the key quantities associated with an open system is the exponential rate of escape.
We will be primarily concerned with the rate of escape with respect to the conformal measures
$m_t$. 
To this end we define,
\[
\log \overline\lambda_t=\limsup_k\frac1k\log m_t(\I^k) \; \text{ and } \; \log \underline\lambda_t=\liminf_k\frac1k\log m_t(\I^k),
\]
and when these two quantities coincide we denote the common value  $\log \lambda_t$.

A fundamental relation between pressure and escape is given by the following proposition,
which we prove in Section~\ref{sec:volume}.

\begin{proposition}
\label{prop:volume}
For any $t \in [0,1]$, we have 
$$\log \underline \lambda_t \ge p^H(t) - p(t)  .$$
\end{proposition}
In fact, Corollary~\ref{cor:variational} shows that for the class of maps we study here, the inequality
above is always an equality.

Whether the punctured pressure $p^H(t)$ is positive or zero has a strong influence on the
dynamics of the open system with respect to the reference measure $m_t$.  The following
series of results characterizes this behavior in the relevant regimes.  To this end,
define $$t^H:=\sup\{t\in \R: p^H(t)>0\}.$$

By Lemma~\ref{lem:non-swal}, the following result applies in all non-swallowing cases.  

\begin{lemma}
If $\I^\infty$ has positive entropy, then $t^H \in (0,1]$ and $p^H(t^H)=0$.
\label{lem:Hdim}
\end{lemma}

\begin{proof}
First note that $t \mapsto p^H(t)$ is a continuous function of $t$ since $\phi$ is bounded.  
We have
$p^H(t) \le p(t) = 0$ for all $t \ge 1$, and indeed $p^H(1) =0$ as well, since for example $h_{\delta_0}(f)+\int\phi~d\delta_0=0$ where $\delta_0$ denotes the point mass at 0.  
Also, $p^H(0) > 0$ since the entropy on $\I^\infty$ is assumed to be positive. 
Thus, by the continuity of $p^H(t)$, $t^H$ is finite and is contained in $(0,1]$. 
\end{proof}

Note that techniques described later in this paper further show that $t^H<1$ for any non-trivial hole.

Our next proposition establishes $t^H$ as the Hausdorff dimension of the survivor set
and describes the behavior of the pressure function on both sides of
$t^H$. 
As in \cite{IomTod11}, we say that a potential $\phi$ is \emph{recurrent} if there exists a finite conservative $(\phi-P(\phi))$-conformal measure; and \emph{transient} otherwise.   Recurrence can also be related to induced potentials.
The potential for the induced map $F$ corresponding to $\phi$ is 
\begin{equation}
\Phi := \sum_{i=0}^{\tau-1} \phi \circ f^i.\label{eq:ind pot}
\end{equation}

The following is \cite[Lemma 4.1]{IomTod10}.  We note that the result only requires that the induced potential $\Phi$ has good distortion properties, for example is locally H\"older (see Section~\ref{ssec:sym}).

\begin{lemma}
$P(\Phi-P(\phi)\tau )\le0$.
\label{lem:IT ind neg}
\end{lemma}

As described in Section~\ref{ssec:sym} below, for our choice of inducing scheme, $\phi$ is recurrent if and only if $P(\Phi-P(\phi)\tau )=0$.

\begin{proposition}
Suppose that $H$ is a non-swallowing Markov hole.
\begin{itemize}
\item[(a)] $p^H(t^H)=0$ and $P(t^H \Phi^H) = 0$.
\item [(b)] $\dim_H(\I^\infty)=t^H$
\item[(c)] $t^H> \frac\gamma{1+\gamma}$.
\item[(d)]
If $t>t^H$, then  $p^H(t)=0$ and we have $P(t\Phi^H-p^H(t)\tau)<0$.   Hence $t\phi^H$ is transient.
\item[(e)] If $t < t^H$, then $p^H(t)>0$ and $P(t\Phi^H-p^H(t)\tau)=0$. 
Hence $t\phi^H$ is recurrent.
\item[(f)] $Dp^H(t^H)=0$ if $t^H\in (\gamma/(1+\gamma), 2\gamma/(1+\gamma)]$. 
Otherwise\footnote{By $D^-$, we mean the derivative with respect to $t$ from the left.} $D^-p^H(t^H)<0$.
\end{itemize}
\label{prop:tH}
\end{proposition}

Proposition~\ref{prop:tH} 
suggests that $t=t^H$ is a dividing line between qualitatively different
behaviors of the dynamics with respect to the conformal measures $m_t$.  
The following theorem demonstrates that the dynamics is strongly hyperbolic
in the regime $t \in [0,t^H)$.
The proof of the theorem uses the induced map $F : Y \circlearrowleft$ to construct
an extension of the open system, known as a Young tower, which we denote by
$\Delta$.  Young towers are defined precisely
in Section~\ref{sec:tower def}.

\begin{theorem}(The case $t \in [0, t^H)$: uniformly hyperbolic behavior.)
Fix $q > 0$ which will determine the class of $C^q$ functions that we will lift to $\Delta$.

If \footnote{The condition $-\log \bar\lambda_t < p(t)$ requires the rate of escape to be 
slow compared to the 
  pressure, which in this case coincides with the rate of mixing of the closed system with respect
  to the equilibrium state $\mu_t$.  For a given hole and fixed values of $\gamma$ and $t$,
  this condition can be verified numerically:  the pressure can be approximated via
  periodic orbits (for example, by adapting the ideas in \cite{polli}) and the escape rate can be approximated by
  volume estimates.   }
 $-\log\overline\lambda_t< p(t)$, then the following hold. 
\begin{enumerate}
\item $\lambda_t<1$ exists and is the spectral radius of the punctured transfer operator on the
tower.  The associated eigenvector projects to a nonnegative function $g_t$, which is bounded
away from zero
on $I \setminus H$ and satisfies $\hLp_t g_t = \lambda_t g_t$.
\item $p^H(t)=\log\lambda_t+p(t)>0$.
\item There is a unique $(t\phi^H-p^H(t))$-conformal measure $m_t^H$. This is singular with respect to
$m_t$ and supported on $\hI^\infty$.
\item The measure $\nu_t :=g_t m_t^H$ is an equilibrium state for $t\phi^H - p^H(t)$.   Moreover, 
 \[
 \nu_t(\psi) = \lim_{n\to \infty} \lambda_t^{-n} \int_{\hI^n} \psi g_t \, dm_t, \qquad
 \mbox{for all $\psi \in C^0(I)$.}
 \]
 \item The measure $\mu_t^H := g_t m_t$ is a conditionally invariant measure with eigenvalue 
$\lambda_t$
and is a limiting distribution in the following sense.
Let $\psi \in C^q(I)$ satisfy $\psi \ge 0$, with $\nu_t(\psi) >0$.  
Then
 \[
\left| \frac{\hLp^n_t \psi}{|\hLp^n_t \psi|_{L^1(m_t)}} - g_t \right|_{L^1(m_t)} \le C \sigma^n |\psi|_{C^q} 
\]
for some $C>0$ independent of $\psi$ and $\sigma<1$ depending only on $q$.  
\end{enumerate}
\label{thm:rate gap}
\end{theorem}

We remark that the convergence to $\mu_t^H$ described in (5)
above holds for a larger class of functions than $C^q(I)$.  In particular, it holds for
$g^0_t$, where $g^0_t$ is the invariant density defining the equilibrium state $\mu_t$
for the closed system with potential $t\phi - p(t)$; 
this also implies that the escape rates with respect
to both $m_t$ and $\mu_t$ are the same. 

The following variational principle relating the escape rate to the pressure for $t \in [0,1]$ is a simple
consequence of Proposition~\ref{prop:volume} and Theorem~\ref{thm:rate gap}.
 Notice that it also justifies identifying the condition
$- \log \overline \lambda_t < p(t)$ with the condition $t \in [0, t^H)$ in the statement
of Theorem~\ref{thm:rate gap}.

\begin{corollary}
\label{cor:variational}
For all $t \in [0,1]$, 
\[
\mbox{$-\log\overline\lambda_t< p(t)$ if and only if $p^H(t)>0$ if and only if $t<t^H$. }
\] 
Moreover, for all $t \in [0,1]$, $\lambda_t$ exists and $\log\lambda_t=p^H(t)-p(t)$.
\end{corollary}

Our next result shows that in the regime $t \in [0, t^H)$, the conditionally invariant measures
$\mu^H_t$ we construct vary continuously as the hole shrinks to a point. 

\begin{theorem}
\label{thm:limit}
Fix $z \in (0,1]$ and let $( H_i )_{i \in \N}$ be a nested sequence
of intervals which
are non-swallowing Markov holes for $f$ and for which $\cap_{i >0} H_i  = \{ z \}$.  
Then letting $- \log \lambda_t^{H_i}$ denote the associated escape rate and $g^{H_i}_t$
denote the (normalized) eigenvector associated to $\lambda^{H_i}_t$ from Theorem~\ref{thm:rate gap},
 we have
$\lambda_t^{H_i} \to 1$ and
$g^{H_i}_t \to g^0_t$ in $L^1(m_t)$ as $i \to \infty$, where $\mu_t = g^0_t m_t$ is the unique 
equilibrium measure for $t \phi - p(t)$.
\end{theorem}

It follows from Corollary~\ref{cor:variational} and Theorem~\ref{thm:limit} that 
$p^{H_i}(t) \to p(t)$ and $t^{H_i} \to 1$ as $i \to \infty$ (see Lemma~\ref{lem: eps 0 press 0}).  
Thus each fixed $t <1$ eventually
satisfies $t < t^{H_i}$ for all $H_i$ sufficiently small and so Theorem~\ref{thm:limit} 
implies $\mu_t$ is stable with
respect to small leaks in the system.

\begin{remark}
A natural question in light of the continuity of $\lambda_t^H$ proved in Theorem~\ref{thm:limit}
is whether
 $\lambda_t^H$ is differentiable as well (as a function of $H$).  While on a global scale the graph of the escape rate function
 forms a devil's staircase \cite{dem wright}, the derivative of the escape rate may still
 exist as the hole shrinks to a point, as in \cite{buni, KL09}.  Although
 this result is likely to hold in the present setting, a sequence of holes requires a sequence
 of increasingly refined Markov partitions.  Thus proving such a result would require either adapting
 the approach of \cite{FerPol12} to the countable state setting, or constructing a uniform sequence of towers over
 a single base as in \cite{DemTodd}.  Since the present paper is already of
 considerable length, we do not include this result here.
 \end{remark}

Finally, we fix $H$ and address the regime $t \ge t^H$, where we obtain weaker results 
than Theorem~\ref{thm:rate gap}
due to the
absence of a spectral gap.

\begin{theorem}
\label{thm:limit point}
For each $t \in [0,1)$, all limit points of the sequence
$\left\{ \frac{\hLp_t^n 1 }{|\hLp_t^n 1|_{L^1(m_t)}} \right\}_{n \in \N}$ are 
absolutely continuous with respect to $m_t$ with log-Holder continuous densities on
elements of $\mathcal{P}_{N_0}$.

Moreover, setting $a_j = j^{t(1+\frac{1}{\gamma}) -1} \lambda_t^{-j}$ and 
$Z_n = \sum_{j=1}^{n} a_j |\hLp_t^j 1|_{L^1(m_t)}$, all limit points of the averages
$\frac{1}{Z_n} \sum_{i=1}^{n} a_i \hLp_t^i 1$ are absolutely continuous
conditionally invariant measures with eigenvalue $\lambda_t$, and the averages converge
in $L^1(m_t)$.
\end{theorem}

Although Theorem~\ref{thm:limit point} applies to all $t \in [0,1)$, in light of Theorem~\ref{thm:rate gap},
it only gives
new information for $t \in [t^H,1)$.
It may be of independent interest that the proof of absolute continuity for all limit points
of $\left\{ \frac{\hLp_t^n 1 }{|\hLp_t^n 1|_{L^1(m_t)}} \right\}_{n \in \N}$ holds 
independently of the proof of
Theorem~\ref{thm:rate gap}.  
This is in sharp contrast to the case $t=1$
for which $\frac{\hf^n_* m_1}{|\hf^n_* m_1|}$ converges weakly
to the point mass at the neutral fixed point \cite{DemFer13}.

Next we turn our attention to invariant measures on the survivor set.  By 
Theorem~\ref{thm:rate gap}, we have $\nu_t$, an equilibrium measure for $t\phi^H - p^H(t)$
for $t < t^H$ that is supported on $\hI^\infty$.  
Our next result shows that in fact, one can construct physically meaningful
invariant measures on $\hI^\infty$ even for $t \ge t^H$, including for $t = 1$, which are
not simply the point mass at the neutral fixed point and indeed contain no atoms.
For $t > t^H$, these measures do not maximize pressure on the survivor set (in this regime
the point mass at 0 does this), but they do converge to the equilibrium measure for
the unpunctured potential $t\phi - p(t)$ as the hole shrinks to a point.

In order to obtain sufficient expansion for our map,
we will consider the induced map $\hF$ and work with $\hF^2$ rather than $\hF$.  
The induced hole for $F$ is defined by 
$\tH = H$ if $H \subset [1/2,1]$ and $\tH = F^{-1}(f^\tau(H))$ if $H \subset [0,1/2]$,
where $\tau$ is the first hitting time to $Y = [1/2, 1]$. 
Since we will be working with $\hF^2$, the hole will effectively be $\tH \cup F^{-1}(\tH)$, which
always has countably many connected components in $Y$. 
Let $Y_n \subset Y$ be such that $f(Y_n) = J_n = [\ell_n, \ell_{n-1}) \in \mathcal{P}_1$, 
i.e. these are 1-cylinders for $F$ before
the introduction of the hole.
Set $Y_{i,j} = Y_i \cap F^{-1}(Y_j)$ and note that this is the maximal
partition on which the return time $\tau^2 = \tau + \tau \circ F$ 
is constant.  

Fix $z \in (0,1]$ and 
let $( H_\ve )_{\ve \in [0, \ve_0]}$ be a nested family of intervals 
(not necessarily elements of a Markov partition) containing
$z$ and such that $H_\ve$ has length $\ve$.  Let $\hF_\ve$ denote the map corresponding
to $H_\ve$ and let $\hY^n_\ve = \cap_{i=0}^n F^{-i}(Y \setminus \tH_\ve)$.
Let $\{ Z_{k, \ve}\}_{k \in \N}$ denote the countable collection of maximal intervals on which
$\hF_\ve^2$ is smooth.  Note that each interval $Z_{k,\ve}$ is contained in some
$Y_{i,j}$.
We shall need the following condition on the family $( H_\ve )_{\ve \le \ve_0}$.

\begin{enumerate}
\item[{\bf (H)}]  $\quad$ Let $( H_\ve )_{\ve \le \ve_0}$ be as above.  Assume
$\displaystyle \inf_{\ve \in [0, \ve_0]} \inf_{k} |\hF^2_\ve ( Z_{k, \ve}) | > 0$.
\end{enumerate}

\begin{remark}
Assumption {\bf (H)} is generically satisfied:
if $z$ is not an endpoint of one of the intervals $Y_n \subset (1/2,1]$ or one of the
intervals $J_n \subset (0,1/2]$, then {\bf (H)} is satisfied for $\ve_0$ sufficiently small.

If one is interested only in the case $t=1$, then one can work with $F$ rather than $F^2$
and condition {\bf (H)} can be stated in terms of $\hF_\ve(Y_n)$.  In that case,
only the points $z=0$ and $z=1/2$ would be excluded by {\bf (H)}.
\end{remark}

For $t \in [0,1]$, recalling the definition of induced potential in \eqref{eq:ind pot}, let $P_{t,\ve}  = P(t\Phi - \tau p^{H_\ve}(t))$.  In Lemma~\ref{lem:conformal}
we will prove that there exists a $(t\Phi - \tau p^H(t) - P_{t,\ve})$-conformal measure 
$\tm_{t,H}$ for $F$
on $Y$, which has no atoms.

\begin{theorem}
\label{thm:stable}
Let $( H_\ve )_{\ve \le \ve_0}$ be a nested family of intervals (not necessarily elements of a 
Markov partition) satisfying {\bf (H)}.

If $\ve_0$ is sufficiently small, for $t \in [0,1]$,  $\hF_\ve$ admits a physical  conditionally invariant
measure $\mu_{Y,\ve}$, absolutely continuous with respect to $\tm_{t,H}$, 
with eigenvalue $\Lambda_{t, \ve}$.  The limit
\begin{equation}
\label{eq:nu def}
\nu_{Y,\ve} (\psi) = \lim_{n \to \infty} \Lambda_{t, \ve}^{-n} \int_{\hY_\ve^n} \psi \, d\mu_{Y,\ve}
\end{equation}
exists for each $\psi \in C^0(Y)$ and defines an ergodic invariant 
probability measure $\nu_{Y, \ve}$
supported on the survivor set $\hY_\ve^\infty = \hI_\ve^\infty \cap Y$.  Moreover, $$\log \Lambda_{t,\ve} = P(t\Phi^{H_\ve} - \tau p^{H_\ve}(t)) - P(t\Phi - \tau p^{H_\ve}(t))$$ and 
$\nu_{Y,\ve}$ is an equilibrium measure for the potential $t\Phi^{H_\ve} - \tau p^{H_\ve}(t) -  P(t\Phi^{H_\ve} - \tau p^{H_\ve}(t))$.

The measure $\nu_{Y, \ve}$ projects to a probability measure $\nu_{H_\ve}$ with the following properties: 

\begin{enumerate}
\item $\nu_{H_\ve}$ is an invariant measure for $f$ supported on the survivor set $\hI_\ve^\infty$;
\item $\nu_{H_\ve}$ is an equilibrium state for $t\phi^{H_\ve} - p^{H_\ve}(t) - P(t\Phi^{H_\ve} - \tau p^{H_\ve}(t)) \cdot 1_Y$;
so if $t \le t^{H_\ve}$ then $\nu_{H_\ve}$ is an equilibrium state for $t\phi^{H_\ve} -p^{H_\eps}(t)$,
i.e.~it coincides with the measure $\nu_t$ from Theorem~\ref{thm:rate gap};
\item for fixed $t$, the free energy $h_{\nu_{H_\ve}}(f) + \int t \phi^{H_\ve} \, d\nu_{H_\ve}$ is continuous
in $\ve$ for $\ve\ge 0$ close enough to 0;
\item for fixed $\ve$, the free energy of $\nu_{H_\ve} = \nu_{H_\ve, t}$ is analytic for
$t \in (t^H,1)$ and continuous on the closure, $[t^{H_\ve}, 1]$;
\item $\nu_{H_\ve}$ converges weakly (when integrated against both continuous functions and
functions of bounded variation) to the equilibrium measure $\mu_t$ for the closed system
as $\ve \to 0$.
\end{enumerate}
\end{theorem}

\begin{remark} 
Note that $\nu_{H_\ve}$ is not a measure which maximizes the pressure
$h_\nu(f) + t \int \phi^{H_\ve} \, d\nu$ on 
$\hI_\ve^\infty$ when $t > t^{H_\ve}$.  One has 
\[
0 = h_{\delta_0}(f) + t \int \phi^{H_\ve} \, d\delta_0 ,
\]
where $\delta_0$ denotes the point mass at 0.
On the other hand, for the induced system, for any $t\le 1$,
 we have for the original map $f$,
\begin{equation}
\begin{split} h_{\nu_{H_\ve}}(f) + t\int \phi^{H_\ve} \, d\nu_{H_\ve} &= \frac{P(t\Phi^{H_\eps}-\tau p^{H_\eps}(t))}{\int_Y \tau \, d\nu_{Y, \ve}} + p^{H_\ve}(t) \\
&=  \left(P(t\Phi^{H_\eps}-\tau p^{H_\eps}(t))\right)\nu_{\ve}(Y) + p^{H_\ve}(t),
\end{split}\label{eq:VP nue}
\end{equation}
where the final equality follows from Kac's formula.
Using Proposition~\ref{prop:tH}, when $t > t^{H_\ve}$ then $p^{H_\ve}(t) = 0$ and $P(t\Phi^{H_\ve})<0$ so the right hand side of \eqref{eq:VP nue} is negative.  

Conversely, when $t< t^{H_\ve}$ then $p^{H_\ve}(t) > 0$ and $P(t\Phi^{H_\ve} - \tau p^{H_\ve}(t)) =0$, so the right hand side of  \eqref{eq:VP nue} is positive and 
$\nu_{H_\ve}$ is indeed an equilibrium state for $t\phi^{H_\ve} - p^{H_\eps}(t)$ 
as stated in Theorem~\ref{thm:stable}. 

When $t=t^{H_\ve}$, $p^{H_\ve}(t^{H_\ve}) = P(t^{H_\ve} \Phi^{H_\ve}) =0$ as in Proposition~\ref{prop:tH}(a), so the right hand side  of\eqref{eq:VP nue} is 0 and again  $\nu_{H_\ve}$ is an equilibrium state for 
$t\phi^{H_\ve} - p^{H_\eps}(t)$. 
\end{remark}

Note that $t=1$ is the only value of $t \in [0,1]$ which is not eventually less than $t^H$ as 
$H$ shrinks to a point.  Despite this, the sequence of measures $\nu_{H_\ve}$ constructed
in Theorem~\ref{thm:stable} converges to the SRB measure $\mu_1$ for the closed system
as $\ve \to 0$.  From the point of view of the invariant measure on $\hI^\infty$, then, this
theorem recovers a form of stability of the SRB measure for the system in the 
presence of small leaks.  This is in contrast to the instability of the SRB measure from the 
point of view of limiting distributions in the open system, since
$\hf^n_*m_1/|\hf^n_*m_1| \to \delta_0$ as $n \to \infty$ \cite{DemFer13}.


\section{Basic pressure results}
\label{sec:basic press}

In this section we will start by recalling thermodynamic formalism for symbolic systems, and then push this onto our system, proving Proposition~\ref{prop:tH}. 


\subsection{Thermodynamic formalism  in symbolic spaces}
\label{ssec:sym}

 Let $(\Sigma, \sigma)$  be a one-sided  Markov shift over the countable alphabet $\N$. This means
that there exists a matrix $(t_{ij})_{\N \times \N}$ of zeros and ones (with no row and no column made entirely of zeros) such that
\[\Sigma:= \left\{ (x_n)_{n \in \N} : t_{x_ix_{i+1}}=1 \text{ for every } i \in \N \right\}.  \]
The \emph{shift map} $\sigma: \Sigma \to \Sigma$ is defined by $\sigma(x_1x_2 x_2 \dots)=(x_2 x_2 \dots)$. We will always assume the system $(\Sigma, \sigma)$ to be topologically transitive
(but not necessarily mixing), which means that for any two elements $a, b\in \N$, there is a sequence $(x_n)_{n\in \N}\in \Sigma$ with $x_0=a$ and $x_n=b$ for some $n\in \N$.  Note that the theory usually assumes the stronger condition of topological mixing (see \cite{Sar99} for a precise definition), but in \cite{BuzSar03} this was shown to be unnecessary. The space $\Sigma$ endowed with the topology generated by the cylinder sets
\[C_{i_1 i_2 \dots i_n}:=\left\{ (x_n)_{n\in \N} \in \Sigma : x_i=i_j \text{ for } j \in \{1,2,3 \dots n\}\right\},\]
is a non-compact space. We define the \emph{$n^{\mbox{\scriptsize th}}$ variation} of a function $\phi:\Sigma \to \R$ by
\[var_n(\phi)= \sup_{(i_1 \dots i_n) \in \N^n} \sup_{x,y \in C_{i_1 i_2 \dots i_n}} |\phi(x)-\phi(y)| . \]
A function $\phi:\Sigma \to \R$ is locally H\"older if there exists $0< \theta <1$ and $C>0$ such that for every $n \in \N$ we have $var_n(\phi) \leq C \theta^n$. 

Given a potential $\phi:\Sigma\to \R$, let $S_n\phi(x):=\sum_{k=0}^{n-1}\phi(\sigma^k x)$, be the $n$-th ergodic sum.  A measure $\mu$ on $\Sigma$ is called a \emph{Gibbs measure} for $\phi$ if there exist $K\ge 1$ and $P\in \R$ such that for all $n$-cylinders $C_{i_1 i_2 \dots i_n}$ and all 
$x\in C_{i_1 i_2 \dots i_n}$, 
$$\frac1K\le \frac{\mu(C_{i_1 i_2 \dots i_n})}{e^{S_n\phi(x)-nP}}\le K.$$
Here $P$ is called the \emph{Gibbs constant} of $\mu$.

The \emph{Gurevich Pressure} of  a locally H\"older potential $\phi:\cup_nX_n\to \R$ was introduced by Sarig in \cite{Sar99}, generalizing Gurevich's definition of entropy.  It is defined by letting
\[
Z_n(\phi)=  \sum_{\sigma^n x=x} e^{S_n\phi(x)}  \mathbbm{1}_{X_{i}}(x),  
\]
 where $ \mathbbm{1}_{X_{i}}(x)$ denotes the characteristic function of the cylinder $X_i$, and the \emph{Gurevich pressure} is 
\[ 
P_G(\phi):= \lim_{n \rightarrow \infty} \frac{\log(Z_n(\phi))}{n},
\]
where the limit exists by almost superadditivity (\cite[Theorem 1]{Sar99}).
 The limit always exists and its value does not depend on the cylinder $X_{i}$ considered. This notion of pressure satisfies the following variational principle: if $\phi$ is a locally H\"older potential then 
 by \cite[Theorem 3]{Sar99},
\begin{equation*}
P_G(\phi)=P(\phi).
\end{equation*}
Hence we can write $P$ in place of $P_G$.
 A measure attaining the supremum above will be called an \emph{equilibrium measure} for $\phi$.  
  
 The potential $\phi$ is called \emph{recurrent} if
 $$\sum_nZ_n(\phi)e^{-nP_G(\phi)}=\infty,$$
 and otherwise it is called \emph{transient}.   Note that due to \cite[Theorem 1]{Sar01}, in this setting this definition of transience is equivalent to that given in the previous section. 
 Defining $Z_n^*(\phi)$ similarly to $Z_n(\phi)$, but only summing over those periodic points which make their first return to $X_i$ at time $n$, we say that a recurrent potential $
 \phi$ is \emph{positive recurrent} if 
  $$\sum_nnZ_n^*(\phi)e^{-nP_G(\phi)}<\infty,$$
and otherwise $\phi$ is \emph{null recurrent}.  Again this definition is independent of $X_i$, and indeed a $k$-cylinder yields the same result.   

It is easy to see from the definition of $P_G$ that the pressure function is convex, when finite (one can also easily prove this from the basic definition of pressure $P$).  Hence we have the following lemma.

\begin{lemma}
Suppose that $\phi:\Sigma\to \R$ is locally H\"older and $t_1<t_2$ are such that $P_G(t\phi)<\infty$ for $t\in (t_1, t_2)$. Then $t\mapsto P_G(t\phi)$ is continuous on $(t_1, t_2)$.
\label{lem:pres cts}
\end{lemma}
 
We say that $(\Sigma, \sigma)$ has the \emph{big images and preimages (BIP) property} if
$$ \exists b_1, \ldots, b_N\in \N \text{ such that } \forall\ a\in \N, \ \exists i,j \text{ such that } t_{b_ia}t_{ab_j}=1.$$
A simple example of such a system is the full shift on $\N$. As in \cite{Sar03}, we can set
 \[\tilde Z_n(\phi)= \sum_{\sigma^n x=x} e^{S_n\phi(x)},  \]
i.e., we needn't restrict ourselves to $X_i$, and it can be shown that 
\[ 
P(\phi)= \lim_{n \rightarrow \infty} \frac{\log(\tilde Z_n(\phi))}{n}.
\]
Moreover, $P(\phi)<\infty$ if and only if $\tilde Z_1(\phi))<\infty$. 

\begin{theorem}[\cite{Sar03}] 
\label{thm:Gibbs}
$P_G(\phi)<\infty$ if and only if there is an invariant Gibbs measure $\mu$ for $\phi$ with Gibbs constant $P_G(\phi)$.  Moreover, if $h(\mu)<\infty$ (equivalently $\int\phi>-\infty$) then $\mu$ is an equilibrium state for $\phi$.
\end{theorem}

Let $C_n$ be a cylinder and $\bar\sigma:C_n\to C_n$ the first return map to $C_n$ with return time $r_{C_n}$.  $(C_n, \bar\sigma)$ is known as the \emph{induced} system on $C_n$.  Given a potential $\psi:\Sigma \to \R$, let $\Psi:C_n\to C_n$ be defined by 
$\Psi(x) := S_{r_{C_n}} \psi(x) = \sum_{k=0}^{r_{C_n}-1}\psi(\sigma^k x)$.

Given a $\sigma$-invariant measure $\mu$, giving positive mass to $C_n$, we call $\bar\mu=\frac{\mu|_{C_n}}{\mu(C_n)}$ the \emph{lift} of $\mu$.  By Kac's Lemma, this is $\bar\sigma$-invariant.  Conversely, given a 
$\bar\sigma$-invariant measure $\bar\nu$, if $\nu$ lifts to $\bar\nu$, then $\nu$ is called the \emph{projection} of $\bar\nu$.
Abramov's formula gives
$$h(\bar\nu)=\left(\int r_{C_n}~d\bar\nu\right)h(\nu) \text{ and } \int\Psi~d\bar\nu = \left(\int r_{C_n}~d\bar\nu\right)\left(\int\psi~d\nu\right).$$
We note that these results also pass to induced maps which are not first return maps.

Given a metric on $\Sigma$, a potential $\psi:\Sigma\to \R$ is called a \emph{metric potential} if there exists $K\ge 1$ such that
$$\frac1K \prod_{j=0}^{n-1}\frac1{\psi(\sigma^jx)}\le \diam( [i_1,\ldots, i_n])\le K \prod_{j=0}^{n-1}\frac1{\psi(\sigma^jx)}.$$

The following is \cite[Theorem 3.1]{Iom05}, adapted slightly.  The proof uses inducing to some domain to produce a BIP system, so we change the statement to include this explicitly rather than talking about recurrent points as in \cite{Iom05}.

\begin{theorem}
If $(\Sigma, \sigma)$ has the BIP property and is topologically transitive and $\psi$ is a metric potential, then $$\dim_H(\Sigma)=t^*:=\inf\{t: P_G(-t\log\psi)\le 0\}.$$
\label{thm:Iommi}
\end{theorem}


\subsection{Preliminary results on pressure for our systems}

For $A$ contained in some interval, we say that the map $T:A\to A$ is \emph{Markov} if there 
exists a countable Markov shift $(\Sigma, \sigma)$ and a continuous bijective map $\pi :\Sigma \to A$  
such that  $T \circ \pi = \pi \circ \sigma$. We will use the notation 
$[i_1,\ldots, i_n]:= \pi(C_{i_1 \dots i_n})$. We will also lift potentials 
$\phi:A\to \R$ to their symbolic version $\phi\circ\pi$ which we will require to be locally H\"older.

Now we return to our open system $(\hf, \hI, H)$.  Recall that
$X = \cup_{Q \in \mathcal{Q}} Q$ and
let $\hX^\infty$ denote the set of points which map infinitely often into $X$ under $\hF$.
In the following, we will use the fact that the natural symbolic coding of the system $\hF:\hX^\infty \to \hX^\infty$ satisfies the BIP property and is transitive, although it may be not mixing.
This means that all the results in Section~\ref{ssec:sym} pass to our system: here our potential $\Phi$ lifts to a metric potential on the symbolic model (that is, compatible with the Euclidean metric on $[0, 1)$), thus also inducing a compatible metric.  
We also note that any ergodic measure on $\hI^\infty$ with positive entropy must give positive measure to $X$, which, since $\hF$ is a first return map, by Kac's Lemma means that it must lift to the induced system $(X, \hF)$.


The next lemma follows immediately from the structure of our system.

\begin{lemma}
Suppose that $H$ is a non-swallowing Markov hole.   Then 
$$\I^{\infty} \sm\left(\cup_{k\ge 0} f^{-k}(\hX^{\infty})\right)$$
consists of at most the countable set of preimages of 0.  
\label{lem:Y inf}
\end{lemma}

We close this section with the proof of our first main result, Proposition~\ref{prop:tH}.

\begin{proof}[Proof of Proposition~\ref{prop:tH}]
The fact that $p^H(t^H)=0$ is part of Lemma~\ref{lem:Hdim}.  
The following claim then completes the proof of (a).

\begin{claim}
$P(t^H\Phi^H)=0$.
\end{claim}

\begin{proof}
The fact that $P(t^H\Phi^H)\le 0$ is Lemma~\ref{lem:IT ind neg}.  Moreover, by definition of $t^H$, if $t<t^H$ there must be a measure $\mu$ on $\hI^\infty$ of positive entropy with $h(\mu)+\int t\phi~d\mu>0$.  Since any measure on $\hI^\infty$  of positive entropy lifts to our inducing scheme, Abramov's formula implies that  
$P(t\Phi^H)>0$.  So by the continuity of $t\mapsto P(t\Phi^H)$, when finite (Lemma~\ref{lem:pres cts}), to complete the proof of the claim, we need to show that there is $t<t^H$ such that $P(t\Phi^H)<\infty$.

 As described in Section~\ref{ssec:sym}, $P(t^H\Phi^H)\le 0$ implies that $\tilde Z_1(t^H\Phi)<\infty$, which means that $\sum_i|X_i|^{t^H}<\infty$.  
Since  by Lemma~\ref{lem:basic tail}, the diameter of each of the domains in this sum is polynomially small in the inducing time, and the number of domains with the same inducing 
time is uniformly bounded, this also implies that there exists $t< t^H$ such that $\tilde Z_1(t\Phi)<\infty$.  Hence $P(t\Phi^H)<\infty$, as required. \end{proof}

By Theorem~\ref{thm:Iommi},  $\dim_H(\hX^\infty) =t^H$.  By Lemma~\ref{lem:Y inf}, $\dim_H(\I^\infty)=\dim_H(\hX^\infty) =t^H$, proving (b).

For (c), by Lemma~\ref{lem:basic tail}, we have $\sum_i|X_i|^t<\infty$ if and only 
if $t>\frac\gamma{1+\gamma}$, since the number of domains with the same inducing time
is uniformly bounded by $C_{N_0}$.  As in the claim,  $\sum_i|X_i|^{t^H}<\infty$ so (c) follows immediately.

We next prove (d).  Let $\delta_0$ denote the Dirac mass at 0.
Since $H$ is a non-swallowing hole and $h(\delta_0)=\int \phi^H~d\delta_0=0$, the variational definition of pressure implies $p^H(t)\ge 0$ for all $t\in \R$.  The claim implies $p^H(t^H)=0$, so since $p^H$ is decreasing, $p^H(t)=0$ for all  $t>t^H$.  Hence $P(t\Phi^H-\tau p^H(t))=P(t\Phi^H)$ for $t>t^H$, so the final part of (d) follows since $t\mapsto P(t\Phi^H)$ is strictly decreasing.

For (e), by definition $p^H(t)> 0$.  Again as in  Lemma~\ref{lem:IT ind neg}, $P(t\Phi^H-p^H(t)\tau)\le0$.  By the continuity of pressure, in domains where it is finite (Lemma~\ref{lem:pres cts}) we only need show that for any small $\delta>0$, $0<P(t\Phi^H-(p^H(t)-\delta)\tau)<\infty$.  However, from the variational definition of pressure, there must exist an ergodic invariant measure with positive entropy $\mu$ such that $$h(\mu)+\int t\phi^H~d\mu>p^H(t)-\delta.$$
This measure must lift to a measure $\mu_F$ on $(Y, F)$. 
By the Abramov formula,
$$h(\mu_F)+\int t\Phi^H-(p^H(t)-\delta)\tau~d\mu_F=\left(\int\tau~d\mu_F\right)\left(h(\mu)+\int t\phi^H-(p^H(t)-\delta)~d\mu\right)>0.$$
Hence the variational principle implies $P(t\Phi^H-(p^H(t)-\delta)\tau)>0$.  The fact that this pressure is also finite when $\delta>0$ is small follows from the fact that if $\delta<p^H(t)$, then clearly $\tilde Z_1(t\Phi^H-(p^H(t)-\delta)\tau)<\infty$ since $|\{\tau=n\}|$ is subexponential 
and the number of domains with $\tau = n$ is uniformly bounded by Lemma~\ref{lem:basic tail}. 

(f) is a standard consequence of the null-recurrence of $t^H\phi^H$, see for example \cite[Section 8--9]{IomTod10}.
\end{proof}


\section{Proof of Proposition~\ref{prop:volume}}
\label{sec:volume}

The proof relies on a volume lemma argument (c.f.
\cite{young trans, DWY1}) applied to the conformal measures $m_t$.
However, in order to obtain the volume estimates we need, we shall rely
on the following cylinder structure, which is coarser than $\P_n$. 

Let $\D:=\{[0, 1/2), [1/2, 1)\}$ and let $\D_n:=\bigvee_{k=0}^{n-1}f^{-k}\D$.  Now for $x\in [0, 1)$, let $D_n(x)$ denote the element of $\D_n$ containing $x$.  The next lemma follows from `tempered distortion', see \cite{JorRam11} for a proof.

\begin{lemma}
There exists a sequence $(V_n)_n\subset(0, \infty)$ where $V_n\to 0$ as $n\to \infty$ such that for any $D\in \D_n$ and any $x, y\in D$, $|\log Df^n(x)-\log Df^n(y)|\le nV_n$.
\label{lem:temp dis}
\end{lemma}

Using this lemma in conjunction with the Mean Value Theorem, we will estimate the $m_t$-measure of elements of $\D_n$.

Recalling that $H$ is a union of elements of $\P_{N_0}$, let $\hat N_0$ be such that $H$ is a union of elements of $\D_{\hat N_0}$.

Suppose $\nu$ is an ergodic invariant probability measure for $\hf$ supported on $\hI^\infty$
such that $\nu(\partial \D_1) = 0$. This assumption excludes $\nu = \delta_0$ and $\nu =\delta_1$, 
the point masses at 0 and 1, respectively.
Since $\D_1$ is a generating partition for $f$,
the Shannon-McMillan-Breiman Theorem yields,
\[
\lim_{n \to \infty} - \frac 1n \log \nu(D_n(x)) = h_{\nu}(f)
\]
for $\nu$-a.e. $x$.

By Lemma~\ref{lem:temp dis}, conformality and the fact that 
$f^n(D_n(x))=[0,1)$, we have
\begin{equation}
\label{eq:expansion}
m_t(D_n(x)) \ge m_t(f^n(D_n(x))) e^{tS_n \phi(x) - n (p(t)+V_n)}=e^{tS_n \phi(x) - n (p(t)+V_n)}.
\end{equation}

Now for $\ve > 0$, define 
\[
G_{\ve,n} := \left\{ x \in \hI^\infty : \frac 1n S_n t\phi(x) \ge t\int \phi ~d\nu - \ve 
\mbox{ and } \nu(D_n(x)) \le e^{-n (h_\nu - \ve)} \right\} .
\]
By the ergodic theorem and the Shannon-McMillan-Breiman theorem, for $\sigma>0$ we may
choose $n$ so large that $\nu(G_{\ve,n} ) > 1 - \sigma$ for all $n$ sufficiently large.

By choice of $H$, for $n > \hat N_0$ and $D \in \D_n$, if $\hI^\infty \cap D \neq 0$, then
$D \subset \hI^{n-\hat N_0-1}$.  Let 
$\mathcal{K}_n = \{ D \in \D_n : D \cap G_{\ve,n}  \neq \emptyset \}$.
Note that by construction we must have $\nu(\cup_{D \in \mathcal{K}_n} D) \ge 1-\sigma$
for sufficiently large $n$.  Thus by definition of $G_{\ve, n}$, the cardinality of
$\mathcal{K}_n$ must be at least $(1-\sigma)  e^{n(h_\nu - \ve)}$.
These considerations together with \eqref{eq:expansion} yield,
\[
m_t(\hI^{n-\hat N_0-1}) 
\ge \sum_{D \in \mathcal{K}_n } m_t(D)
  \ge \sum_{D \in \mathcal{K}_n }  e^{n (\int t \phi d\nu - p(t) - \ve)}  
\ge (1-\sigma) e^{n (h_\nu + t \int \phi d\nu - p(t) - 2 \ve)} .
\]
Now taking logs and dividing by $n$ yields,
$\log \underline{\lambda}_t \ge h_\nu + t \int \phi d\nu - p(t) - 2\ve$, and since this is true
for each $\ve>0$, we conclude $\log \underline{\lambda}_t \ge h_\nu + t\int \phi d\nu - p(t)$. 

We treat the case $\nu = \delta_0$ separately.  In this case $P_{\delta_0}(t\phi^H) =0$,
so the required inequality will hold if $\log \underline{\lambda}_t \ge - p(t)$.  This is
immediate since $\hI^n \supset [0,\ell_{n+k}) = \cup_{i \ge n+k} J_i$ for some $k > 0$ and all
$n$ sufficiently large.  Thus
\[
m_t(\hI^n) \ge m_t([0,\ell_{n+k})) \ge C \sum_{i \ge n+k} |J_i|^t e^{-ip(t)}
\ge C' (n+k)^{-t(1 + \frac{1}{\gamma})} e^{-(n+k) p(t)} , 
\]
and $\log \underline{\lambda}_t \ge - p(t)$ follows.

The only other ergodic invariant measure which gives positive mass to $\partial \D_1$ is
$\delta_1$, the point mass at 1.  Clearly $P_{\delta_1}(t\phi^H) = -t\log 2$ and 
so by our previous work,
\[
\log \underline{\lambda}_t \ge - p(t) \ge -t\log 2 - p(t) = P_{\delta_1}(t\phi^H) - p(t).
\]

We have shown that
\[
\log \underline{\lambda}_t \ge \sup \Big\{ h_\nu(f) + t \int \phi^H d\nu : \nu \mbox{ is } f\text{-invariant and ergodic and } \nu(H)=0 \Big\} - p(t),
\]
which is precisely what is required for the proposition.


\section{Proofs of Theorems~\ref{thm:rate gap} and \ref{thm:limit}
and Corollary~\ref{cor:variational}}
\label{sec:rate gap}

In this section we prove results in the uniformly hyperbolic regime:
for $t \in [0, t^H)$, the exponential tail decays faster than the rate of escape.

We will prove Theorem~\ref{thm:rate gap} by using the induced map $F:Y \circlearrowleft$,
and an associated object known as a Young tower.  We begin by recalling some basics
about Young towers.


\subsection{Defining the Young Tower}
\label{sec:tower def}

Given the inducing scheme with a hole, $(Y, F, \tau, H)$, we define the
corresponding Young tower as follows.  
Recall the finite partition $\Q$ of images in $Y$ on which $\hF$ is transitive on elements defined
in Section~\ref{intro hole}.  Define $\Delta_0 = X$
and denote 
by $\Delta_{0,i}$, the finitely many elements of $\Q$ comprising $X$.
Let 
\[
\Delta = \{ (x, n) \in \Delta_0 \times \mathbb{N} \mid n < \tau(x) \} .
\]
$\Delta$ is viewed schematically as a tower with $\Delta_\ell = \Delta|_{n = \ell}$ as the
$\ell$th level of the tower.  The tower map, $f_\Delta$, is defined by $f_\Delta(x, \ell) = (x, \ell + 1)$
if $\ell + 1 < \tau(x)$ and either $f^\tau(x) \in H$, in which case we 
define a hole $H_{\tau,j} \subset \Delta_\tau$, or $f^\tau(x) \in \Delta_0$, in which case
$f_\Delta(x, \tau(x) -1) = (f^\tau(x), 0) = (F(x), 0)$.
There is a canonical projection $\pi : \Delta \to I$ satisfying $\pi \circ f_\Delta = f \circ \pi$.
$\Delta_0$ is identified  with $\cup_{Q \in \Q'} Q$ so that $\pi|_{\Delta_0} = Id$.  

Let $\{ X_i \}$ denote the maximal partition of $\Delta_0 = X$ 
into intervals such that $\tau_Y$
is constant on each $X_i$.  The partition $\{ X_i \}$ then
induces a countable Markov partition $\{ \Delta_{\ell,j} \}$ on $\Delta$ via the identification
$\Delta_{\ell,j} = f_{\Delta}^\ell(X_j)$, for $0 < \ell < \tau(X_j)$.  On level $\ell=0$, we insist
on keeping the partition finite and use $\Delta_{0,i}$ as our partition elements.
Since $f$ is expanding, the partition $\{ X_i \}$, and hence the partition 
$\{ \Delta_{\ell,j} \}$, is generating.

Note that since $H$ is a 1-cylinder in $\P_{N_0}$, by definition of $\{X_i \}$, 
$\tH := \pi^{-1} H$ is the union of countably many
partition elements $\Delta_{\ell,j}$.  We set $\hDelta = \Delta \setminus \tH$ and refer to
the corresponding partition elements as $\hDelta_{\ell,j}$.
Similarly, we define $\hDelta^n = \bigcap_{i=0}^n f_\Delta^{-i} \hDelta$
and $\hf_\Delta^n = f_\Delta^n|_{\hDelta^n}$,
$n \in \N$.

Given a potential $\vf$ and a $\vf$-conformal reference measure $m$ on $I$,
we define a reference measure $\bm$ on $\Delta$ by $\bm = m$ on $\Delta_0$ and
$\bm|_{\Delta_\ell} = (f_\Delta)_* \bm|_{\Delta_{\ell-1} \cap f_\Delta^{-1} \Delta_\ell}$ 
for $\ell \ge 1$.  
For $x \in \Delta_\ell$, let $x^- := f^{-\ell}x$ 
denote the pullback of $x$ to $\Delta_0$.
We define the induced
potential on $\Delta$ by
\begin{equation}
\label{eq:Delta potential}
\vf_\Delta(x) = S_\tau \vf (x^-) \; \; \mbox{for } x \in f_\Delta^{-1}(\Delta_0)
\; \; \mbox{and} \; \; \vf_\Delta = 0 \; \; \mbox{on } \Delta \setminus f_\Delta^{-1}(\Delta_0) .
\end{equation}
With these definitions, the measure $\bm$ is $\vf_\Delta$-conformal.

\begin{lemma}
\label{lem:exponential tail}
For $t \in [0,1)$, let $\bm_t$ be the measure on $\hDelta$ induced by $m_t$,
the conformal measure for the potential $t\phi - p(t)$.  There exists $C>0$, independent of $t$, such that
for $n \ge 0$,
$\bm_t(\Delta_n) \le C \frac{1}{p(t)} n^{-t(1+ \frac{1}{\gamma})}e^{-np(t)}$.
\end{lemma}

\begin{proof}
This follows immediately from Lemma~\ref{lem:basic tail} and the definition of $\bm_t$
since 
$m_t(\tau = n) \approx n^{-t(1 + \frac{1}{\gamma})} e^{-np(t)}$ due to conformality and the growth
in $Df^n$ given by (D1) of Section~\ref{apply tower}.  
\end{proof}

We define the transfer operator 
$\hLp_{\vf^H_\Delta}$ associated with the punctured potential $\vf_\Delta^H$ 
and acting on $L^1(\bm)$ by
\[
\hLp^n_{\vf^H_\Delta} \psi(x)  = \sum_{f_\Delta^n y = x} \psi(x) e^{S_n\vf_\Delta(y)} 1_{\hDelta^n}(y) = \Lp^n_{\vf_\Delta}(\psi 1_{\hDelta^n})(x).
\]
With these definitions, $\hLp$ satisfies the following change of variable formula,
\[
\int_{\hDelta} \hLp^n_{\vf^H_\Delta} \psi \, d\bm = \int_{\hDelta^n} \psi \, d\bm ,
\]
which in turn links the spectral properties of $\hLp_{\vf^H_\Delta}$ to the escape rate
with respect to $\bm$.  

In order to translate between densities on $I$ and on $\Delta$, for
$\tpsi \in L^1(\bm)$
on $\Delta$, define the projection
\begin{equation}
\label{eq:P pi}
P_{\pi,m} \tpsi(x) = \sum_{y \in \pi^{-1}(x)} \frac{\tpsi(y)}{J_m \pi(y)},
\end{equation}
where $J_m \pi$ is the Jacobian of $\pi$ with respect to the measures $m$ and $\bm$.
Then $P_{\pi,m} \tpsi \in L^1(m)$ and 
\[
\hLp_{\vf^H}^n(P_{\pi,m} \tilde\psi)=P_{\pi,m}(\hLp_{\vf^H_\Delta}^n\tilde\psi) .
\]

Indeed, the following lemma shows that the escape rates from
$I$ and from $\Delta$ are the same.

\begin{lemma}
Let $\bm_t$ be the measure on $\Delta$ induced by $m_t$ as in Lemma~\ref{lem:exponential tail}. 
Then for each $t \in [0,1]$,
\[
\log \overline \lambda_t = \limsup_n \frac1n\log m_t(\I^n)=\limsup_n\frac1n\log\bm_t(\hDelta^n).
\]
\label{lem:base tower decay}
\end{lemma}

\begin{proof}
When $t=1$, $p(t)=0$ and the above quantities are all 0, so equality is trivial.
Now assume $t<1$.

Recall the finite partition $\Q$ of $[1/2, 1] \setminus H$ defined in Section~\ref{intro hole} 
on which $\hF$ is transitive
on elements.  It follows from the definition of non-swallowing that $\hI \subset \cup_{i \ge 0} f^i(\cup_{Q \in \Q} Q) 
= \pi(\hDelta)$ up to a countable collection of points comprising the pre-images of 0.  Thus
\begin{equation}
\label{eq:tower project}
\log \overline{\lambda}_t = \limsup_n \frac 1n \log m_t(\hI^n) 
= \limsup_n \frac 1n \log m_t(\hI^n \cap \pi(\hDelta))  .
\end{equation}

Due to the transitivity and
finitely many elements of $\Q$, we may choose a collection of indices, 
$\mathcal{K} := \{ (\ell,j) \}$  with $\ell \le L$ for some $L>0$ such that 
$\pi(\cup_{(\ell,j) \in \mathcal{K}} \dlj) = \pi(\hDelta)$ and $\pi(\dlj) \cap \pi(\Delta_{\ell',j'}) = \emptyset$
for all pairs $(\ell,j) \neq (\ell',j')$ in $\mathcal{K}$.
Further, we may choose $\mathcal{K}$ so that all elements of the base, 
$\Delta_{0,i}$, belong to $\mathcal{K}$.

Now let $P_{\pi,t}$ be the projection defined by \eqref{eq:P pi} with respect to the measures
$m_t$ and $\bm_t$.  Denote by $J_t\pi$ the relevant Jacobian.
On each $\dlj$ for $(\ell,j) \in \mathcal{K}$, define
$\tpsi(x) = J_t\pi(x)$.  Set $\tpsi \equiv 0$ elsewhere on $\hDelta$.  
Then by construction of $\mathcal{K}$,
$P_{\pi,t} \tpsi = 1_{\pi(\hDelta)}$.  

Note that by conformality, for $x \in \Delta_\ell$, $J_t \pi(x) = e^{-tS_\ell \phi(x) + \ell p(t)}$. 
Thus there exists $M>0$, depending on $L$, such that $1 \le \tpsi \le M$ on $\Delta$.
Also, $J_t\pi =1$ on $\Delta_0$ so that 
$\tpsi =1$ on $\hDelta_0$.  

Now integrating,
\[
\bm_t(\hDelta^n)  =\int_{\hDelta^n}1~d\bm_t \ge \frac1M \int_{\hDelta^n}\tilde\psi~d\bm_t 
 =\frac1M \int_{\pi(\hDelta^n)}P_{\pi,t} \tpsi~dm_t = \frac1M m_t(\I^n \cap \pi(\hDelta)),
\] 
so we deduce $\limsup_n \frac1n\log m_t(\I^n)\le \limsup_n\frac1n\log\bm(\hDelta^n)$,
using \eqref{eq:tower project}.

On the other hand, notice that since $\hI^n \supset [0,\ell_{n+h})$, where $h$ denotes the
maximal index such that $J_h \cap H \neq \emptyset$, we have 
$\overline{\lambda}_t \ge e^{-p(t)}$.  Thus by Lemma~\ref{lem:exponential tail},
\[
\begin{split}
\bm_t(\hDelta^n) & = \int_{\hDelta^n \cap (\tau \le n)} d\bm_t + \int_{\hDelta^n \cap (\tau>n)} d\bm_t
\le n \int_{\hDelta^n \cap \Delta_0}  d\bm_t + Ce^{-p(t) n} \\
& \le n  \int_{\hDelta^n} \tilde \psi \, d\bm_t + Ce^{-p(t) n}
\le n \, m_t(\I^n) + Ce^{-p(t) n} .
\end{split}
\]
Since as noted above, $\overline \lambda_t \ge e^{-p(t)}$, we conclude that
$\limsup_n \frac 1n \log \bm_t(\hDelta^n) \le \limsup_n \frac 1n \log m_t(\I^n)$.
\end{proof}


\subsection{Abstract Results for Young Towers with Holes}
\label{abstract tower}

In this section, we prove results about an abstract Young tower with holes which may be
of independent interest:  We prove the transfer operator on the tower has a spectral gap
in a space of H\"older continuous functions under the assumption that the
escape rate is slower than the decay rate in the levels of the tower.  After stating
our assumptions formally below, we will describe how the present results generalize the
existing results of \cite{BruDemMel10} and related references.

We assume that we have a tower map $f_\Delta : \Delta \circlearrowleft$ with a countable
generating Markov partition $\{ \dlj \}$ and such that the return map to the base has finitely
many images, denoted by $\Delta_{0,i}$.
We further assume that our reference measure $\bm$ is conformal with respect to a
potential $\vf_\Delta$.

In this context, we assume the following properties of the tower map.

\begin{itemize}
  \item[(P1)]  (Exponential tail.)  There exist constants $C, \alpha >0$
such that  $\bm(\Delta_n) \le Ce^{-\alpha n}$, for $n \in \N$. 
\end{itemize}

We define a natural metric adapted to the dynamics as follows. 
Let $\tau^n(x)$ be the time of the $n$th return of $x$ to $\Delta_0$. 
Define the \emph{separation time} on $\Delta$ to be
\[
\begin{split}
s(x,y) = \min \{ n \geq 0 : \; & f_\Delta^{\tau^n}(x), f_\Delta^{\tau^n}(y) \mbox{ lie in different partition 
elements $\Delta_{0,i}$} \}. \\
\end{split}
\]
$s(x,y)$ is finite $\bm$-almost everywhere 
since $\{ \dlj \}$ is a generating partition for $f^\tau$.
Choose\footnote{In this abstract setting, the choice of
$\delta >0$ is constrained only by (P3); however, in applications, $\delta$ will be constrained
by the expansion and regularity of the underlying map $f$.  We will introduce this restriction on
$\delta$ when we apply this abstract framework to our map with neutral fixed point 
in Section~\ref{apply tower}.}
 $\delta >0$ and define a metric $d_\delta$ on $\Delta$ by
$d_\delta(x,y) = e^{-\delta s(x,y)}$.

We introduce a hole $H$ in $\Delta$ which is the union of countably many
partition elements $\dlj$, i.e. $H = \cup_{\ell,k} H_{\ell,k}$ where $H_{\ell,k} = \dlj$
for some $j$.  Set $H_\ell = \cup_j H_{\ell,j} \subset \Delta_\ell$.
For simplicity we assume that the base $\Delta_0$ contains no holes (this can always
be arranged in the construction of the tower by choosing a suitable reference set $X$).
We assume that $\hf_\Delta$ is transitive and aperiodic on the elements $\{ \Delta_{0,i} \}$
after the introduction of the hole.

\begin{itemize}
  \item[(P2)]  (Slow escape.) Define $\log \blambda = \limsup_{n \to \infty} \frac 1n \log \bm(\hDelta^n)$.
  We assume that $- \log \lambda < \alpha$.
  \item[(P3)]  (Bounded distortion.) We suppose that $e^{\vf_\Delta}$ is 
  Lipschitz in the metric $d_\delta$.
    Furthermore, we assume there exits $C_d>0$ such that for all $x, y \in \Delta$ and $n \geq 0$,
\begin{equation}
\label{eq:bounded dist delta}
\Big| e^{S_n\vf_\Delta(x) - S_n\vf_\Delta(y)} -1 \Big| \le C_d d_\delta(f_\Delta^nx, f_\Delta^ny) .
\end{equation}
  \item[(P4)] (Subexponential growth of potential)
 For each $\ve >0$, there exists $C>0$, such that 
 \begin{equation}
\label{eq:finite pressure}
 |S_\tau\vf_\Delta(x) | \le  C e^{\ve \tau(x)} \; \;\; \mbox{at first return times $\tau$ for all 
 $x \in \Delta_0$.}
\end{equation} 
\end{itemize}

A spectral gap for tower maps with holes was established in \cite{BruDemMel10} (see also
\cite{demers tower, demers logistic, DWY1, DWY2} for applications) under stronger conditions
than those listed here.  

\begin{remark}
\label{rem:BDM}
There are two significant differences between our assumptions
and those in \cite{BruDemMel10}.  (1) The metric $d_\delta$ in \cite{BruDemMel10} uses
a stronger notion of separation time which requires the derivative of the underlying map $f$
to grow exponentially at return times; since our maps have only polynomial growth in the
derivative, we adopt a weaker metric which requires significant changes 
to our function space arguments; in
particular, see the proofs of Lemmas~\ref{lem:lasota yorke} and \ref{lem:log contract}.
 (2)  The assumption in \cite{BruDemMel10} on the size of the hole is comparatively strict,
allowing one to control the maximum amount of mass lost in a single iterate of the map; by contrast,
our assumption (P2) only assumes that the tail decay is faster than the escape rate
asymptotically, which again requires significant revisions to the proof of the spectral gap.
\end{remark}
 
Choose $\beta$ satisfying 
$-\log \overline \lambda <\beta< \alpha$.  
We define the standard weighted $L^\infty$-norm on the space of functions on $\Delta$, 
given by 
$$\|\psi\|_\infty=\sup_\ell \sup\{e^{-\beta\ell}|\psi(x)|:x\in \Delta_\ell\},$$
along with a Lipschitz norm 
$$|\psi|_{Lip}=\sup_{\ell,j} e^{-\beta \ell} \sup \{e^{-\delta s(x, y)}|\psi(x)-\psi(y)|:x, y\in \Delta_{\ell,j} \} .$$  
We define a Banach space $(\B, \|\cdot\|_\B)$ where $\|\psi\|_\B=|\psi|_{Lip}+\|\psi\|_\infty$.

We proceed to prove the quasi-compactness for the punctured transfer operator
$\hLp_\Delta := \hLp_{\vf^H_\Delta}$ (i.e., with punctured potential $\vf^H_\Delta$) 
acting on $\B$ under assumptions (P1)--(P4).  

To do this, let $1_\beta$ denote the function which
takes constant value $e^{\beta \ell}$ on $\Delta_\ell$.  Note $1_\beta \in \B$
since $\| 1_\beta \|_\B =1$.  Define
\[
\overline \lambda_\beta = \limsup_n \frac 1n \log \int_{\hDelta^n} 1_\beta \, d\bm,
\] 
and note that $\overline \lambda_\beta \ge \overline \lambda > e^{-\alpha}$ by (P2).

\begin{lemma}
\label{lem:spectral radius}
The spectral radius of $\hLp_\Delta$ on $\B$ is at least $\overline\lambda_\beta$.
\end{lemma}
\begin{proof}
Note 
\[
\int_{\hDelta^n} 1_\beta \, d\bm = \int \hLp_\Delta^n(1_\beta) \, d\bm
\le \sum_\ell e^{\beta \ell} \bm(\hDelta_\ell) \| \hLp^n_\Delta(1_\beta) \|_\infty 
 \le \sum_\ell C e^{(\beta - \alpha)\ell} \| \hLp^n_\Delta \| \| 1_\beta \|_\B
\le C' \| \hLp^n_\Delta \| .
\]
Thus $\limsup_n \frac 1n \log \| \hLp^n_\Delta \| \ge \limsup_n \frac 1n \log \int_{\hDelta^n} 1_\beta \, d\bm = \log \overline \lambda_\beta$ as required.
\end{proof}

Next we prove that the essential spectral radius is strictly smaller than $\overline\lambda_\beta$.
This bound, together with the preceding lemma, will imply that $\hLp_\Delta$ is quasi-compact
as an operator on $\B$.

\begin{lemma}[Lasota-Yorke Inequality]
\label{lem:lasota yorke}
There exists $\sigma < \overline \lambda_\beta$ and $C>0$ such that for all $\psi \in \B$ and $n \ge 0$,
\[
\|\hLp_\Delta^n\psi\|_\B\le C\sigma^n \|\psi\|_\B+C |\psi|_{L^1(\bm)}.
\]
\end{lemma}
\begin{proof}
Fix $\psi \in \B$ and $n \ge 0$.  

\medskip
\noindent
Step 1.  For $\ell \ge n$ and $x \in \Delta_\ell$, we estimate
\[
|\hLp^n_\Delta \psi(x) | =| \psi(f_\Delta^{-n}x) | \le e^{\beta (\ell - n)} \| \psi \|_\infty .
\]
Thus $\| \hLp_\Delta^n \psi|_{\Delta_\ell} \|_\infty \le e^{-\beta n} \| \psi \|_\infty$.
Similarly, for $x, y \in \Delta_{\ell, j}$, 
\[
e^{-\delta s(x,y)} |\hLp^n_\Delta \psi(x) - \hLp^n_\Delta \psi(y)|
= e^{-\delta s(f_{\Delta}^{-n}(x), f_\Delta^{-n}(y))}|\psi(f_\Delta^{-n}x) - \psi(f_\Delta^{-n}y)| ,
\]
since the separation time for $x,y$ is the same as that for $f_\Delta^{-n}x$ and $f_\Delta^{-n}y$.
Thus $\| \hLp_\Delta^n \psi|_{\Delta_\ell} \|_{Lip} \le e^{-\beta n} \| \psi \|_{Lip}$.

\medskip
\noindent
Step 2.  Now let $x \in \Delta_0$.  We have 
\begin{equation}
\label{eq:LY begin}
\hLp^n_\Delta \psi(x)  = \sum_{u \in \hf^{-n}_\Delta(x)} \psi(u) e^{t S_n\vf_{\Delta}^H(u)} 
= \sum_{u \in \hf^{-n}_\Delta(x)} (\psi(u) - \psi(v)) e^{t S_n\vf_{\Delta}^H(u)}
+ \psi(v) e^{t S_n\vf_{\Delta}^H(u)}
\end{equation}
where $E_n(u)$ is the $n$-cylinder containing $u$ and $v \in E_n(u)$ satisfies
$\psi(v) = \bm(E_n)^{-1} \int_{E_n} \psi \, d\bm$.

Due to bounded distortion of $\vf_{\Delta}$ given by (P2) and the fact that 
$\hf_\Delta^n(E_n) = \Delta_{0,i}$ for some $i \ge 0$,
the second term in the above sum is bounded by $C \bm(\Delta_{0,i})^{-1} \int_{E_n(u)} \psi \, d\bm$
and summing over $u$ yields the bound $C' \int_{\hDelta^n} \psi \, d\bm$.

To estimate the first term, we split it into two parts:  For $T>0$, let $A_{n,T}$ denote the set of
points in $\hDelta^n \cap \f_{\Delta}^{-n}(\Delta_0)$ 
that make at least $n/T$ returns to $\Delta_0$ by time $n$, and let
$A_{n,T}^c = (\hDelta^n \cap \f_{\Delta}^{-n}(\Delta_0)) \setminus A_{n,T}$.  Then letting $\ell(u)$ denote the level containing
$u$, 
\begin{equation}
\label{eq:LY first}
\begin{split}
& \sum_{u \in \hf^{-n}_\Delta(x)}  (\psi(u) - \psi(v)) e^{t S_n\vf_\Delta^H(u)} \\
& = \sum_{\substack{u \in \hf^{-n}_\Delta(x) \\ u \in A_{n,T}}} (\psi(u) - \psi(v)) e^{t S_n\vf_\Delta^H(u)}
 + \sum_{\substack{u \in \hf^{-n}_\Delta(x) \\ u \in A_{n,T}^c}} (\psi(u) - \psi(v)) e^{t S_n\vf_\Delta^H(u)} \\
& \le \sum_{\substack{u \in \hf^{-n}_\Delta(x) \\ u \in A_{n,T}}} C e^{-\delta n/T} 
\| \psi \|_{Lip} e^{\beta \ell(u)} \bm(E_n(u))
+ \sum_{\substack{u \in \hf^{-n}_\Delta(x) \\ u \in A_{n,T}^c}} C \| \psi \|_{Lip} e^{\beta \ell(u)}
\bm(E_n(u)) \\
& \le C e^{-\delta n/T} 
\| \psi \|_{Lip} \int_{\hDelta^n} 1_\beta \, d\bm 
 + C \| \psi \|_{Lip} \int_{\hDelta^n \cap f_\Delta^{-n}(\Delta_0) \cap A_{n,T}^c} 1_\beta \, d\bm .
\end{split}
\end{equation}
This first term above clearly contracts at an exponential rate $e^{-\delta n/T} \int_{\hDelta^n} 1_\beta \, d\bm < \overline \lambda_\beta^n$.  Thus it remains to prove the contraction
in the second term is sufficiently fast.

We associate to each $x \in \hf_\Delta^{-n}(\Delta_0) \cap A_{n,T}^c$ a sequence of 
times $r_1, \ldots r_s$, such that
$\hf_\Delta^{r_i}(x) \in \Delta_0$ for each $i$; it follows from the definition of
$A_{n,T}^c$ that $s \le n/T -1$.
Note also that $\sum_{i=1}^s r_i = n$ since $x \in \hf_\Delta^{-n}(\Delta_0)$.
However, each connected component 
in $\hf_\Delta^{-n}(\Delta_0) \cap A_{n,T}^c$ on level $\Delta_\ell$,
can be uniquely associated with a connected component in $\Delta_0$ such that each
$y = f_\Delta^{-\ell}(x)$ in this interval is associated with the sequence of 
return times $r_1, \ldots, r_s$ such that
$\sum_{i=1}^s r_i = n + \ell$.  Now
\[
\#\left\{ s\mbox{-tuples with } \sum_{i=1}^s r_i = n + \ell \right\} = { n + \ell - 1 \choose s-1 } 
\le {n + \ell - 1 \choose \frac nT -1} \; \le \; C (1+\eta_T)^{n+\ell} ,
\]
where $\eta_T \to 0$ as $T \to \infty$.
Also, at each return, $\bm(\tau=r_i) \le C_0 e^{-\alpha r_i}$.  So conditioning $s$ times, we have
\begin{equation}
\label{eq:small level}
\bm(A_{n,T} \cap \hDelta_\ell \cap \hf_\Delta^{-n}(\Delta_0)) 
\le \sum_{s=1}^{\frac nT -1} \sum_{\mbox{\scriptsize relevant $s$-tuples}} C_0^s e^{-(n+\ell)\alpha} 
\; \le \; C C_0^{n/T} (1+\eta_T)^{n+\ell} e^{-(n+\ell)\alpha} \; .
\end{equation}
Fix $0 < \ve < \min \{ \alpha - \beta, \alpha + \log \overline \lambda_\beta \}$.  
Then choose $T$ sufficiently large that 
$(1+\eta_T) C_0^{1/T} \le e^{\ve}$.  Using \eqref{eq:small level}, we have the contraction in the
second term at the end of \eqref{eq:LY first} bounded by
\begin{equation}
\label{eq:few returns}
\int_{\hDelta^n \cap f_\Delta^{-n}(\Delta_0) \cap A_{n,T}^c} 1_\beta \, d\bm
\le \sum_{\ell \ge 0} C e^{\beta \ell} e^{-(n+\ell)(\alpha - \ve)}
\le C e^{-n(\alpha-\ve)} .
\end{equation}
Putting this estimate together with \eqref{eq:LY first}, we have
\begin{equation}
\label{eq:LY infinity}
|\hLp_\Delta^n \psi(x)| \le C \sigma^n \| \psi \|_{Lip} + C \int_{\hDelta^n} \psi \, d\bm,
\end{equation}
for all $x \in \Delta_0$ and some constant $\sigma < \overline \lambda_\beta$.

Now for the bound on the Lipschitz norm, we estimate similarly for $x,y \in \Delta_{0,i}$,
\[
\begin{split}
|\hLp^n_\Delta \psi(x) - \hLp^n_\Delta \psi(y)| 
& \le \sum_{\substack{u \in \hf_{\Delta}^{n}(x) \\ v \in \hf_{\Delta}^{-n}(y)}}
|\psi(u) - \psi(v)| e^{S_n \vf_\Delta(u)} + |\psi(v)|
|e^{S_n \vf_\Delta(u)} - e^{S_n \vf_\Delta(v)}| \\
& \le \sum_{u,v} \| \psi \|_{Lip} e^{-\delta s(u,v)} e^{\beta \ell(u)} e^{S_n \vf_\Delta(u)}
+ C e^{-\delta s(x,y)} |\psi(v)| e^{S_n \vf_\Delta(u)} ,
\end{split}
\]
where we have used bounded distortion in the last line.  Notice that once we divide by
$e^{-\delta s(x,y)}$, this is precisely the same expression which had to be estimated in 
\eqref{eq:LY begin} and so is also bounded by \eqref{eq:LY infinity}.
Thus
\begin{equation}
\label{eq:LY lip}
\| \hLp_\Delta^n \psi |_{\Delta_0} \|_{Lip} \le C \sigma^n \| \psi \|_{Lip} + C \int_{\hDelta^n} \psi \, d\bm .\end{equation}

\noindent
Step 3.  We complete the proof of the proposition by taking $x, y \in \Delta_\ell$, $\ell < n$, using
Step 1 to estimate the first $\ell$ steps from $\Delta_\ell$ to $\Delta_0$, and then Step 2
to estimate the remaining $n-\ell$ steps:
\[
\|\hLp_\Delta^n \psi |_{\Delta_\ell} \|_\infty \le e^{-\beta \ell} \| \hLp_\Delta^{n-\ell} \psi |_{\Delta_0} \|_\infty
\le C e^{-\beta \ell} \sigma^{n-\ell} \| \psi \|_\B + C  |\psi|_{L^1} .
\]
A similar estimate holds for $\| \hLp_\Delta^n \psi \|_{Lip}$, completing the proof of the lemma.

\end{proof}

Since the unit ball of $\B$ is compactly embedded in $L^1(\bm)$, it follows from the
by-now classical results \cite{HH} together with
Lemma~\ref{lem:lasota yorke} that $\hLp_\Delta$ is quasi-compact as an
operator on $\B$:  Its essential spectral radius is bounded by $\sigma < \blambda$, and
for any $\sigma' > \sigma$, its
spectrum outside the disk of radius $\sigma'$ comprises only finitely many
eigenvalues each of finite multiplicity.  It follows from Lemma~\ref{lem:spectral radius}, 
that the spectral radius
of $\hLp_\Delta$ is at least $\blambda > \sigma$ so that the peripheral spectrum is nonempty and
lies outside the disk of radius $\sigma$.

Our next step in the proof is showing that $\hLp_\Delta$
has a real eigenvalue greater than $\sigma$ and a corresponding 
eigenvector which is strictly positive.
Once this is done, we will use it to prove 
that $\hLp_\Delta$ has a spectral gap.

\begin{lemma}
\label{lem:bounded}
There exists $M_0 > 0$ such that 
$\| \frac{\hLp_\Delta^n 1_\beta}{|\hLp_\Delta^n 1_\beta|_1} \|_\B \le M_0$ for all $n \ge 0$.
Moreover, the escape rate $-\log \lambda_\beta$ with respect to $1_\beta \bm$ exists
and $\lambda_\beta = \underline \lambda_\beta = \overline \lambda_\beta$
is the spectral radius of $\hLp_\Delta$ on $\B$.  
\end{lemma}
\begin{proof}
Let $\rho$, $|\rho| > \sigma$, be an eigenvalue of maximum modulus of $\hLp_\Delta$.
Since the peripheral spectrum of $\hLp_\Delta$ is finite dimensional, we
may choose $\rho$ such that $\rho$ has maximum
defect for eigenvalues on the circle of radius $|\rho|$.  Thus there exists $d \ge 1$ and 
$g_i \in \B$, $g_i \neq 0$, $i = 1, \ldots, d$,
such that $(\hLp_\Delta - \rho I) g_i = g_{i-1}$ for $i>1$ and $(\hLp_\Delta - \rho I) g_1 = 0$, and 
$d$ is the maximal such index for eigenvalues with modulus $|\rho|$.
Thus for $n > d$, we have
\[
\hLp_{\Delta}^n g_d = \sum_{i=0}^{d-1} {n \choose i} \rho^{n-i} g_{d-i} .
\]
Integrating over $\hDelta$, we use the above identity to obtain the following lower bound,
\[
\begin{split}
\int_{\hDelta} \hLp_{\Delta}^n |g_d| \, d\bm 
& \ge 
\int_{\hDelta} \left| |\rho|^{n-d+1} {n \choose d-1 } |g_1| - 
\left|\sum_{i=0}^{d-2} {n \choose i} \rho^{n-i} g_{d-i} \right| \right| \, d\bm \\
& \ge \int_{\hDelta}  |\rho|^{n-d+1} {n \choose d-1 } |g_1| \, d\bm
- \int_{\hDelta} \left|\sum_{i=0}^{d-2} {n \choose i} \rho^{n-i} g_{d-i} \right|  \, d\bm
\end{split}
\]
Since the first term is a polynomial of degree $d-1$ while the second term has degree at
most $d-2$,
there exists $C_1>0$ and $N > 0$ such that for $n \ge N$,
\[
\int_{\hDelta} \hLp_\Delta^n |g_d| \ge  C_1 n^{d-1} |\rho|^n .
\]
Thus for $n \ge N$,
\begin{equation}
\label{eq:beta lower}
C_1 n^{d-1} |\rho|^n   \le \int_{\hDelta} \hLp_\Delta^n |g_d| \, d\bm
= \int_{\hDelta^n} |g_d| \, d\bm  \; \le \; \|g_d\|_\infty \int_{\hDelta^n} 1_\beta \, d\bm 
= \| g_d \|_\infty \int_{\hDelta} \hLp^n_\Delta 1_\beta \, d\bm .
\end{equation}

On the other hand, using the spectral decomposition of $\hLp_\Delta$ given by quasi-compactness,
we have $\| \hLp_\Delta^n 1_\beta \|_\B \le C_2 n^{d-1} |\rho|^n$ for some $C_2 >0$
and all $n \ge 0$.  Thus the bound
on $\| \frac{\hLp_\Delta^n 1_\beta}{|\hLp_\Delta^n 1_\beta|_1} \|_\B$ follows using this
together with \eqref{eq:beta lower} for $n \ge N$. The bound for $n < N$ is obtained by taking the maximum
over the finitely many terms and noting that each term is finite since $\hLp_\Delta$ is a bounded
operator on $\B$ and $\bm(\hDelta^n) >0$ for each $n$.

Now \eqref{eq:beta lower} implies $\underline{\lambda}_\beta \ge |\rho|$, while
Lemma~\ref{lem:spectral radius} implies $\blambda_\beta \le |\rho|$. 
Thus $\lambda_\beta$ exists and is the spectral radius of $\hLp_\Delta$.
\end{proof}

From now on, we use the notation $\lambda_\beta$ rather than $\overline \lambda_\beta$ since we
know the escape rate with respect to $1_\beta \bm$ exists.

Let 
\begin{equation}
\label{eq:psi n}
\psi_n = \frac{\sum_{k=1}^n \lambda_\beta^{-k} \hLp_\Delta^k 1_\beta}{\sum_{k=1}^n \lambda_\beta^{-k}
|\hLp_\Delta^k 1_\beta|_1} .
\end{equation}
Notice that $(\psi_n)_n$ is a sequence of probability densities and that\footnote{\label{ft:sum} Here we use that for any two series of positive terms,
$\frac{\sum_k a_k}{\sum_k b_k} \le \sup_k \frac{a_k}{b_k}$, whenever $\sum_k b_k < \infty$.} by Lemma~\ref{lem:bounded},
\begin{equation}
\label{eq:psi norm bound}
\| \psi_n \|_{\B} \le \frac{\sum_{k=1}^n \lambda_\beta^{-k} \| \hLp_\Delta^k 1_\beta \|_{\B}}
{\sum_{k=1}^n \lambda_\beta^{-k} | \hLp_\Delta^k 1_\beta |_1} 
\le \max_{1\le k \le n} \frac{\| \hLp_\Delta^k 1_\beta \|_\B}{|\hLp_\Delta^k 1_\beta|_1} \le M_0.
\end{equation}
Then since $\psi_n$ lies in a ball of radius $M_0$ in $\B$ for all $n$, and this ball is compact in
$L^1(\bm)$, we may choose a subsequence $(n_i)_i$ such that
$\psi_{n_i}$ converges in $L^1(\bm)$ to a function  $\psi_* \in \B$ with $\| \psi_* \|_\B \le M_0$.
Now
\[
\begin{split}
\hLp_\Delta \psi_* & = \lim_{i \to \infty} \frac{\sum_{k=1}^{n_i} 
\lambda_\beta^{-k} \hLp_\Delta^{k+1} 1_\beta}
{\sum_{k=1}^{n_i} \lambda_\beta^{-k} |\hLp_\Delta^k 1_\beta|_1} 
= \lambda_\beta \lim_{i \to \infty} \frac{\sum_{k=1}^{n_i} 
\lambda_\beta^{-k-1} \hLp_\Delta^{k+1} 1_\beta}
{\sum_{k=1}^{n_i} \lambda_\beta^{-k} |\hLp_\Delta^k 1_\beta|_1} \\
& = \lambda_\beta  \left[ \lim_{i \to \infty} \frac{\sum_{k=1}^{n_i} 
\lambda_\beta^{-k} \hLp_\Delta^k 1_\beta}
{\sum_{k=1}^{n_i} \lambda_\beta^{-k} |\hLp_\Delta^k 1_\beta|_1}
+ \frac{ - \lambda_\beta^{-1} \hLp_\Delta 1_\beta + \lambda^{-n_i-1} \hLp_\Delta^{n_i+1} 1_\beta }
{\sum_{k=1}^{n_i} \lambda_\beta^{-k} |\hLp_\Delta^k 1_\beta|_1} \right] .
\end{split}
\]
The first fraction converges to $\psi_*$ by choice of the subsequence $(n_i)_i$.  By
the proof of Lemma~\ref{lem:bounded}, the numerator of the second fraction has $\| \cdot \|_\B$-norm
bounded above by $C n_i^{d-1}$, while by \eqref{eq:beta lower}, the denominator is bounded below by
$C' n_i^d$, thus the second fraction converges to 0 in $\B$ (and also in $L^1(\bm)$) as $i \to \infty$.
This proves that $\hLp_\Delta \psi_* = \lambda_\beta \psi_*$, so that $\lambda_\beta$ is in
the spectrum of $\hLp_\Delta$.

Note that $\psi_* \ge 0$ and $\int \psi_* \, d\bm = 1$, so that necessarily $\lambda_\beta < 1$.
This is because
\[
\lambda_\beta^n = \lambda_\beta^n \int \psi_* \, d\bm = \int \hLp_\Delta^n \psi_* \, d\bm
= \int_{\hDelta^n} \psi_* \, d\bm \xrightarrow{n \to \infty} 0 .
\]
It follows that $\overline \lambda \le \lambda_\beta < 1$,
which we have not assumed is true a priori.

We use the following lemma to show that in fact $\psi_*$ is strictly positive on all
of $\hDelta$.
Define 
\[
\| \psi \|_{\log} = \sup_{\ell, j} \mbox{Lip}(\log \psi|_{\hDelta_{\ell,j}}),
\]
where Lip$(\cdot)$ denotes the Lipschitz constant with respect to the metric
$d_\delta(\cdot, \cdot) = e^{-\delta s(\cdot, \cdot)}$.
For $M>0$, set
$\B_{\log}(M) = \{ \psi \in \B : \psi \geq 0, |\psi|_1 = 1, \| \psi \|_\infty \le M, \| \psi \|_{\log} \le M \} $.

\begin{lemma}
\label{lem:regular}
Let $\psi_n$ be defined by \eqref{eq:psi n}.
There exists $M >0$ such that $\psi_n |_{\Delta_0} \in \B_{\log}(M)$ for all $n \ge 0$.
\end{lemma}
\begin{proof}
We show the above property for the normalized transfer operator 
$$\hNp_\Delta^n \psi := \frac{\hLp^n_\Delta \psi}{|\hLp^n_\Delta \psi|_1},$$
for any $\psi \in \B$ with $\| \psi \|_{\log} < \infty$ and $\int_{\hDelta^n} \psi \, d\bm >0$, $\forall n$.
Given such a $\psi$, clearly $\hLp^n_\Delta \psi \ge 0$ and $|\hNp_\Delta^n \psi |_1 = 1$
since the normalization is well defined.  So the first two properties of 
$\B_{\log}(M)$ are obviously satisfied by $\hNp_\Delta^n \psi$ for all $n \ge 0$.  

To estimate the $\log$-Lipschitz constant, 
let $x, y \in \Delta_{0,i}$ and denote by $u \in \hf_\Delta^{-n}$ and $v \in \hf_{\Delta}^{-n}$
two points in in the same $n$-cylinder in $\hDelta$.
We estimate
\begin{equation}
\label{eq:psi log}
\begin{split}
\hLp_\Delta^n \psi(x) & = \sum_{u \in \hf_\Delta^{-n}x} \psi(u) e^{S_n \vf_\Delta(u)}  
\; \le \; \sum_{v \in \hf_\Delta^{-n}(y)} e^{\| \psi \|_{\log} d_\delta(u,v)} \psi(v) e^{S_n \vf_\Delta(v)} 
(1+C_d d_{\delta}(x,y)) \\
& \le e^{\| \psi \|_{\log} d_\delta(x,y)} (1+C_d d_{\delta}(x,y)) \hLp_{\Delta}^n \psi(y)
\end{split}
\end{equation}
where we have used bounded distortion (P3) in the second line and 
$d_\delta(u,v) \le d_\delta(x,y)$ in the third.  This yields
\[
\mbox{Lip}(\log \hLp_\Delta^n \psi|_{\Delta_0}) \le \| \psi \|_{\log} + C_d,
\]
where we have used the estimate $\log (1+z) \le z$ for $z \ge 0$.  
Since $\| \cdot \|_{\log}$ is scale invariant, we have
$\| \hNp_\Delta^n \psi |_{\Delta_0} \|_{\log} \le \| \psi \|_{\log} +C_d$ for all $n \ge 0$.

Finally, we estimate the $L^\infty$ norm of $\hLp_{\Delta}^n \psi |_{\Delta_0}$.
Let $x \in \Delta_0$, and again using the notation $u \in \hf_{\Delta}^{-n}(x)$, let
$v$ denote a point in the $n$-cylinder $E_n(u)$ containing $u$ such that 
$\psi(v) \le \frac{1}{\bm(E_n(u))} \int_{E_n(u)} \psi \, \bm$.
Then since
$e^{S_n\vf_\Delta(u)} \le  (1+C_d) \bm(E_n(u))/\bm(\Delta_0,i)$ 
for some $i$ by (P3), we estimate following  \eqref{eq:psi log}
\[
\begin{split}
| \hLp_\Delta^n \psi(x)| & \le  \sum_{u \in \hf_\Delta^{-n}(x)} e^{\| \psi \|_{\log}} \psi(v) (1+C_d) 
\tfrac{\bm(E_n(u))}{\bm(\Delta_{0,i})} \\
& \le C e^{\| \psi \|_{\log}} \sum_{u \in \hf_{\Delta}^{-n}(x)} \int_{E_n(u)} \psi \, d\bm  
\le C e^{\| \psi \|_{\log}} \int_{\hDelta^n} \psi \, d\bm ,
\end{split}
\]
where $C = (1+C_d)/\min_i \{ \bm(\Delta_{0,i}) \}$.
Dividing by $| \hLp_\Delta^n \psi|_1 = \int_{\hDelta^n} \psi \, d\bm$ 
completes the estimate on the $\| \cdot \|_\infty$ norm.

Now since $\| 1_\beta \|_{\log} = 0$, the above argument implies 
$\| \hNp_\Delta^k 1_\beta |_{\Delta_0} \|_\infty \le C$ and 
$\| \hNp_\Delta^k 1_\beta |_{\Delta_0} \|_{\log} \le C_d$ for each $k \ge 1$.
Equation~\ref{eq:psi norm bound} implies that the uniform bound on 
$\| \hNp_\Delta^k 1_\beta |_{\Delta_0} \|_\infty$ passes to 
$\| \psi_n |_{\Delta_0} \|_\infty$.  Finally, while $\| \cdot \|_{\log}$ is not linear,
it does satisfy the convex inequality,
\[
\left\|  \frac 1n \sum_{k=1}^n \lambda_\beta^{-k} \hLp_\Delta^k 1_\beta \right\|_{\log}
\le \max_{1 \le k \le n} \| \hLp_\Delta^k 1_\beta \|_{\log} \le C_d. 
\]
Finally, the scale invariance of $\| \cdot \|_{\log}$ implies that this bound passes to $\psi_n$ for each
$n \ge 1$.
\end{proof}

Since $\| \psi_n |_{\Delta_0} \|_{\log} \le M$ for all $n \ge 1$, we have 
$\| \psi_* |_{\Delta_0} \|_{\log} \le M$, and so for each $i$, either $\psi_* > 0$ on 
$\Delta_{0,i}$ or $\psi_* \equiv 0$ on $\Delta_{0,i}$.  
Since for $x \in \hDelta_\ell$ by conditional invariance, 
$\psi_*(x) = \lambda_\beta^{-\ell} \hLp^\ell \psi_*(x) 
= \lambda_\beta^{-\ell} \psi_* \circ \hf^{-\ell}_\Delta(x)$,
the second alternative implies $\psi_* \equiv 0$ on the entire column above $\Delta_{0,i}$.
Again using conditional invariance, this implies that $\psi_* \equiv 0$ on every 
$\Delta_{0,j}$ that eventually maps to $\Delta_{0,i}$.  By transitivity, this is the entire
base $\Delta_0$ and so $\psi_* \equiv 0$ on 
$\hDelta$, which
is impossible since $\int \psi_* d\bm =1$.  Thus there exists $\delta_0 > 0$ such that
$\psi_* \ge \delta_0$ on $\Delta_0$,
and again using conditional invariance, we conclude that $\psi_* \ge \delta_0$ 
on all of $\hDelta$.

The following lemma gives an equivalent expression for $\| \cdot \|_{\log}$, which will
be convenient for the proof of Proposition~\ref{prop:spectral decomp}.

\begin{lemma}
\label{lem:equiv}
Suppose $\psi > 0$.  For each $\dlj$,
\[
\| \psi|_{\dlj} \|_{\log} \le \sup_{x,y \in \dlj} \tfrac{|\psi(x) - \psi(y)|}{\psi(y) d_\delta(x,y)} \le \| \psi \|_{\log} e^{\| \psi \|_{\log}}  . 
\]
\end{lemma}

\begin{proof}
Let $x, y \in \dlj$, $x \neq y$, and set $z = \log (\psi(x)/\psi(y))$.  Then,
\[
| \tfrac{\psi(x) - \psi(y)}{\psi(y)}| = |e^z -1| \le |z| e^{|z|} \le \| \psi \|_{\log} e^{\| \psi \|_{\log}} d_{\delta}(x,y),
\]
proving the second inequality.  The first inequality follows similarly using the fact that
$\log (1+ w) \le w$ for $w \ge 0$.
\end{proof}

\begin{proposition}
\label{prop:spectral decomp}
$\hLp_\Delta$ has a spectral gap, i.e. $\lambda_\beta$ is simple and all other eigenvalues  
have modulus strictly less than $\lambda_\beta$.
\end{proposition}

\begin{proof}
We remark that the stronger assumptions used in \cite{BruDemMel10} (see Remark~\ref{rem:BDM})
allow one to show that 
$\hNp_\Delta(\B_{\log}(M)) \subset \B_{\log}(M)$ and that the semi-norm 
$\| \hLp_\Delta^n \psi \|_{\log}$ 
obeys a uniform Lasota-Yorke inequality for $\psi \in \B_{\log}(M)$.  We prove neither property
here under our weaker assumptions, and give a substantially different proof of the spectral gap.

We begin by assuming that there exists $g \in \B$ such that 
$\hLp_{\Delta} g = \lambda_\beta e^{i \omega} g$ for some $\omega \in [0, 2\pi)$.
We will show this implies $\omega=0$ and $g$ is a multiple of $\psi_*$.  

Using the fact that $g$ is an eigenfunction, we have
\begin{equation}
\label{eq:rotate}
g|_{\hDelta_\ell} = \lambda_\beta^{-\ell} e^{-i \omega \ell} g|_{\Delta_0 \cap \f_\Delta^{-\ell}(\hDelta_\ell)},
\end{equation}
so that $g$ grows like $\lambda_\beta^{-\ell}$ times a rotation up the levels of the tower.

Since $\psi_* \ge \delta_0$ on $\Delta_0$, we may choose $K > |g|_{\Delta_0}|_\infty$ 
sufficiently large so that $g_1 := (Re(g) + K \psi_*)/C >0$ on 
$\hDelta$,
where $Re(g)$ denotes the real part of $g$ and $C>0$ is chosen so that 
$\int_{\hDelta} g_1 \, d\bm = 1$.
Note that by \eqref{eq:rotate} and the conditional invariance of $\psi_*$,
by choice of $K$, there exists $\delta_1$ such that for each $\ell \ge 0$,
\begin{equation}
\label{eq:uniform growth}
\delta_1 \lambda_\beta^{-\ell} \le g_1|_{\hDelta_\ell} \le \delta_1^{-1} \lambda_\beta^{-\ell} .
\end{equation}
Define $g_s = sg_1 + (1-s) \psi_*$ and $J = \{ s \in \R : \inf_{\hDelta} g_s > 0 \}$.
Note that $J$ contains $[0,1]$ and by \eqref{eq:uniform growth}, $J$ is open. 
To see that in fact $J \supset \mathbb{R}^+$, we will use the following lemma.

\begin{lemma}
\label{lem:log contract}
Let $\psi \in \B$ with $\| \psi \|_{\log} < \infty$ and let $A_{n,T} \subset \hDelta^n \cap \f_\Delta^{-n}(\Delta_0)$ be as defined in the proof of Lemma~\ref{lem:lasota yorke}.  Suppose there exists a function $r(n)$ with
$r(n) \to 0$ as $n \to \infty$ such that
\[
\frac{\int_{A_{n,T}^c} \psi \, d\bm}{\int_{\hDelta^n \cap \f_\Delta^{-n}(\Delta_0)} \psi \, d\bm}
\le r(n), \qquad \mbox{for } n \ge 0.
\]
Then 
\[
\| \hNp_\Delta^n \psi |_{\Delta_0} \|_{\log} \le C_1 \max \{ r(n), e^{-\delta n/T} \} \| \psi \|_{\log} 
e^{3 \| \psi \|_{\log} }+ C_d,
\]
for a uniform constant $C_1$ depending only on $C_d$ and the minimum length of 
an element $\Delta_{0,i}$.
\end{lemma}

\begin{proof}[Proof of Lemma]
The estimate is a refined version of the one derived in the proof of Lemma~\ref{lem:regular}.
For $x,y \in \Delta_{0,i}$, let $u \in \f_\Delta^{-n}(x)$, $v \in \f_\Delta^{-n}(y)$, denote corresponding
points in the same $n$-cylinder.  Note that for each pair of pre-images,
$u \in A_{n,T}$ if and only if $v \in A_{n,T}$.  
In light of Lemma~\ref{lem:equiv}, it is equivalent to estimate,
\begin{equation}
\label{eq:log split}
\begin{split}
& \frac{|\hLp_\Delta^n \psi(x) - \hLp_\Delta^n \psi(y)|}{\hLp_\Delta^n \psi(y) e^{-\delta s(x,y)}}
=  
\frac{\sum_u(\psi(u) - \psi(v))  e^{S_n\vf_\Delta(u)}}
{e^{-\delta s(x,y)} \sum_v \psi(v)e^{S_n\vf_\Delta(v)}}
+ \frac{\sum_v \psi(v) (e^{S_n \vf_\Delta(u)} - e^{S_n\vf_\Delta(v)})}
{e^{-\delta s(x,y)} \sum_v \psi(v)e^{S_n\vf_\Delta(v)}} \\
& \le \| \psi \|_{\log} e^{\| \psi \|_{\log}} \frac{\sum_u e^{-\delta s(u,v)}  \psi(v) e^{S_n\vf_\Delta(u)}}
{e^{-\delta s(x,y)} \sum_v \psi(v)e^{S_n\vf_\Delta(v)}}+ C_d \\
& \le (1+C_d) \| \psi \|_{\log} e^{\| \psi \|_{\log}} \left( e^{-\delta n/T} \frac{\sum_{v \in A_{n,T}} \psi(v) e^{S_n\vf_\Delta(v)}}
{\sum_v \psi(v)e^{S_n\vf_\Delta(v)}}  
+ \frac{\sum_{v \in A_{n,T}^c} \psi(v) e^{S_n\vf_\Delta(v)}}{\sum_v \psi(v)e^{S_n\vf_\Delta(v)}} \right)
+ C_d \\
& \le (1+C_d) \| \psi \|_{\log} e^{\| \psi \|_{\log} } \left( e^{-\delta n/T} + \frac{\sum_{v \in A_{n,T}^c} \psi(v) e^{S_n\vf_\Delta(v)}}{\sum_v \psi(v)e^{S_n\vf_\Delta(v)}} \right)
+ C_d,
\end{split}
\end{equation}
where we have used (P3), Lemma~\ref{lem:equiv}, and the
fact that $s(u,v) - s(x,y) \ge n/T$ for $u,v \in A_{n,T}$.  
Again using bounded distortion, we may replace $e^{S_n\vf_\Delta(v)}$ by 
$\bm(E_n(v))/\bm(\Delta_{0,i})$ for some $i \ge 0$,
where $E_n(v)$ is the $n$-cylinder containing $v$; similarly, 
\[
\psi(v) = e^{\pm \| \psi \|_{\log}} (\bm(E_n(v)))^{-1} \int_{E_n(v)} \psi \, d\bm
\]
by the log-Lipschitz continuity of $\psi$.  Thus we estimate \eqref{eq:log split} by, 
\[
\frac{|\hLp_\Delta^n \psi(x) - \hLp_\Delta^n \psi(y)|}{\hLp_\Delta^n \psi(y) e^{-\delta s(x,y)}}
\le C_1 \| \psi \|_{\log} e^{\| \psi \|_{\log}} \left( e^{-\delta n/T} + 
e^{2 \| \psi \|_{\log}} 
\frac{\int_{A_{n,T}^c} \psi \, d\bm}{\int_{\hDelta^n \cap \f_{\Delta}^{-n}(\Delta_0)} \psi \, d\bm } \right)
+ C_d,
\]
which completes the proof of the lemma by assumption on $\psi$.
\end{proof}

Returning to the proof of the proposition, we will apply Lemma~\ref{lem:log contract}
to $g_s$ for $s \in J$.  Note that since $g_s$ is bounded away from 0, we have
$\| g_s \|_{\log} < \infty$.  While $g_s$ is not an eigenfunction for
$\hLp_\Delta$, we do have for each $n \in \N$,
\[
\hLp_\Delta^n g_1 = \lambda_\beta^n (Re(e^{i \omega n} g) + K \psi_* )/C .
\]
Thus there exists a subsequence $(n_j)_{j \in \N}$ such that 
$\lim_{j \to \infty} \lambda_\beta^{-n_j} \hLp_{\Delta}^{n_j} g_1 = g_1$ and it follows that
\[
\lim_{j \to \infty} \lambda_\beta^{-n_j} \hLp_\Delta^{n_j} g_s = g_s \qquad \mbox{for each $s \in \R$.}
\]

By conformality of $\bm$, 
$\int_{\Delta_0} \hLp_\Delta^n g_s \, d\bm  =  \int_{\hDelta^n \cap \f_\Delta^{-n}(\Delta_0)} g_s \, d\bm$
 for each $n \in \N$.  On the other hand, using the
definition of $g_s$, 
\[
\int_{\Delta_0} \hLp_\Delta^n g_s \, d\bm 
= \lambda_\beta^n \int_{\Delta_0}  ( g_s + \tfrac{s}{C} Re((e^{i\omega n}-1)  g) ) \, d\bm .
\]
We claim that in fact, for $s>0$, $s \in J$, the integral on the right must be 
bounded below by some $\kappa_s >0$, independently of $n$.  If not, there is a subsequence 
$(n_k)_k$ such that $\int_{\Delta_0} Re((e^{i\omega n_k} - 1) g) \to - \frac Cs \int_{\Delta_0} g_s < 0$.
Note that for $s' > s$, we have
\[
\frac{g_s}{s} = g_1 + (\tfrac 1s - 1) \psi_* > g_1 + (\tfrac{1}{s'} -1) \psi _0 = \frac{g_{s'}}{s'},
\] 
since $\psi_* >0$.  So 
\[
\lambda_\beta^{-n_k} \int_{\Delta_0} \hLp_\Delta^{n_k} g_{s'} = 
\int_{\Delta_0}  (g_{s'} + \tfrac{s'}{C} Re((e^{i\omega n_k} - 1) g))
\to \int_{\Delta_0} s' \Big(\frac{g_{s'}}{s'} - \frac{g_s}{s} \Big) < 0,
\]
which implies that $g_{s'}$ cannot be positive for any $s' >s$, contradicting the 
fact that $J$ is open.

With the claim proved, we estimate, using \eqref{eq:few returns} and the
conformality of $\bm$,
\[
\frac{\int_{A_{n,T}^c} g_s \, d\bm}{\int_{\hDelta^n \cap \f_\Delta^{-n}(\Delta_0)} g_s \, d\bm }
\le \frac{ \| g_s \|_\infty \int_{A_{n,T}^c} 1_\beta \, d\bm}
{\int_{\Delta_0} \hLp_\Delta^n g_s \, d\bm }
\le \frac{ \| g_s \|_\infty e^{-n(\alpha - \ve)}}{ \lambda_\beta^n \kappa_s }
=: r(n) ,
\]
and clearly $r(n) \to 0$ at an exponential rate by choice of $\ve$ and $T$.  

Invoking Lemma~\ref{lem:log contract},
since $g_s = \lim_j \lambda_\beta^{-n_j} \hLp_\Delta^{n_j} g_s$,
we have $\| g_s |_{\Delta_0} \|_{\log} \le C_d$ whenever $s \in J$.  Indeed, more is true,
since for each fixed $\ell > 0$ and $n_j > \ell$, 
\[
\| \hLp_{\Delta}^{n_j} g_s |_{\hDelta_\ell} \|_{\log}
\le \| \hLp_{\Delta}^{n_j-\ell} g_s |_{\Delta_0} \|_{\log} 
\le C_1 \| g_s \|_{\log} e^{3 \|g_s\|_{\log}} \max \{ r(n_j-\ell),  e^{-\delta (n_j-\ell)/T} \} + C_d,
\]
so using the scale invariance of $\| \cdot \|_{\log}$ to normalize by $\lambda_\beta^{-n_j}$
and letting $n_j \to \infty$,
we conclude that $\| g_s |_{\Delta_\ell} \|_{\log} \le C_d$ for each $\ell \in \N$.
Letting $C_2 = \max \{ C_d, (\bm(\Delta_{0,i}))^{-1} \}$ and using \eqref{eq:rotate}, we have  
$g_s \in \B_{\log}(C_2)$ for each $s \in J$, $s>0$.

Now let $s_0 > 1$ be the
right endpoint of $J$.  Since $g_s \to g_{s_0}$ uniformly on each $\Delta_{\ell,j}$ as
$s \to s_0$ from the left, we have $g_{s_0} \in \B_{\log}(C_2)$ as well.
The following lemma from \cite{BruDemMel10} completes the proof of the proposition.

\begin{lemma}(\cite[Lemma 3.2]{BruDemMel10})
\label{lem:positive}
$g_{s_0}$ is bounded away from 0 on $\Delta$.
\end{lemma}

We refer the interested reader to \cite{BruDemMel10} for the proof of this lemma since
it requires no changes in the current setting.  The proof uses only the mixing property
of $\f_{\Delta}$ and the fact that $g_{s_0} \in \B_{\log}(C_2)$.

The lemma implies $s_0 \in J$ and
so $J \supset \R^+$.  The fact that $g_s > 0$ for all $s > 0$ implies 
$g_1 \ge \psi_*$.  But since $\int g_1 \, d\bm = \int \psi_* \, d\bm =1$, we conclude that
in fact $g_1 = \psi_*$.   This implies that $Re(g) = (C-K) \psi_*$ and using the linearity
of $\hLp_{\Delta}$, we conclude that $Im(g)$ is also a multiple of $\psi_*$ and 
so $\omega =0$. 

We have proved that $\lambda_\beta$ has strictly larger modulus than all other 
eigenvalues of $\hLp_{\Delta}$
and that its multiplicity is one.  The last step is to
eliminate Jordan blocks for $\lambda_\beta$.  Suppose there exists
$g \in \B$ such that $(\hLp_\Delta - \lambda_\beta I)g = \psi_*$.  It follows that
$\hLp_\Delta^n g = \lambda_\beta^n g + n \lambda_\beta^{n-1} \psi_*$, so that
\[
\hLp_\Delta^n (g - \lambda_\beta^{-1} n \psi_*) = \lambda_\beta^n g .
\]
Thus for $x \in \Delta_\ell$,
\[
g(x) = \lambda_\beta^{-\ell} \hLp_\Delta^\ell (g - \lambda_\beta^{-1} \ell \psi_*)(x)
= \lambda_\beta^{-\ell} (g - \lambda_\beta^{-1} \ell \psi_*)\circ \f_\Delta^{-\ell}(x) .
\]
For $\ell$ sufficiently large, $g - \lambda_\beta^{-1} \ell \psi_* < 0$ on $\Delta_0$, so there
exists $L >0$ such that $g < 0$ on $\cup_{\ell \ge L} \hDelta_\ell$.  Now choose $K>0$
so large that $\bar g := g - K \psi_* < 0$ on $\cup_{\ell \le L} \hDelta_\ell$.  Since
$\bar g < g$, we have $\bar g < 0$ on $\hDelta$.  Thus for each $n$, 
\[
0 > \lambda_\beta^{-n} \int_{\hDelta^n} \bar g \, d\bm = \lambda_\beta^{-n} \int_{\hDelta} \hLp_\Delta^n \bar g \, d\bm = \int_{\hDelta} \bar g \, dm + n \lambda_\beta,
\]
which is a contradiction.
\end{proof}

Now that we have proved the existence of a spectral gap for $\hLp_\Delta$, 
the items of the following theorem follow from \cite[Section 2]{BruDemMel10}.

\begin{theorem} 
\label{thm:exp conv}
Assume $(f_\Delta, \Delta; H)$ is mixing and satisfies properties (P1)-(P4).
Then $\hLp_\Delta$ has a spectral gap.  Let $\lambda$ denote the largest eigenvalue of
$\hLp_\Delta$ and let $\bar{g}$ denote the corresponding normalized eigenfunction.
\begin{enumerate}
  \item[(a)]  The escape rates with respect to $\bm$ and $1_\beta \bm$ exist and equal 
  $- \log \lambda$; in particular $\lambda = \lambda_\beta$. 
\item[(b)] $ \displaystyle
\log \lambda = \sup \left\{ h_{\eta}(f_\Delta) + \int_{\hDelta} \vf_\Delta^H d\eta 
\mid \eta \in \mathcal{M}_{f_\Delta} , \eta(-\vf_\Delta^H) < \infty \right\}\ . $
\item[(c)]   The following limit defines a probability measure $\bar{\nu}$,
\[
\bar{\nu}(\varphi) = \lim_{n\to \infty} \lambda^{-n} \int_{\hDelta^n} \psi \, \bar{g} \, d\bm
\qquad \mbox{for all $\psi \in \B_0$, }   
\]
where $\B_0$ is the set of all bounded functions in $\B$ whose Lipschitz constant
is also uniformly bounded.  The measure 
$\bar{\nu}$ attains the supremum in (b), i.e. it is an equilibrium state for
$\vf_\Delta^H$.

  \item[(d)] There exist constants
$D>0$ and $\sigma_0 <1$ such that for all $\psi \in \B$,
\[
\| \lambda^{-n}\hLp^n_\Delta \psi - d(\psi)\bar{g}\|_\B \leq D\|\psi \|_\B \sigma_0^n, \; \; \;
\mbox{where }
d(\psi) = \lim_{n\to \infty} \lambda^{-n} \int_{\hDelta^n} \psi \, d\bm < \infty .
\]
Also, for any $\psi \in \B_0$ with $\bar{\nu}(\psi) > 0$,
\[
\left| \frac{\hLp^n_\Delta \psi}{|\hLp^n_\Delta \psi|_{L^1(\bm)}} - \bar{g} \right|_{L^1(\bm)} 
\le D \| \psi \|_\B \sigma_0^n.
\]
\end{enumerate}
\end{theorem}


\subsection{Application of abstract results and proof of Theorem~\ref{thm:rate gap}}
\label{apply tower}

In this section, we complete the proof of Theorem~\ref{thm:rate gap} by projecting the results
of Theorem~\ref{thm:exp conv} to our underlying map $f$ with neutral fixed point.
In order to invoke these results, we fix $t \in [0,1)$ and show that our constructed tower
satisfies assumptions (P1)-(P4) with respect to the measure $\bm_t$ and
potential $\vf_\Delta^H$ induced by $t\phi^H - p(t)$.

(P1) follows immediately from Lemma~\ref{lem:exponential tail} with $\alpha = p(t)>0$.

(P2) follows from Lemma~\ref{lem:base tower decay} and the assumption that 
$-\log \blambda_t < p(t)$.

(P4) is satisfied since $Df$ is bounded above and also below by 1 so that
$S_n\vf_\Delta(x)$ grows at most linearly in $n$ for $x \in \Delta_0$.

It remains to verify (P3).  First choose $0<\delta \le \frac{\gamma \log 2}{1+ \gamma}$ 
for reasons to be made clear below.
Standard estimates (see for example \cite[Lemma 3.1]{DemFer13})
imply that there exists $C_d > 0$ such that for all $n \in \N$: 
\begin{enumerate}
  \item[(D1)]  if $f^n(x) \in Y$, then $Df^n(x) \ge \max \{ 2, C_d^{-1} n^{1 + \frac{1}{\gamma}} \}$;
  \item[(D2)]  if $f^i(x), f^i(y)$ lie in the same element of $\P_1$ for each
  $0 \le i \le n$, then 
  \[
  \left| \log \frac{Df^n(x)}{Df^n(y)} \right| \le C_d |f^n(x) - f^n(y)|^{\frac{\gamma}{\gamma+1}} .
  \]
\end{enumerate}

We will need the following lemma.

\begin{lemma}
\label{lem:smooth lift}
Suppose $\delta/ \log 2 \le q \le 1$. 
Let $\psi \in C^q(I)$ and define $\tpsi$ on $\Delta$ by $\tpsi = \psi \circ \pi$.
Then $| \tpsi |_\infty \le |\psi|_\infty$ and Lip$(\tpsi) \le C|\psi|_{C^q}$ for some constant
$C$ depending on the minimum length of $\Delta_{0,i}$.
\end{lemma}

\begin{proof}
The bound $| \tpsi |_\infty \le |\psi|_\infty$ is immediate.  To prove the bound on the Lipschitz
norm of $\tpsi$, let $x, y \in \dlj$ and estimate
\[
\frac{|\tpsi(x) - \tpsi(y)|}{d_\delta(x,y)}
= \frac{|\psi(\pi(x))- \psi(\pi(y))|}{|\pi(x) - \pi(y)|^q} \cdot \frac{|\pi(x) - \pi(y)|^q}{e^{-\delta s(x,y)}}.
\]
The first ratio above is bounded by $|\psi|_{C^q(I)}$.  To estimate the second ratio, note that
if $s(x,y) = n$, then by construction of $\Delta$, $f^i(\pi(x))$ and $f^i(\pi(y))$ lie in the
same element of $\P_1$ for all $i \le \tau^n(x)$.  By (D1), 
$Df^{\tau^n}(\pi(x)) \ge 2^n$, so that using also (D2), $|\pi(x) - \pi(y)| \le C 2^{-n}$. Thus
$|\pi(x) - \pi(y)|^q \le C e^{- \delta n}$ as long as $q \ge \delta/\log 2$, as required.
\end{proof}

By (D2) above, we see that the potential $\phi = - \log Df$ is H\"older continuous
with exponent $\gamma/(1+\gamma)$.  
By Lemma~\ref{lem:smooth lift},
the induced potential $\vf_\Delta$ will be Lipschitz in the metric $d_\delta$ if
$\frac{\delta}{\log 2} \le \frac{\gamma}{1+\gamma}$, i.e.~if 
$\delta \le \frac{\gamma \log 2}{1+\gamma}$, which is what we chose initially.  Then (D2) also
implies the bounded requirement of (P3) for $e^{S_n\phi}$ with this choice of $\delta$.

This extends to the potential $t S_n \phi - n p(t)$ for $t \in [0,1]$
since 
\[
\begin{split}
|e^{tS_n \phi(x) - np(t)} & - e^{tS_n \phi(y) - np(t)}|
 = e^{tS_n\phi(y) - np(t)} |e^{tS_n\phi(x) - tS_n\phi(y)}-1| \\
& \le C t |\log e^{S_n \phi(x) - S_n\phi(y)}| \le C' t |f^n(x) - f^n(y)|^{\frac{\gamma}{1+\gamma}},
\end{split}
\]
where we have used the fact that $\phi \le 0$ and $p(t) \ge 0$.
Thus (P3)
is satisfied with this choice of $\delta$ for all $t \in [0,1]$.

The final point to check is that $\hf_\Delta$ is transitive and aperiodic on the partition
$\{ \Delta_{\ell,j} \}$.  This is implied by the following.

\begin{lemma}
\label{lem:nonswallow}
Suppose that $H$ is non-swallowing.  Then for each base $\Delta_{0,i}$, there
exists $N= N(i)$ such that $\f^n(\Delta_{0,i}) \supseteq \Delta_0$ for all $n
\ge N$. 
\end{lemma}

\begin{proof} 
Transitivity of $\hf_\Delta$
is obvious by the non-swallowing assumption (1) on $\Q$ and $\Delta_0 = \cup_{Q \in \Q} Q$.  It remains
to verify aperiodicity.  

By property (2) of the definition of non-swallowing, there exists $i_0 \in \N$ such that
$\Delta_{0,i_0} \supset (1/2, 1/2 + \delta)$ for some $\delta >0$.  
Let $\Delta_{0,j} \subseteq \hF(\Delta_{0,i_0})$.  
Since $\Delta_{0,i_0}$ is recurrent, there exists $n_0 \in \N$ such that
$\hf^{n_0} (\Delta_0,j) \supseteq   \Delta_{0,i_0}$.  Now due to the renewal structure and the
fact that $\Delta_{0,i_0}$ contains all $Y_k$ for $k$ greater than some $k_0$, if
$\hf^{n_1}(\Delta_{0,i_0}) \supseteq \Delta_{0,j}$, then also 
$\hf^{n_1+1}(\Delta_{0,i_0}) \supseteq \Delta_{0,j}$.  So we have both
$\hf^{n_0 + n_1}(\Delta_{0,i_0}) \supseteq \Delta_{0,i_0}$ and
$\hf^{n_0 + n_1+1}(\Delta_{0,i_0}) \supseteq \Delta_{0,i_0}$.
Thus $\hf_\Delta$ is aperiodic.
\end{proof}

With properties (P1)-(P4) verified, we conclude by the results of the previous
section that $\hLp_\Delta$ has a spectral gap on $\B$ and the conclusions of 
Theorem~\ref{thm:exp conv} apply.

The last step in the proof of Theorem~\ref{thm:rate gap} is to show that we can project the
results of Theorem~\ref{thm:exp conv} to our open system $(\hf, m_t; H)$.
For this we need the following proposition, which is an adaption of 
\cite[Prop. 4.2]{BruDemMel10}.

\begin{proposition}
\label{prop:lift project}
Recall the projection $P_{\pi,t}$ defined by \eqref{eq:P pi} with respect to the measure 
$m_t$.    Let $C^q(\hI)$ be the set of H\"older continuous functions with exponent
$q$ supported on $\hI$.  Then $Lip(J_t\pi|_{\Delta_\ell}) < \infty$ for each $\ell$
and $C^q(\hI) \subset P_{\pi,t} \B_0$ for all $q \ge \delta/\log 2$.
\end{proposition}
\begin{proof}
Due to conformality and the definition of $\bm_t$, for $x \in \Delta_\ell$, $x = f_\Delta^\ell(y)$,
we have
\[
J_t \pi(x) = \frac{dm_t(\pi x)}{d\bm_t(x)} = \frac{dm_t(\pi(f^\ell_\Delta y))}{d\bm_t(y)}
= \frac{dm_t(f^\ell(\pi y))}{dm_t(\pi y)} = e^{-tS_\ell \phi(\pi y) + \ell p(t)} .
\]
Now by the proof of Lemma~\ref{lem:smooth lift}, $J_t\pi$ is Lipschitz in the metric
$d_\delta$ with Lipschitz constant depending only on the level $\ell$ and the distortion
constant from (D2).

Now let $\psi \in C^q(\hI)$ for some $q \ge \delta/\log 2$.
Recall from the proof of Lemma~\ref{lem:base tower decay}
the index set $\mathcal{K}$ of pairs $(\ell,j)$ such that $\ell \le L$,
$\pi(\cup_{(\ell,j) \in \mathcal{K}} \dlj) = \pi(\hDelta)$, and $\pi(\dlj) \cap \pi(\Delta_{\ell',j'}) = \emptyset$
for all pairs $(\ell,j) \neq (\ell',j')$ in $\mathcal{K}$.

For $(\ell,j) \in \mathcal{K}$ and $x \in \dlj$, define $\tilde \psi(x) = \psi \circ \pi(x) J_t\pi(x)$.
Define $\tilde \psi \equiv 0$ elsewhere on $\Delta$.  Then it follows from fact that
$J_t\pi$ is Lipschitz on $\cup_{\ell \le L} \Delta_\ell$ and Lemma~\ref{lem:smooth lift}
that $\tilde \psi \in \B_0$.  Also, since the images $\pi(\dlj)$ and $\pi(\Delta_{\ell',j'})$
are disjoint
for all pairs $(\ell,j) \neq (\ell',j')$ in $\mathcal{K}$, we have $P_{\pi,t} \tilde \psi = \psi$.
\end{proof}

We proceed to prove the items of Theorem~\ref{thm:rate gap}.
Fix $q>0$ as in the statement of Theorem~\ref{thm:rate gap} and choose
$\delta > 0$, such that $\delta \le \min \{ \frac{q}{\log 2}, \frac{\gamma \log 2}{1 + \gamma} \}$
so that the conditions of Lemma~\ref{lem:smooth lift} and Proposition~\ref{prop:lift project}
are satisfied as well as (P3).

(1) The characterization of $\lambda_t  < 1$ as the spectral radius and the existence of a spectral gap
for $\hLp_{\Delta}$ follow immediately from 
Lemma~\ref{lem:exponential tail} and Theorem~\ref{thm:exp conv}.

It follows from the definitions of $\hLp_{t \phi^H - p(t)}$ and 
$\hLp_{\vf^H_\Delta}$
that given $\psi \in C^q(\pi(\hDelta))$, and $\tilde \psi$
as defined above in the proof of Proposition~\ref{prop:lift project}, that
\begin{equation}
\label{eq:evolve}
\hLp^n_{t \phi^H - p(t)} \psi = P_{\pi, t} \hLp^n_{\vf^H_\Delta} \tilde \psi, \qquad \forall n \ge 0 .
\end{equation}
This relation also holds if $\psi$ is supported on $H \cup \pi(\hDelta)$ since the part
of $\psi$ on $H$ is deleted in one step.

Define $g_t = P_{\pi,t} \bar{g}_t$.   Since $\bar{g}_t$ is bounded away from 0 on $\hDelta$,
it follows using the index set $\mathcal{K}$ from the proof of Lemma~\ref{lem:base tower decay}
that $g_t$ is bounded away from 0 on $\I$.  Moreover, by \eqref{eq:evolve},
we have $\hLp_t g_t = P_{\pi,t} \hLp_{\vf_\Delta^H} \bar{g}_t = \lambda_t g_t$.

(2), (3) and (4).   Define $\nu_t = \pi_* \bar{\nu}_t$. Now since
\[
\int_{\hI} \psi \, dm_t = \int_{\hDelta} \tilde \psi \, d\bm_t,
\]
the characterization of $\nu_t(\psi)$ as the limit of $\lambda_t^{-n} \int_{\hI^n} \psi g_t \, dm_t$
follows from \eqref{eq:evolve} 
and Theorem~\ref{thm:exp conv}(c).

With this definition, $\nu_t$ satisfies 
\[
\log \lambda_t = h_{\nu_t}(f) + \int t\phi^H \, d\nu_t - p(t),
\]
due to Theorem~\ref{thm:exp conv}(b) and (c).  To project this relation
from $\Delta$ to $I$, we use Lemma~\ref{lem:exponential tail}, the definition of $\vf^H$ and
the fact that $\pi : \Delta \to I$ is at most countable-to-one to deduce
$h_{\nu_t}(f) = h_{\bar{\nu}_t}(f_\Delta)$.

Indeed, $\nu_t$ attains the supremum of pressures over all measures $\eta$ that lift to
$\Delta$.  However, due to Proposition~\ref{prop:volume}, we conclude that the supremum
of measures that lift to $\Delta$ is in fact the supremum of all ergodic, invariant measures
supported on $\hI^\infty$.  Thus $\log \lambda_t = p^H(t) - p(t)$, completing
the proof of item (2).  Moreover, it follows immediately that $\nu_t$ is the
equilibrium state for the potential $t \phi^H - p^H(t)$.

To complete the proof of items (3) and (4), we need only construct the conformal measure $m^H_t$.
To this end, define,
\[
m_t^H(\psi) = \lim_{n \to \infty} \lambda_t^{-n} \int_{\hI^n} \psi \, dm_t .
\]
The limit exists again using \eqref{eq:evolve} and Theorem~\ref{thm:exp conv}(d).
It follows that $m_t^H$ is supported on $\hI^\infty$, and indeed that $\nu_t = g_t m_t^H$,
completing item (4).

To see that $m_t^H$ is conformal, note that for a small interval $A \subset I$ centered at
a $m_t^H$-typical point $x$, we have
\[
\begin{split}
\frac{m_t^H(A)}{m_t^H(f(A))}
& = \lim_{n \to \infty} \frac{\lambda_t^{-n-1} \int_{\hI^{n+1}} 1_A \, dm_t}{\lambda_t^{-n} \int_{\hI^n} 1_{f(A)} \, dm_t} 
= \lim_{n \to \infty} \frac{\lambda_t^{-n-1} \int_{\hI^n} \hLp_{t\phi^H - p(t)}( 1_A ) \, dm_t}{\lambda_t^{-n} \int_{\hI^n} 1_{f(A)} \, dm_t} \\
& = \lim_{n \to \infty} \frac{\lambda_t^{-1}  \int_{\hI^n} 1_{f(A)} \cdot e^{t\phi^H \circ f_A^{-1} - p(t)} \, dm_t}{ \int_{\hI^n} 1_{f(A)} \, dm_t},
\end{split}
\] 
where $f_A = f|_A$ is injective on $A$.  Taking the limit as $A \to \{ x\}$, we have
\[
\frac{m_t^H(x)}{m_t^H(f(x))} = \lambda_t^{-1} e^{t\phi^H(x) - p(t)} = e^{t\phi^H(x)-p^H(t)},
\]
using item (2), which proves item (3).

(5) The characterization of $\mu^H_t := g_t m_t$ as a conditionally invariant measure and
a limiting distribution follows immediately from Theorem~\ref{thm:exp conv}(d), again
using \eqref{eq:evolve} in addition to Proposition~\ref{prop:lift project}, which allows us to
lift any $\psi \in C^q(\hI)$
to the function space $\B_0$ on $\hDelta$.  The convergence extends also to $\psi$ supported
fully on $I$ since in one iterate, $\hLp_{t\phi^H - p(t)} \psi$ is supported on $\hI$ so the definition of $\psi$
on $I \setminus \hI$ is irrelevant to the value of the limit.


\subsection{Proof of Corollary~\ref{cor:variational}}
First note that the statement $p^H(t)>0$ if and only if $t < t^H$ 
is simply the definition of $t^H$ together with Proposition~\ref{prop:tH}(a).

We next prove $-\log \blambda_t < p(t)$ if and only if $p^H(t) > 0$.
Recall that by Proposition~\ref{prop:volume}, we have
\[
\log \overline \lambda_t \ge \log \underline \lambda_t \ge p^H(t) - p(t) .
\]

Assume $p^H(t)>0$.  Then by the above inequality,  $\log \overline \lambda_t > - p(t)$
 and we are
in the setting of Theorem~\ref{thm:rate gap}: the associated transfer operator on the tower
has a spectral gap and in particular, the escape rate exists, $\lambda_t = \overline \lambda_t 
= \underline \lambda_t$ and the variational principle holds,
\[
\log \lambda_t = p^H(t) - p(t) .
\]
On the other hand, assume $\log \overline \lambda_t > - p(t)$.  Then again we are in the
setting of Theorem~\ref{thm:rate gap} and the same set of results holds, including the variational
principle.  This implies in particular, that $p^H(t) = p(t) + \log \lambda_t > 0$.

Finally, we show that $\lambda_t$ exists and the variational principle holds in all cases.  
The only case that remains to be addressed 
is when $p^H(t)=0$.  In this case, Proposition~\ref{prop:volume}
implies
\[
\log \overline \lambda_t \ge \log \underline \lambda_t \ge -p(t) .
\]  
In fact, all inequalities must be equalities otherwise $\log \overline \lambda_t > -p(t)$
and again we are in the case of the spectral gap which forces $p^H(t) >0$
by Theorem~\ref{thm:rate gap}, contrary to our assumption.


\subsection{Proof of Theorem~\ref{thm:limit}}

Let $z \in I$ and let $H_\ve \supset \{ z \}$, $\ve \in (0,\ve_0)$ be a nested sequence of intervals
that are nonswallowing holes.  Our first step is to show that the sequence of
conditionally invariant densities $g^{H_\ve}_t$ given by item (2) of Theorem~\ref{thm:rate gap}
enjoy some uniform
regularity in $\ve$.  To this end we define the {\em variation} 
of a function $\psi$ on an interval $J = [c,d]$ by,
\begin{equation}
\label{eq:var def}
\bigvee_J \psi := \sup \sum_{i=1}^n |\psi(x_i) - \psi(x_{i-1})|
\end{equation}
where the supremum is taken over all finite
collections of points $\{ x_i \}_{i=0}^n$ such that $c = x_0 < x_1 < \cdots < x_n = d$.
$\psi$ is said to be of \emph{bounded variation} on $J$ if $\bigvee_J \psi < \infty$.  We call the
set of such functions $BV(J)$.

Before we state our estimate on the variation of our densities, observe that since $g_t^{H_\ve} = P_{\pi,t} \bar{g}_t^{H_\ve}$, the fact that 
$\inf_{\Delta} \bar{g}_t^{H_\ve} >0$ and the existence of the set of partition elements
$\mathcal{K}$ from  Lemma~\ref{lem:base tower decay} imply that
\begin{equation}
\label{eq:g inf}
\inf_{I \setminus H_\ve} g_t^{H_\ve} =: c_\ve > 0 .
\end{equation}

\begin{lemma}
\label{lem:var g}
There exists $\ve_1 > 0$ and a constant $C_t >0$ such that 
$\bigvee_I g_t^{H_\ve} \le C_t$ for all $\ve \in [0, \ve_1]$.
\end{lemma}

\begin{proof}
In what follows, for brevity, we denote $\hLp_{t\phi^{H_\ve} - p(t)}$ by $\hLp_t$.  
Let $\{ L^n_{j, \ve} \}_j$ denote the images under $\hf^n$ of the finitely many intervals 
$\{ K^n_{j, \ve} = (a^n_j, b^n_j) \}_j$ of monotonicity for $\hf^n$.  
Let $\xi^n_j$ denote the inverse branch of
$\hf^n$ on $L^n_{j,\ve}$.
Since $e^{t S_n \phi - n p(t)} \le e^{-n p(t)}$, by standard
estimates (see also the proof Lemma~\ref{lem:LY}), we have for $\psi \in BV(I)$ and $n \ge 0$,
\[
\begin{split}
\bigvee_I \hLp_t^n \psi & \le \sum_j \bigvee_{L^n_{j,\ve}} (\psi e^{t S_n \phi - n p(t)} ) \circ \xi^n_j
+ \sum_j |\psi(a^n_j)| e^{t S_n \phi(a^n_j) - n p(t)} + |\psi(b^n_j)| e^{t S_n \phi(b^n_j) - n p(t)}  \\
& \le \sum_j e^{-np(t)} \bigvee_{K^n_{j,\ve}} \psi 
+ \sup_{K^n_{j,\ve}} |\psi| \bigvee_{K^n_{j,\ve}} e^{t S_n\phi - n p(t)} + e^{-n p(t)} \sum_j \Big( \bigvee_{K^n_{j,\ve}} \psi + 2 \inf_{K^n_{j,\ve}} |\psi| \Big) \\
& \le 3 e^{-np(t)} \bigvee_I \psi + \frac{3}{\min_j m_t(K^n_{j, \ve})} \int_{\hI^n_\ve} |\psi| \, dm_t,
\end{split}
\]
where 
$\hI^n_\ve = \cap_{i=0}^n f^{-i}(I \setminus H_\ve)$.
Also, we used the fact that $e^{tS_n\phi - np(t)}$ is monotonic on each interval to bound
\[
\bigvee_{K^n_{j,\ve}} e^{t S_n\phi - n p(t)} \le \sup_{K^n_{j,\ve}} e^{tS_n\phi - np(t)}
\le e^{-np(t)} .
\]
Now using Corollary~\ref{cor:variational} and letting $\lambda_\ve$ denote the largest eigenvalue
of $\hLp_t$, we estimate,
\[
\bigvee_I \lambda_\ve^{-n} \hLp_t^n \psi \le 3 e^{-np^{H_\ve}(t)} \bigvee_I \psi + \frac{3}{\min_j m_t(K^n_{j, \ve})} \lambda_\ve^{-n} \int_{\hI^n_\ve} |\psi| \, dm_t .
\]
Since $t < t^{H_\ve}$, we have $p^{H_\ve}(t) > 0$ so we may choose $n$ sufficiently large that
$3e^{-np^{H_\ve}(t)} =: \rho_{\ve} < 1$.  Then iterating the above relation, we estimate for all $k \ge 0$,
\begin{equation}
\label{eq:bootstrap}
\bigvee_I \lambda_\ve^{-kn} \hLp_t^{kn} \psi \le \rho_\ve^k \bigvee_I \psi + \frac{3}{\min_j m_t(K^n_{j, \ve})} \sum_{j=1}^k \rho_\ve^j \lambda_\ve^{-(k-j)n} \int_{\hI^{(k-j)n}_\ve} |\psi| \, dm_t .
\end{equation}

Since $\lambda_\ve^{-kn} \hLp_t^{kn} 1 \to c_1 g_t^{H_\ve}$ as $k \to \infty$ for some $c_1>0$
by Theorem~\ref{thm:rate gap}, we will apply \eqref{eq:bootstrap} to $\psi \equiv 1$ to conclude that
$g_t^{H_\ve} \in BV(I)$.  First note that by \eqref{eq:g inf},
\[
\lambda_\ve^{-k} \int_{\hI^k_\ve} 1 \, dm_t \le c_\ve^{-1} \lambda_\ve^{-k} \int_{\hI_\ve^k} g_t^{H_\ve} \, dm_t = c_\ve^{-1},
\]
for all $k \ge 0$ using the conditional invariance of $g_t^{H_\ve}$, i.e. $\int_{\hI_\ve^k} g_t^{H_\ve} \, dm_t = \lambda_\ve^k$.  
Using this together with \eqref{eq:bootstrap} yields,
\[
\bigvee_I \lambda_\ve^{-kn} \hLp_t^{kn} 1 \le  \frac{3}{\min_j m_t(K^n_{j, \ve})} \frac{c_\ve^{-1}}{1- \rho_\ve} .
\]
Thus $\bigvee_I g_t^{H_\ve} \le \frac{3}{c_1 c_\ve (1-\rho_\ve) \min_j m_t(K^n_{j, \ve})}$, so that
$g_t^{H_\ve} \in BV(I)$.

Once we know $g_t^{H_\ve} \in BV(I)$, we may apply \eqref{eq:bootstrap} once more
with $\psi = g_t^{H_\ve}$ to obtain,
\[
\bigvee_I g_t^{H_\ve} = \bigvee_I \lambda_\ve^{-kn} \hLp_t^{kn} g_t^{H_\ve} 
\le \rho_\ve^k \bigvee_I g_t^{H_\ve} + \frac{3}{\min_j m_t(K^n_{j, \ve})} \sum_{j=1}^k \rho_\ve^j ,
\]
and letting $k \to \infty$, we conclude
\[
\bigvee_I g_t^{H_\ve} \le \frac{3}{(1- \rho_\ve) \min_j m_t(K^n_{j, \ve})} .
\]
Finally, we may choose this constant to be independent of $\ve$ for $\ve$ sufficiently small.  This is
because $\rho_\ve \le \rho_{\ve_0} < 1$ by monotonicity of $p^{H_\ve}(t)$ in $\ve$.  Also,
since the sequence of holes is nested by assumption and $\hf^n$ has finitely many branches, 
$\min_j m_t(K^n_{j,\ve})$ can only increase for $\ve$ sufficiently small.
\end{proof}

Now consider the sequence of measures
$\mu_t^{H_\ve} = g_t^{H_\ve} m_t$ for $\ve >0$.  
Define the BV norm, $\| g_t^{H_\ve} \|_{BV} = \bigvee_I g_t^{H_\ve} + | g_t^{H_\ve} |_{L^1(m_t)}$.
Since $\| g_t^{H_\ve} \|_{BV} \le C_t + 1$ for all $\ve < \ve_1$, 
and BV is compact in $L^1(m_t)$, any limit point of the sequence must be absolutely continuous
with respect to $m_t$ and in fact, have a density in $BV(I)$.

Fix a subsequence $\{ \ve_n \}_{n \in \N}$ such that $\{ g_t^{H_{\ve_n}} \}_{n \in \N}$ 
converges in $L^1(m_t)$ to a density $h_t \in BV(I)$ with $\| h_t \|_{BV} \le C_t +1$.  
Let $\mu_\infty = h_t m_t$.

We claim $\mu_\infty$ must be invariant as well, making it the unique invariant measure
$\mu_t$ for $f$ absolutely continuous with respect to $m_t$.
To see this, recall the following characterization of the spectral radius,
\begin{equation}
\label{eq:complement}
\lambda_t^{H_\ve} = \int_{\hI} \hLp_t ( g_t^{H_\ve}) \, dm_t
= \int_{\hI^1} g_t^{H_\ve} \, dm_t = 1 - \int_{\hI \setminus \hI^1} g_t^{H_\ve} \, dm_t .
\end{equation}
Due to the uniform integrability of $g_t^{H_\ve}$ given by 
the proof of Lemma~\ref{lem:var g}, it follows that $\lambda_t^{H_\ve} \to 1$ as
$\ve \to 0$.

Now,
\[
|\Lp_t h_t - h_t |_1  \le |\Lp_t h_t - \lambda_{\ve_n}^{-1} \hLp_t g_t^{H_{\ve_n}}|_1 + |g_t^{H_{\ve_n}} - h_t|_1 
 \le |1- \lambda_{\ve_n}^{-1}| | \Lp_t h_t|_1 + \lambda_{\ve_n}^{-1} |\Lp_t h_t - \hLp_t g_t^{H_{\ve_n}} |_1 + |g_t^{H_{\ve_n}} - h_t|_1
\]
The first and third terms above clearly approach 0 as $n \to \infty$.  We split the second term again,
\[
|\Lp_t h_t - \hLp_t g_t^{H_{\ve_n}} |_1 
\le | \Lp_t h_t - \hLp_t h_t|_1 + |\hLp_t (h_t - g_t^{H_{\ve_n}}) |_1 
\le \int_{I \setminus \hI^1_{\ve_n}} h_t \, dm_t + |h_t - g_t^{H_{\ve_n}}|_1,
\]
and again both terms vanish as $n \to \infty$, using the fact that $| h |_\infty \le C_t +1$.


\section{Proof of Theorem~\ref{thm:limit point}:  The case where the exponential tail equals
 the exponential escape: $t\in [t^H, 1)$}

\label{sec:large t}

Although Theorem~\ref{thm:limit point} applies to all $t \in [0,1)$, 
it provides new information only for $t \in [t^H,1)$.
For ease of notation, we will denote $\hLp_t = \hLp_{t\phi^H - p(t)}$ in this section.
We do not prove that $\hLp_t$ has a spectral gap for $t \in [t^H,1)$, yet we will show
that all limits points of the sequence
$\left\{ \frac{\hLp_t^n 1 }{|\hLp_t^n 1|_{L^1(m_t)}} \right\}_{n \in \N}$ are 
absolutely continuous with respect to $m_t$.
We will also use an averaging technique to construct 
absolutely continuous conditionally invariant probability measures with eigenvalue $\lambda_t$.

Let $g_t^0$ denote the invariant probability density for the transfer operator before the
introduction of the hole, $\Lp_t$.  By Lemma~\ref{lem:var g} applied to $\ve = 0$, $g_t^0 \in BV(I)$
and \eqref{eq:g inf} implies that
\begin{equation}
\label{eq:inf sup}
\frac{\inf_I g_t^0}{\sup_I g_t^0} =: c_1 > 0 .
\end{equation}
Moreover, since each $x \in \hI$ has at least one preimage under $\hf$ by definition of
$H$ being non-swallowing, we have
\begin{equation}
\label{eq:L1 inf}
\inf_{x \in \hI} \hLp_t 1(x) \ge e^{-p(t)} \inf_{x \in \hI} |Df(x)|^{-t} =: c_0 > 0
\end{equation}
First we prove the following lemma.

\begin{lemma}
\label{lem:decay}
There exists $C, \eta > 0$ such that for all $n \in \N$,
$|\hLp_t^n 1|_\infty \le C (1-\eta)^n$.
\end{lemma}
\begin{proof}
Since $H$ is a union of $N_0$-cylinders by assumption, it follows that 
$\# \{ y \in f^{-N_0}(x) \cap H \} \ge 1$ for each $x \in I$.  Thus
\[
\hLp_t^{N_0} g_t^0(x)  = \Lp_t^{N_0} g_t^0(x) - \Lp_t^{N_0} (1_{I \setminus \hI^{N_0}} g_t^0)(x)  
 \le g_t^0(x) [ 1- c_1 \Lp_t^{N_0} (1_{I \setminus \hI^{N_0}})(x) ] 
\; \le \; g_t^0(x) [ 1- c_1 c_0^{N_0}] . 
\]
Iterating this relation, we obtain $\hLp_t^{kN_0} g_t^0(x) \le (1-c_1c_0^{N_0})^k g_t^0(x)$, for each
$k \ge 0$ and $x \in \hI$.  Using again the upper and lower bounds on $g_t^0$,
we complete the proof of the lemma with $1-\eta = (1-c_1c_0^{N_0})^{1/N_0}$.
\end{proof}

Next we address the regularity of $\hNp_t^n 1 := \frac{\hLp_t^n 1}{|\hLp_t^n 1|_1}$ to show that all limit points of the sequence are
absolutely continuous with respect to $m_t$.  Here, $| \cdot |_1$ denotes the $L^1(m_t)$-norm.

To this end,
for $q > 0$, define the log-H\"older constant (which might be $\infty$) of a function $\psi \ge 0$  by 
\[
\| \psi \|_{q, \log} = \sup_{J \in \P_{N_0}} \sup_{x,y \in J} |x-y|^{-q} \, |\log \psi(x)/\psi(y)| .
\]
If $\psi \equiv 0$ on an element of $\P_{N_0}$, we simply set its H\"older constant on that element
equal to 0.  

The following lemma is essentially \cite[Prop. 3.7]{DemFer13}, adapted to the potentials
$t\phi^H - p(t)$.  It is also of note that the constants appearing below are uniform in $H$
and $N_0$.

\begin{lemma}
\label{lem:reg}
Fix $t$  and let
$H$ be an element of $\P_{N_0}$.
For all $n \in \N$, 
\[
\big\| \hNp_t^n 1 \big\|_{q, \log} \le tC_d,
\]
where $C_d$ is from property (D2) and $q = \gamma/(1+\gamma)$. 

\end{lemma}
\begin{proof}

Fix $n\in \N$, $J \in \P_{N_0}$ and for $x,y \in J$, denote by $x_i$ (resp.~$y_i$) the pre-images 
$\f^{-n}(x)$ (resp.~$\f^{-n}(y)$) such that each pair $x_i, y_i$ lies in the same branch of
 $\hf^{-n}$.  Now,
\[
\log \frac{\hLp_{t\phi^{H_{\ve}}-p(t)}^n 1(x)}{\hLp_{t\phi^{H_{\ve}}-p(t)}^n 1(x)}
= \log \frac{\sum_i e^{tS_n\phi(x_i) - np(t)}}{\sum_i e^{tS_n\phi(y_i)- np(t)}}
\le \max_i t \log \frac{Df^n(x_i)}{Df^n(y_i)},
\]
where we have used footnote~\ref{ft:sum}.
The last quantity above is bounded by 
$tC_d|x-y|^{\gamma/(1+\gamma)}$ by the standard distortion estimate (D2) introduced
in the verification of (P3) in the proof of Theorem~\ref{thm:rate gap},  
proving the lemma.  
\end{proof}

By Lemma~\ref{lem:reg}, we have a uniform bound on
$\big\| \hNp_t^n 1 \big\|_{q, \log}$, $n \in \N$.
By \cite[Lemma~3.6]{DemFer13} any sequence of functions with a uniform bound on 
the log-H\"older constant
lies in a compact set in the space of
probability measures on $I$ and any (weak) limit point $\mu$ must be of the form,
\begin{equation}
\label{eq:form}
\mu = s \delta_0 + (1-s)\mu_*, \qquad \mbox{for some $s \in [0,1]$},
\end{equation}
where $\delta_0$ is the point mass at 0 and $d\mu_* = \psi_* dm_t$ for some function $\psi_*$ with
$\| \psi_* \|_{q, \log} \le tC_d$.

Now suppose $\mu$ is the limit of $\hNp_t^{n_j} 1$, 
for some subsequence $(n_j)_{j \in \N}$ such that $s>0$.  It follows that for each $k \in \N$, 
\[
\lim_{j \to \infty} \frac{\hLp_t^{n_j+k}1}{|\hLp_t^{n_j}1|_1} = \hLp^k \mu = s \delta_0 + (1-s) \hLp^k \mu_* .
\]
Thus for each $k \in \N$, we can find $n_j$ large enough that
\[
\frac{|\hLp_t^{n_j+k}1|_1}{|\hLp_t^{n_j}1|_1} \ge \frac s2 .
\]
But choosing $k$ sufficiently large so that $C(1-\eta)^k < \frac s3$, we have 
by Lemma~\ref{lem:decay} for all $n_j \ge 0$,
\[
\frac{|\hLp_t^{n_j+k}1|_1}{|\hLp_t^{n_j}1|_1} 
= \frac{\int_{\hI^{n_j}} \hLp_t^k 1 \, dm_t}{\int_{\hI^{n_j}} 1 \, dm_t} \le \sup_{\hI} \hLp_t^k1 \le C(1-\eta)^k < \frac s3,
\]
which is a contradiction.  Thus for any limit point $\mu$, we must have $s=0$.  This proves the
first part of Theorem~\ref{thm:limit point}.

We next address conditionally invariant measures obtained as averages of 
$\hLp_t^n 1$, suitably renormalized.  Unfortunately,
the naive average $\frac 1n \sum_{i=0}^{n-1} \frac{\hLp_t^i 1}{|\hLp_t^i 1|_1}$ does not yield
a conditionally invariant measure in general since the operation
$\psi \mapsto \frac{\hLp_t \psi}{|\hLp_t \psi|_1}$ is not linear.  Thus we adopt the point of
view taken in  \cite{CMM1} (see also \cite{DemYoung06}).

Define $b_j = j^{t(1+ \frac{1}{\gamma})-1}$, $Z_n = \sum_{j = 1}^{n} \lambda_t^{-j} b_j |\hLp_t^j 1|_1$
and
\[
\psi_n = \frac{1}{Z_n} \sum_{i=1}^{n} \lambda_t^{-j} b_j \hLp_t^j 1 .
\]
By Lemma~\ref{lem:reg} and using the convexity of $\| \cdot \|_{q,\log}$, $\psi_n$ is a sequence of probability densities with
$\| \psi_n \|_{q,\log} \le tC_d$ for each $n$.  Thus by \cite[Lemma 3.6]{DemFer13}, any limit
point of this sequence must again be of the form \eqref{eq:form}.

Suppose $\mu$ is the limit of $\{ \psi_{n_j} \}_{j \in \N}$.  Then since $\mu$ gives 0 weight
to the discontinuities of 
$\hLp_t 1$, we have
\begin{equation}
\label{eq:conditional}
\begin{split}
& \hLp_t \mu  = \lim_{j \to \infty} \frac{1}{Z_{n_j}} \sum_{i=1}^{n_j} \lambda_t^{-i} b_i \hLp_t^{i+1}1 \\
& = \lim_{j \to \infty} \left[ \frac{\lambda_t}{Z_{n_j}} \sum_{i=1}^{n_j} \lambda_t^{-i} b_i \hLp_t^i 1
+ \frac{\lambda_t}{Z_{n_j}} \sum_{i=1}^{n_j} \lambda_t^{-i} b_i \hLp_t^i 1 \left( \frac{b_{i-1}}{b_i} - 1 \right)
 - \frac{b_0 \hLp_t 1}{Z_{n_j}} + \frac{\lambda_t^{-n_j} b_{n_j} \hLp_t^{n_j+1}1}{Z_{n_j}} \right] ,
\end{split}
\end{equation}
where $b_0 := b_1 = 1$.  The first term on the right hand side 
clearly converges to $\lambda_t \mu$, while the
second term converges to 0 (in $L^1(m_t)$) since $\lim_{i \to \infty} \frac{b_{i-1}}{b_i} = 1$.

Next, consider the normalization factors $Z_n$.  By the conformality of $m_t$, we have
\[
\int \hLp_t^n 1 \, dm_t = \int_{\hI^n} 1 \, dm_t \ge \bigcup_{k \ge n} m_t(J_k) 
\ge \sum_{k \ge n+h} C k^{-t(1+\frac 1\gamma)} e^{-kp(t)} \ge C (n+h)^{-t(1+ \frac 1\gamma)+1} e^{-np(t)} .
\]
Since $\lambda_t = e^{-p(t)}$, it follows that
\[
\lambda_t^{-i} b_i |\hLp_t^i 1|_1 \ge C' \qquad \mbox{for all $i \ge 1$},
\]
which implies that the sequence $Z_n$ is increasing and unbounded.  Thus the
third term on the right hand side of \eqref{eq:conditional} converges to 0 (again in $L^1(m_t)$)
as $j \to \infty$.  Finally, the fourth term converges to 0 as well, since the numerator is the
final summand of the subexponentially diverging series in the denominator.

We have shown that $\hLp_t \mu = \lambda_t \mu$ and iterating this relation
yields $\hLp_t^n \mu = \lambda_t^n \mu$, which implies $s=0$ so that $\mu = \psi_* m_t$
is an absolutely continuous conditionally invariant probability measure with eigenvalue $\lambda_t$.
Indeed, due to the regularity of $\psi_n$, which is inherited by $\psi_*$, $\psi_n$ converges 
pointwise uniformly to $\psi_*$ on each element of $\mathcal{P}_{N_0}$.  Thus
the convergence of $\psi_n$ to $\psi_*$ holds in $L^1(m_t)$.


\section{Proof of Theorem~\ref{thm:stable}}
\label{stable proof}

In this section, we will work principally with the domain
$Y := [1/2,1]$ and the induced map 
$F = f^{\tau} : Y \circlearrowleft$,
where $\tau$ is the first return time to $Y$ defined earlier.  
Recall that $F$ has countably many branches created by the preimages of 
the intervals $J_n$, $n \ge 0$.  Let $Y_n \subset [1/2,1]$ be the interval such that
$f(Y_n) = J_n$.  Then $F(Y_n) = Y$ for each $n \ge 0$ so that
$F$ is a full-branched Gibbs-Markov map.

Now fix a family of holes $( H_\ve )_{\ve \le \ve_0}$ satisfying assumption {\bf (H)}
of Section~\ref{main results}.  Unlike in previous sections, the holes $H_\ve$ are not required to
be elements of a Markov partition for $f$.

Since we are interested in the limit as $H_\ve$ shrinks to a point, and thus $t^{H_\eps}$ increases to 1,
we will make a standing assumption throughout this section that 
\begin{equation}
\label{eq:tH}
t^{H_\ve} > \frac{2\gamma}{1+\gamma} .
\end{equation}

Our goal is to construct an invariant measure on the survivor set compatible with
the punctured potential $t\phi^{H_\ve} - p^{H_\ve}(t)$ and then show that
this sequence of singular measures converges weakly to the absolutely continuous (with respect to
$m_t$)
equilibrium state $\mu_t$ for the closed system as $\ve \to 0$.  

This program will be carried out in several steps.  
We first derive a uniform bound on the essential spectral radius of the transfer operator associated with the induced
punctured potential $t\Phi^{H_\ve} - \tau p^{H_\ve}(t)$  
(Section~\ref{induced}) and then show that for small holes it has a spectral gap 
(Section~\ref{ssec:pert}).  
In Section~\ref{S invariant} we use the spectral gap to construct an invariant measure
for $\hF$ supported on $Y$.  This measure projects to an invariant measure
for $\hf$ on $\hI^\infty$ and items (1)-(5) of Theorem~\ref{thm:stable} are proved in Sections~\ref{T invariant} and \ref{ssec:analytic}.

Since we normalize the induced potential $t\Phi^{H_\ve}$
by the punctured pressure $p^{H_\ve}(t)$ rather than $p(t)$, the measure $m_t$ is no longer
a conformal measure for this potential.
In order to proceed, we first prove the existence of a conformal measure
for this potential.  

Recall the definition of the variation of a function $\psi$ on an interval $J$
defined by \eqref{eq:var def}.

\begin{lemma}
\label{lem:conformal}
For $t < 1$, let 
$$P_{t,\ve} = P(t\Phi - \tau p^{H_\ve}(t)).$$  Then $P_{t,\ve} >0$ and 
there exists a $(t\Phi - \tau p^{H_\ve}(t) - P_{t,\ve})$-conformal measure $\tm_{t,H_\ve}$ 
for $F$ on $Y$,
which has no atoms.
\end{lemma}
\begin{proof}
First note that $P(t\Phi - \tau p(t)) = 0$.  So using the fact that $\tau \ge 1$, we have
\[
0 = P(t\Phi - \tau p(t) + \tau p^{H_\ve}(t) - \tau p^{H_\ve}(t))
\le p^{H_\ve}(t) - p(t) + P(t\Phi - \tau p^{H_\ve}(t)).
\]
If $H_\ve$ is a Markov hole, then
since $p^{H_\ve}(t) - p(t) = \log \lambda_t < 0$ by Corollary~\ref{cor:variational}, 
we conclude that $P_{t,\ve} >0$ as required.
On the other hand, if $H_\ve$ is not a Markov hole, we can always find $H_{\ve'} \subset H_\ve$
such that $H_{\ve'}$ is a Markov hole.  Then the above argument implies $P_{t, \ve'} >0$ so that
by monotonicity, $P_{t,\ve} >0$ as well.

In order to prove the existence of $\tm_{t,H_\ve}$, we will check that the potential
$t\Phi - \tau p^{H_\ve}(t)$ is contracting in the sense of \cite{LSV1}.

\noindent
{\em (i) $e^{t\Phi - \tau p^{H_\ve}(t)}$ is of bounded variation.}  Note that 
for each $n$, $\tau$ is constant on $Y_n$, while
$e^{t\Phi}|_{Y_n} = e^{t S_n \phi}$ is monotonically decreasing.
Thus the variation of $e^{t\Phi - \tau p^{H_\ve}(t)}$ is bounded by
\[
\bigvee_Y e^{t\Phi - \tau p^{H_\ve}(t)} \le \sum_{n=0}^\infty \sup_{Y_n} e^{t\Phi - \tau p^{H_\ve}(t)} 
\le \sum_{n=0}^\infty C (n+1)^{t(1+\frac{1}{\gamma})} e^{-p^{H_\ve}(t) (n+1)}
\]
and the series converges for all $t \in [0,1]$ since for $t < t^{H_\ve}$, we have $p^{H_\ve}(t) >0$, while
for $t \ge t^{H_\ve}$, we have $t(1+\frac{1}{\gamma}) > 1$ by our standing assumption \eqref{eq:tH}
that $t^{H_\ve} > \frac{2\gamma}{1+\gamma}$.  Thus
$e^{t\Phi - \tau p^{H_\ve}(t)}$ has bounded variation.

\noindent
{\em (ii) $\sum_{n\ge 0} \sup_{Y_n} e^{t\Phi - \tau p^{H_\ve}(t)} < \infty$.}
This is the same calculation as above.

\noindent
{\em (iii) There exists $n_0 \in \N$ such that }
\[
\sup_Y e^{tS_{n_0} \Phi - \tau^{n_0} p^{H_\ve}(t)} < \inf_{Y} \Lp^{n_0}_{t\Phi - \tau p^{H_\ve}(t)} 1 .
\]
This is trivial since the left hand side decreases exponentially in $n_0$, while the right
hand side is greater than $\inf_Y \Lp^{n_0}_{t\Phi -\tau p(t)} 1$.  Since
$F$ is full-branched Gibbs Markov and the spectral radius of $\Lp_{t\Phi - \tau p(t)}$ on
$C^1(Y)$ is 1, 
$\Lp^{n_0}_{t\Phi -\tau p(t)} 1$ converges uniformly
on $Y$ to a smooth invariant density for $F$ which is bounded below away from 0.

Now that we have verified that the potential $t\Phi - \tau p^{H_\ve} (t)$ is contracting, we may
apply \cite[Theorem 3.1]{LSV1} to conclude the existence of the 
$(t\Phi - \tau p^{H_\ve} (t)-P_{t, \eps})$-conformal
measure $\tm_{t,H_\ve}$. 
\end{proof}

The importance of $\tm_{t,H_\ve}$ is that it enables us to compute 
the escape of mass from $Y$ under $\hF$ via a change of variables.
Define $\Psi_{1,\ve} = t \Phi - \tau p^{H_\ve}(t) - P_{t,\ve}$ and the corresponding punctured
potential by $\Psi_{1, \ve}^{H} = t \Phi^{H_\ve} - \tau p^{H_\ve}(t) - P_{t,\ve}$.
Let $\tilde{H}_\ve$ be the hole in $Y$ induced by $H_\ve$ and let
$\hY_\ve^n = \cap_{i=0}^n F^{-i}(Y \setminus \tilde{H}_\ve)$ denote the set of points in $Y$ which 
do not escape in the first $n$ iterates of the induced map $F$. 
The corresponding transfer operators are denoted
\[
\hLp_{\Psi_{1, \ve}^{H}} \psi = \Lp_{\Psi_{1, \ve}} (1_{\hY_\ve^1} \psi).
\]  
Thus,
\[
\int_Y \hLp_{\Psi_{1, \ve}^{H}}^n \psi \, d\tm_{t,H_\ve} = \int_{\hY^n} \psi \, d\tm_{t,H_\ve} ,
\]
so that the rate of escape with respect to $\tm_{t,H_\ve}$ is governed by the spectral radius of
$\hLp_{\Psi_{1, \ve}^{H}}$.  
We note that for $t=1$, the existence of such a conformal measure is trivial since
$p^{H_\ve}(1) = p(1) = 0$ and $m_1$ remains the conformal measure for $\Lp_\Phi$.


\subsection{Lasota-Yorke Inequalities for the Induced Potentials}
\label{induced}

In the next two sections, we will establish strong spectral properties
for the induced open system $\hF_\ve : \hY_\ve^1 \to Y$ and transfer operators with respect to several
potentials.  Recall $\Psi_{1, \ve}$ and $\Psi_{1, \ve}^H$ defined above.  Similarly, define
$\Psi_2 = t\Phi - \tau p(t)$ and its punctured counterpart $\Psi_2^{H_\ve} = t\Phi^{H_\ve} - \tau p(t)$.
We will study the spectral properties of the corresponding transfer operators on spaces
of functions of bounded variation.

For $\psi \in BV(Y)$, define $\| \psi \|_{BV} = \bigvee_Y \psi + |\psi|_{L^1(m_t)}$.
In this section, we prove the following proposition.

\begin{proposition}
\label{prop:quasi}
Let $\{ H_\ve \}_{\ve \le \ve_0}$ be a nested family of intervals satisfying assumption {\bf (H)}.
There exist constants $C> 0$ and $\sigma_1 < 1$ such that for all 
$\psi \in BV(Y)$, $\ve \in [0, \ve_0]$ and $n \ge 0$,
\[
\| \hLp^n_{\Psi_2^{H_\ve}} \psi \|_{BV} 
 \le C \sigma_1^n \| \psi \|_{BV} + C |\psi|_{L^1(m_t)} .
\]
As a consequence, the essential spectral radius of $\hLp_{\Psi_2^{H_\ve}}$
as an operator on $BV(Y)$ is uniformly bounded by $\sigma_1$
for all $\ve \in [0,\ve_0]$.
\end{proposition}

\begin{remark}
Similarly,  if we define the norm
$\| \psi \|_{BV, H_\ve} = \bigvee_Y \psi + |\psi|_{L^1(\tm_{t,H_\ve})}$, one can also prove
the inequality
\[
\| \hLp^n_{\Psi_{1, \ve}^{H}} \psi \|_{BV, H_\ve} 
 \le C \sigma_1^n \| \psi \|_{BV, H_\ve} + C |\psi|_{L^1(\tm_{t, H_\ve})} 
\]
following closely the proof of Proposition~\ref{prop:quasi}, 
and showing directly that the essential spectral radius of
$\hLp_{\Psi_{1, \ve}^H}$ is again bounded by $\sigma_1$.  
We will not need this estimate, however, so we do not prove it.
\end{remark}

Proposition~\ref{prop:quasi} together with Lemma~\ref{lem:small pert} 
suffice to prove quasi-compactness
of $\hLp_{\Psi_2^{H_\ve}}$ for all $\ve>0$ sufficiently small.

Before proceeding to the proof of Proposition~\ref{prop:quasi}, we make some
observations about our holes $H_\ve$.
Let $\tH_\ve$ be the hole in $Y$ induced by $H_\ve$ and recall the sets
$\{ Y_{i,j} \}_{i, j \ge 0}$ defined in Section~\ref{main results}, which denote the maximal intervals 
on which $\tau^2$ is constant,
$Y_{i,j} = Y_i \cap F^{-1}(Y_j)$. 
Since $F$ is full branched and $H_\ve$ has finitely many components, it follows that
$\hF^2 := F^2|_{\hY_\ve^2}$ enjoys the finite images condition:  The set 
$\{ \hF^2(Y_{i,j}) \}_{i,j \ge 0}$ comprises a finite union of intervals.
The number of these intervals varies depending on the placement of $H_\ve$, but is uniformly
bounded above.

For example, suppose $z = H_0 \subset [0,1/2]$ lies in the interior of one of the
intervals $J_{i_0, j_0} := J_{i_0} \cap f_L^{-n}(Y_{j_0})$, for some $i_0 \ge 1$, $j_0 \ge 0$, where 
$f_L$ is the left branch of $L$.  It follows from condition {\bf (H)} that $H_\ve \subset J_{i_0, j_0}$
for $\ve \le \ve_0$ since otherwise, as $\ve \to 0$, image intervals of arbitrarily short length
would be created.
For $i \in \N$ and $j < i_0$, $j \neq j_0$, we have
$\hF_\ve^2(Y_{i,j}) = Y$, while for $j \ge i_0$, $j \neq j_0$, we have 
$\hF_\ve^2(Y_{i,j}) = Y \setminus f^{i_0}(H_\ve)$, which is a union of two intervals,
$A_1 = [1/2, a_1]$ and $A_2 = [a_2, 1]$.  Note that $f^{i_0}(H) \subset Y_{j_0}$.  It remains
to consider the intervals $Y_{i,j}$ with $j = j_0$.  If $i < i_0$ and $j_0 < i_0$, then again,
$\hF_\ve^2(Y_{i,j_0}) = Y$, while if $j_0 \ge i_0$ then 
$\hF_\ve^2(Y_{i,j_0}) = Y \setminus f^{i_0}(H_\ve)$.
If $i \ge i_0$ and $j < i_0$, then $\hF^2_\ve(Y_{i,j_0}) = Y \setminus F(f^{i_0}H_\ve)$, again
the union of two intervals, $A_3 = [1/2, a_3]$ and $A_4 = [a_4,1]$.
Finally, if $i \ge i_0$ and $j_0 \ge i_0$, then 
$\hF_\ve^2(Y_{i,j_0}) = Y \setminus (f^{i_0}(H_\ve) \cup F(f^{i_0}H_\ve))$, which can be at most
3 intervals, $A_5 = [1/2, a_5]$, $A_6 = [a_6, 1]$ and $A_7 = [a_7, a_8]$.
Other cases for $z$ on the boundary of two consecutive $J_{i,j}$ or in $[1/2, 1]$ are similar.
In all cases, our assumption (H) guarantees that the minimum length of these image 
intervals is uniformly
bounded away from 0 in $\ve$, making it possible to obtain uniform Lasota-Yorke inequalities.

Finally, although $1_{Y \setminus \tH_\ve}$ has infinite variation when $H_\ve \subset [0,1/2]$, 
$\Lp_{\Psi_{1, \ve}} (1_{Y \setminus \tH_\ve})$ has finite variation since 
$\Lp_{\Psi_{1, \ve}} (1_{Y \setminus \tH_\ve})$ is smooth on each of the finitely many images
of $Y$ under $\hF$.  The same holds true for $1_{Y \setminus (\tH_\ve \cup F^{-1}(\tH_\ve))}$
and $\hLp_{\Psi_{1, \ve}}^2 (1_{Y \setminus (\tH_\ve \cup F^{-1}(\tH_\ve))})$ as well as
$\hLp_{\Psi_{2}}^2 (1_{Y \setminus (\tH_\ve \cup F^{-1}(\tH_\ve))})$

\begin{lemma}
\label{lem:LY}
Let $( H_\ve )_{\ve \le \ve_0}$ be a family of holes satisfying assumption (H).  
Then there exist constants $C_3 >0$ and $\sigma <1$, independent of $\ve \le \ve_0$, 
such that for all $\psi \in BV(Y)$,
\[
\| \hLp_{\Psi_2^{H_\ve}}^2 \psi \|_{BV} \le \sigma \| \psi \|_{BV} + C_3 |\psi|_{L^1(m_t)} . 
\]
\end{lemma}

\begin{proof}
Despite the countably many components of $\tH_\ve$, the proof follows the 
standard line.  We include it to show that there is sufficient contraction
uniformly for $t \in [0,1]$ and that
the constants are independent of $\ve$ under assumption {\bf (H)}.  

For convenience, let us reindex the countably many intervals on which $\hF^2$ is smooth
and injective by $Z_n = [a_n,b_n]$, $n \in \N$.
Note that each $Z_n \subset Y_{i,j}$ for some pair $(i,j)$, although some $Y_{i,j}$ will contain
two or at most three $Z_n$ as described earlier.
We denote by $\xi_n$ the inverse of $\hF^2$ restricted to $Z_n$.  
For brevity, we will denote the potential for $\hLp_{\Psi_2^{H_\ve}}^2$ by
$G = \sum_{i=0}^1 (t\Phi^{H_\ve} - \tau p(t)) \circ \hF_\ve^i$.  Then for
$\psi \in BV(Y)$, we write,
\begin{equation}
\label{eq:var split}
\begin{split}
\bigvee_Y \hLp_{\Psi_2^{H_\ve}}^2 \psi & = \bigvee_Y \left( \sum_{n} \psi \circ \xi_n \cdot e^{G \circ \xi_n} \right) 
 \le \sum_{n} \bigvee_{\hF^2(Z_n)} (\psi \circ \xi_n \cdot e^{G \circ \xi_n}) \\
& \; \; \; \; \; \; + \sum_{n} |\psi| \circ \xi_n(a_n) \cdot e^{G \circ \xi_n(a_n)} 
+ |\psi| \circ \xi_n(b_n) \cdot e^{G \circ \xi_n(b_n)} ,
\end{split}
\end{equation}
Note that the sum over the endpoints $a_n = 1/2$ or $b_n =1$ may be omitted so that most intervals
(excepting those of type $A_7$ described above) will have at most one endpoint to consider, and
the full branched ones none at all.
Since we must estimate the worst case, however, we will not keep track of these differences in
our estimates.

For an interval $J$ on which $\xi_n$ is smooth, we estimate
\begin{equation}
\label{eq:var 1}
\begin{split}
\bigvee_J \psi \circ \xi_n \cdot e^{G \circ \xi_n}
& = \bigvee_{\xi_n(J)} \psi \cdot e^G
\le \sup_{\xi_n(J)} e^G \bigvee_{\xi_n(J)} \psi + \sup_{\xi_n(J)} |\psi| \bigvee_{\xi_n(J)} e^G \\
&  \le \sup_{\xi_n(J)} e^G \left( 2 \bigvee_{\xi_n(J)} \psi 
+ \frac{1}{m_t(\xi_n(J))} \int_{\xi_n(J)} |\psi| \, dm_t \right),
\end{split}
\end{equation}
where we have estimated $\bigvee_{\xi_n(J)} e^G \le \sup_{\xi_n(J)} e^G$ since
$G$ is monotonic on each $\xi_n(J)$.

Let $\sigma_n <1$ denote the maximum of $e^G$ restricted to $Z_n$.
Using bounded distortion and the conformality of $m_t$, we have
\begin{equation}
\label{eq:LY dist}
\sup_{\xi_n(J)} e^G \frac{1}{m_t(\xi_n(J))} \le \frac{C_d}{m_t(J)} .
\end{equation}
Combining this with our previous estimates, we bound the variation
of $\psi \circ \xi_n \cdot e^{G \circ \xi_n}$ using \eqref{eq:var 1},
\begin{equation}
\label{eq:var 2}
\bigvee_J \psi \circ \xi_n \cdot e^{G \circ \xi_n} \le 2 \sigma_n \bigvee_{\xi_n(J)} \psi 
+ \frac{C_d}{m_t(J)} \int_{\xi_n(J)} |\psi| \, dm_t .
\end{equation}

It remains to estimate the sum over endpoints in \eqref{eq:var split}.  Now
\[
\begin{split}
|\psi| \circ \xi_n(a_n) \cdot  e^{G \circ \xi_n(a_n)} 
& + |\psi| \circ \xi_n(b_n) \cdot e^{G \circ \xi_n(b_n)} 
\le \sigma_n \Big(  |\psi| \circ \xi_n(a_n)  
+ |\psi| \circ \xi_n(b_n)  \Big) \\
& \le \sigma_n \Big(2 \inf_{Z_n} |\psi| + \bigvee_{Z_n} \psi \Big) 
\le \sigma_n \Big( \frac{2}{m_t(Z_n)} \int_{Z_n} |\psi| \, dm_t + \bigvee_{Z_n} \psi \Big) \\
& \le 2C_d m_t(\hF_\ve^2(Z_n))^{-1} \int_{Z_n} |\psi| \, dm_t + \sigma_n \bigvee_{Z_n} \psi ,
\end{split}
\]
where we have used the bounded distortion estimate \eqref{eq:LY dist}
in the last step.  

Using these estimates together with \eqref{eq:var 2} in \eqref{eq:var split} yields,
\[
\bigvee_Y \hLp_{\Psi_2^{H_\ve}}^2 \psi 
\le \sum_{n} 3 \sigma_n \bigvee_{Z_n} \psi + 3C_d\sup_{n} \left\{ m_t(\hF_\ve^2(Z_n))^{-1} \right\} 
\int_Y |\psi| \, dm_t .
\]

\begin{sublemma}
\label{lem:contract}
There exists $\sigma < 1$ such that for all $t \in [0,1]$ and all $\gamma \in (0,1)$,
  $\max_n 3\sigma_n \le \sigma$.
\end{sublemma}

Sublemma~\ref{lem:contract} completes the proof of Lemma~\ref{lem:LY}, 
using the estimate above, the
fact that $\sum_n \bigvee_{Z_n} \psi \le \bigvee_{Y} \psi$ and assumption {\bf (H)} 
that the lengths (and therefore the $m_t$-measures) of the image intervals are bounded 
below away from 0 by a constant independent of $\ve$ and $t$.
\end{proof}

\begin{proof}[Proof of Sublemma~\ref{lem:contract}]
We want to maximize $e^G$ on $Y$ and show that this maximum is less than 1/3.
For $t=1$, this maximum is 1/4 since $p^H(1) = p(1) = 0$.  From now on, we assume $t<1$.

We begin by estimating the weakest contraction due to 
\[
e^{t\Phi^{H_\ve} - \tau p(t)} = (DF)^{-t} e^{-\tau p(t)} .
\]  
Clearly, this is maximized when $\tau = 1$ and $DF = 2$, i.e. at a point in $Y  \cap f^{-1}(Y)$.
We will maximize this by minimizing its reciprocal,\footnote{To prove Sublemma~\ref{lem:contract}
for the potential $\Psi_{1,\ve}^H$, one must instead minimize,
$2^t e^{p^{H_\ve}(t) + P_{t,\ve}}$.
But observe that 
\[
0  = P(t\Phi - \tau p(t)) = P(t\Phi - \tau p^{H_\ve}(t) + \tau (p^{H_\ve}(t) -  p(t))) 
 \le p^{H_\ve}(t) - p(t) + P(t\Phi - \tau p^{H_\ve}(t))
\]
since $\tau \ge 1$.  Thus $P_{t,\ve} \ge p(t) - p^{H_\ve}(t)$ and so 
$2^t e^{p^{H_\ve}(t) + P_{t,\ve}} \ge e^{g_\gamma(t)}$, and the estimate reduces to the current
estimate for $\Psi_2^{H_\ve}$.}
i.e.
$e^{g_\gamma(t)}$,

where
\[
g_\gamma(t) = t \log 2 + p_\gamma(t).
\]
We have added the subscript $\gamma$
to the expression for the pressure $p(t)$ to emphasize its dependence on $\gamma$. 
We proceed to minimize $g_\gamma(t)$ over $\gamma \in (0,1)$ and $t \in [0,1]$.

Note that $g_\gamma(t)$ is strictly convex with $g_\gamma(0) = \log 2 = g_\gamma(1)$, 
for all $\gamma \in (0,1)$,
so that its minimum occurs at an interior point of $[0,1]$.  We will find the minimum of $g_\gamma$
by finding the point of intersection of two lines that lie below it: lower bounds on the
tangent lines to $g_\gamma$ at $t=0$ and $t=1$.

At $t=1$, $g_\gamma'(1) = \log 2 - \chi_\gamma(\mu_1) >0$ where $\chi_\gamma(\mu_1)$
is the positive Lyapunov exponent with respect to the SRB measure $\mu_1$ for $f = f_\gamma$.
Note that $\chi_\gamma(\mu_1) \downarrow 0$ as $\gamma \uparrow 1$
so we take as a lower bound for this tangent line, the line $u = t \log 2$.
This line lies below $g_\gamma$ and
$u(1) = \log 2 = g_\gamma(1)$.  Although $u(t)$ is not tangent to $g_\gamma(t)$,
it is the limit of tangent lines at $t=1$ as $\gamma \to 1$.  

At $t=0$,
$g_\gamma'(0) = \log 2 - \chi_\gamma(\mu_0) < 0$, where $\mu_0$ is the measure of maximal
entropy for $f$. 
We proceed to derive an upper bound for $\chi_\gamma(\mu_0)$ that is independent of
$\gamma \in (0,1)$.

Let $a_\gamma := f_L^{-1}(1/2)$, where $f_L$ denotes the left branch of $f = f_\gamma$.
We will use the fact that $\mu_0$ gives equal weight to all two-cylinders of the partition
$\{ [0, 1/2), [1/2, 1] \}$, i.e.,
\[
\mu_0([0,a_\gamma]) = \mu_0([a_\gamma, 1/2])
= \mu_0([1/2, 3/4]) = \mu_0([3/4, 1]) = 1/4 .
\]
We want to maximize
\begin{equation}
\label{eq:chi}
\chi_\gamma(\mu_0) = \int_{[0,a_\gamma]} \log |Df_\gamma| \, d\mu_0
+  \int_{[a_\gamma, 1/2]} \log |Df_\gamma| \, d\mu_0
+ \int_{[1/2, 1]} \log |Df_\gamma| \, d\mu_0 .
\end{equation}
The last integral above simply equals $\frac 12 \log 2$.  
For the first two integrals, notice that $Df_\gamma(x)$ is strictly increasing for $x \in [0,1/2]$,
so that for all $x \in [a_\gamma, 1/2]$ and all $\gamma \in (0,1)$,
\[
Df_\gamma(x) \le Df_\gamma(1/2) = 2 + \gamma \le 3 = Df_1(1/2) .
\]
On the other hand, for $x \in [0, a_\gamma]$, 
\[
Df_\gamma(x) \le Df_\gamma(a_\gamma) = 1 + (1+\gamma) (2a_\gamma)^\gamma .
\]

We claim that this expression is increasing in $\gamma$ and so is maximized when $\gamma=1$.

\begin{claim}
\label{claim:gamma}
$\sup_{\gamma \in (0,1)} Df_\gamma(a_\gamma) = Df_1(a_1) = \sqrt{5}$.
\end{claim}

Postponing the proof of the claim and
applying these observations to \eqref{eq:chi}, we have the following upper bound
for $\chi_\gamma(\mu_0)$,
\begin{equation}
\label{eq:chi bound}
\sup_{\gamma \in (0,1)} \chi_\gamma(\mu_0) \le \frac{\log \sqrt{5}}{4} + \frac{\log 3}{4}
+ \frac{\log 2}{2} = \frac{\log (12 \sqrt{5})}{4} .
\end{equation}

Thus the slope $g_\gamma'(0) \ge \log 2 - \frac{\log (12 \sqrt{5})}{4}$ independently of $\gamma$.  
This implies that the minimum
of $g_\gamma(t)$ will be at least as large as the point of intersection between 
$u(t)$ and this lower bound
for the tangent line to $g_\gamma(t)$ at $t=0$.  This occurs when
\[
\begin{split}
t \log 2 & = \log 2 + t(\log 2 - \tfrac 14 \log (12 \sqrt{5})) \\
\implies  \log 2 & = \tfrac t4 \log (12 \sqrt{5}) 
\implies t = \frac{4 \log 2}{\log (12 \sqrt{5})} .
\end{split}
\]
Thus 
\[
\inf_{t \in [0,1]} g_\gamma(t) \ge \frac{ 4 (\log 2)^2}{\log (12 \sqrt{5})},
\]
and so 
\[
e^{G} \le e^{-2g_\gamma} \le e^{-\frac{ 8 (\log 2)^2}{\log (12 \sqrt{5})}} 
< \frac{1}{3.216} ,
\]
for all $\gamma \in (0,1)$,
completing the proof of the sublemma.
\end{proof}

\begin{proof}[Proof of the Claim]
Note that $a_\gamma$ by definition satisfies the following relation,
\begin{equation}
\label{eq:a gamma}
f_\gamma(a_\gamma) = a_\gamma + 2^\gamma a_\gamma^{\gamma + 1} = \tfrac 12
\implies (2a_\gamma)^\gamma = \tfrac{1}{2a_\gamma} -1.
\end{equation}
For $\gamma \in (0,1)$, we want to maximize
\[
M(\gamma) := Df_\gamma(a_\gamma) = 1 + (1+\gamma) (2a_\gamma)^\gamma
= 1 + (1+\gamma) (\tfrac{1}{2a_\gamma} -1) = \tfrac{\gamma +1}{2a_\gamma} - \gamma ,
\]
where we have used \eqref{eq:a gamma} to simplify the expression.
Differentiating with respect to $\gamma$ we obtain,
\begin{equation}
\label{eq:M prime}
M'(\gamma) = \frac{2a_\gamma - (1+\gamma) 2a_\gamma'}{4a_\gamma^2} - 1
= \frac{1}{2a_\gamma} \Big( 1 - \frac{(1+\gamma)a_\gamma'}{a_\gamma} - 2a_\gamma \Big),
\end{equation}
where $a_\gamma' = \frac{da_\gamma}{d\gamma} >0$.  In order to eliminate $a_\gamma'$,
we differentiate \eqref{eq:a gamma} with respect to $\gamma$ to obtain,
\begin{equation}
\label{eq:a prime}
a_\gamma'( 1 + (2a_\gamma)^\gamma) 
+ a_\gamma (2a_\gamma)^\gamma [\log (2a_\gamma) + \gamma \tfrac{a_\gamma'}{a_\gamma}] =0 
\; \; \implies \; \;
\frac{a_\gamma'}{a_\gamma} = \frac{- \log (2a_\gamma)}{\frac{1}{(2a_\gamma)^\gamma} + 1 + \gamma} .
\end{equation}
Substituting this expression into \eqref{eq:M prime} yields,
\[
M'(\gamma) = \frac{1}{2a_\gamma} \Big( 1 + \frac{(1+\gamma) \log (2a_\gamma)}{1 + \gamma+ \frac{1}{(2a_\gamma)^\gamma}} - 2a_\gamma \Big)
\geq \frac{1}{2a_\gamma} \Big( 1 + \frac{(1+\gamma) \log (2a_\gamma)}{2 + \gamma} - 2a_\gamma \Big),
\]
where we have used the fact that $(2a_\gamma)^\gamma \le 1$ and $\log (2a_\gamma)<0$
to obtain the lower bound for $M'(\gamma)$.  To show that $M'(\gamma) > 0$, it suffices to
show that the expression 
\[
h(\gamma) := 1 + \frac{(1+\gamma) \log (2a_\gamma)}{2 + \gamma} - 2a_\gamma
\]
remains positive for $\gamma \in (0,1)$.  Differentiating again, we obtain
\[
\begin{split}
h'(\gamma) & = \frac{(2+\gamma)[\log (2a_\gamma) + (1+\gamma)\frac{a_\gamma'}{a_\gamma}]
- (1+\gamma) \log(2a_\gamma) }{(2+\gamma)^2} - 2 a_\gamma' \\
& = \frac{\log (2a_\gamma) + (1+\gamma)(2+\gamma) \frac{a_\gamma'}{a_\gamma} 
- 2a_\gamma (2+\gamma)^2 \frac{a_\gamma'}{a_\gamma}}{(2+\gamma)^2} \\
& = \frac{- \log (2a_\gamma)}{(2+\gamma)^2} 
\left[ -1 + \frac{(1+\gamma)(2+\gamma) - 2a_\gamma(2+\gamma)^2}{1+ \gamma + \frac{1}{(2a_\gamma)^\gamma}} \right],
\end{split}
\]
where we have used \eqref{eq:a prime} in the last line.  Since $- \log (2a_\gamma) >0$, it suffices
to determine the sign of the expression in square brackets above.  Now we use the fact that 
$a_\gamma \ge 1/4$ (attained when $\gamma =0$) and
$(2a_\gamma)^\gamma < 1$ to write,
\[
\begin{split}
-1- \gamma - \frac{1}{(2a_\gamma)^\gamma} & + (1+\gamma)(2+\gamma) - 2a_\gamma(2+\gamma)^2 \\
& < - 2 - \gamma + 2 + 3 \gamma + \gamma^2 - 2 - 2\gamma - \tfrac{\gamma^2}{2} 
= -2 + \tfrac{\gamma^2}{2} \le - \tfrac 32 < 0 .
\end{split}
\]
We conclude that $h'(\gamma) <0$ so that the minimum of $h$ occurs at $h(1)$.  
Since $a_1 = \tfrac{\sqrt{5}-1}{4}$, we have
\[
h(1) = 1 + \tfrac{2}{3} \log (2a_1) - 2a_1 = \tfrac{3 - \sqrt{5}}{2} + \tfrac 23 \log ( \tfrac{\sqrt{5}-1}{2} ) > 0 .
\]
Since $h(\gamma)$ is strictly positive, we conclude that $M'(\gamma)$ is strictly positive and
thus that $M(\gamma)$ attains its maximum at $\gamma =1$.  Now
$M(1) = Df_1(a_1) = \sqrt{5}$, completing the proof of the claim.
\end{proof}

\begin{proof}[Proof of Proposition~\ref{prop:quasi}]
Even without strict contraction, the estimates of
Lemma~\ref{lem:LY} show that $\| \hLp_{\Psi_2^{H_\ve}} \psi \|_{BV} \le C \| \psi \|_{BV}$ 
for any $\psi \in BV(Y)$.
This, together with the fact that $| \hLp_{\Psi_2^{H_\ve}} \psi |_{L^1(m_t)} \le |\psi|_{L^1(m_t)}$ implies
that for any $n \in \N$ and $\psi \in BV(Y)$,
\[
\| \hLp_{\Psi_2^{H_\ve}}^n \psi \|_{BV} \le C (\sigma^{n/2} \| \psi \|_{BV} + \tfrac{C_3}{1-\sigma} |\psi|_{L^1(m_t)} ),
\]
for a uniform constant $C$, independent of $H_\ve$.  This is the standard Lasota-Yorke inequality.  
This inequality, together with
the compactness of the unit ball of $BV(Y)$ in $L^1(m_t)$, implies that the essential spectral
radius of $\hLp_{\Psi_2^{H_\ve}}$
on $BV(Y)$ 
is bounded by $\sigma^{1/2}$.   
\end{proof}


\subsection{Perturbation Results}
\label{ssec:pert}

In this section, we will prove the following result.

\begin{proposition}
\label{prop:spectral gap}
Let $( H_\ve )_{\ve \le \ve_0}$ be a family of nested intervals satisfying {\bf (H)}.  Then for each
$t \in [0,1]$ and $\ve$ sufficiently small, 
$\hLp_{\Psi_{1,\ve}^{H}} = \hLp_{t\Phi^{H_\ve} - \tau p^{H_\ve}(t) - P_{t,\ve}}$ has
a spectral gap on $BV(Y)$ equipped with the $\| \cdot \|_{BV}$ norm.

Indeed, the spectrum of $\hLp_{\Psi_{1, \ve}^{H}}$ outside the disk of radius $\sigma^{1/2}$
is H\"older continuous in $\ve$ and the
spectral projectors vary H\"older continuously in the $| \cdot |_{L^1(m_t)}$ norm.
\end{proposition}

\begin{proof}
Since on the one hand, $m_t$ is not conformal with respect to the potential $\Psi_{1,\ve}^{H}$ 
while on the other,  the
conformal measures $\tm_{t, H_\ve}$ depend on $H_\ve$, we will prove this
proposition in two steps.  First, notice that since $F$ is a full-branched Gibbs-Markov map, 
the unpunctured
operator $\Lp_{\Psi_2} = \Lp_{t\Phi - \tau p(t)}$ enjoys a spectral gap on $BV(Y)$ equipped with the
$\| \cdot \|_{BV}$ norm since the potential $\Psi_2$ is contracting 
in the sense of \cite{LSV1}.   We will show that the punctured transfer operator
$\hLp_{\Psi_2^{H_\ve}} = \hLp_{t\Phi^{H_\ve} - \tau p(t)}$
is a perturbation of $\Lp_{\Psi_2}$ using the framework of \cite{KelLiv99} 
to conclude that this
spectral gap persists for the punctured transfer operator for sufficiently small holes
under assumption {\bf (H)}.  Indeed, it will follow
that the spectral gap enjoyed by $\hLp_{t\Phi^{H_\ve} - \tau p(t)}$ has a 
lower bound that is uniform in $\ve$.  Second,   we will show that the punctured transfer operator
$\hLp_{\Psi_{1, \ve}^{H}} = \hLp_{t\Phi^{H_\ve} - \tau p^{H_\ve}(t) - P_{t,\ve}}$ 
is a perturbation of $\hLp_{t\Phi^{H_\ve} - \tau p(t)}$ in a strong sense in $BV(Y)$.
This will imply that for sufficiently small $\ve$, $\hLp_{\Psi_{1, \ve}^{H}}$ enjoys 
a spectral gap as well.

{\em Step 1.}  In this step, we will prove that the spectra of 
$\hLp_{\Psi_2^{H_\ve}} = \hLp_{t\Phi^{H_\ve} - \tau p(t)}$ and 
$\Lp_{\Psi_2} = \Lp_{t\Phi - \tau p(t)}$ are close 
in the sense of \cite{KelLiv99}.  To this end, for two operators $P_1, P_2$ from $BV(Y)$ to $L^1(m_t)$,
define the following norm.
\[
||| P_1 \psi - P_2 \psi |||
= \sup \{ |P_1 \psi - P_2 \psi |_{L^1(m_t)} : \| \psi \|_{BV} \le 1 \} .
\]
We begin with the following lemma.

\begin{lemma}
\label{lem:small pert}  
Let $H$ be a hole in $I$ and let $\tH$ be the induced hole for the map $F$ in $Y$.  Then
$$||| \Lp_{\Psi_2} - \hLp_{\Psi_2^{H}} ||| \le m_t(\tH \cup F^{-1}(\tH))  .  $$ 
\end{lemma}

\begin{proof}
Let $\psi \in BV(Y)$, $\| \psi \|_{BV} \le 1$.  Then in particular, $|\psi|_\infty \le 1$.  So,
\[
|\Lp_{\Psi_2} \psi - \hLp_{\Psi_2^{H}} \psi |_{L^1(m_t)} 
= \int |\Lp_{\Psi_2} (1_{\tH \cup F^{-1}(\tH)} \psi) |  \, dm_t 
\le \int_{\tH \cup F^{-1}(\tH)} |\psi| \, dm_t \le m_t(\tH \cup F^{-1}(\tH)) .
\]
\end{proof}

For a family of holes satisfying {\bf (H)},
since $m_t(\tH_\ve \cup F^{-1}(\tH_\ve)) \to 0$ as $\ve \to 0$ and using 
Proposition~\ref{prop:quasi} and Lemma~\ref{lem:small pert}, it follows from 
\cite[Corollary 1]{KelLiv99} that the spectrum and spectral projectors corresponding
to eigenvalues outside the disk of radius $\sigma^{1/2}$ vary Holder continuously in
the size of the perturbation.  

Since $\Lp_{\Psi_2}$ has spectral radius 1 and enjoys a spectral gap, let $\bar\beta_{t,0} < 1$ 
denote the magnitude of its second largest eigenvalue.  Letting $\bar\Lambda_{t, \ve}$ and 
$\bar\beta_{t,\ve}$ denote the largest and second largest
eigenvalues of $\hLp_{\Psi_2^{H_\ve}}$, respectively, 
we conclude that both vary continuously in
$\ve$ for $\ve$ sufficiently small.  In particular, we may choose 
$\ve_0$ sufficiently small that $\bar\Lambda_{t, \ve} - \bar\beta_{t,\ve} > (1 - \bar\beta_{t,0})/2$ 
for all $\ve \le \ve_0$, i.e. $\hLp_{\Psi_2^{H_\ve}}$ has a spectral gap.

{\em Step 2.}  
In this step, we will show that the
transfer operators $\hLp_{\Psi_{1, \ve}^{H}} = \hLp_{t\Phi^{H_\ve} - \tau p^{H_\ve}(t) - P_{t,\ve}}$ and 
$\hLp_{\Psi_2^{H_\ve}} = \hLp_{t\Phi^{H_\ve} - \tau p(t)}$ are close in the
$\| \cdot \|_{BV}$ norm.

\begin{lemma}
\label{lem:BV close}
Suppose $H = H_\ve$ belongs to a family of holes satisfying assumption (H).  Then there exists
$C>0$, independent of $\ve$, such that
\[
\| \hLp_{\Psi_{1, \ve}^{H}} - \hLp_{\Psi_2^{H_\ve}} \|_{BV} \le C (p(t) - p^{H_\ve}(t) + P_{t, \ve}) ,
\]
where $\| \psi \|_{BV} = \bigvee_Y \psi + |\psi|_{L^1(m_t)}$.
\end{lemma}

\begin{proof}
Let $Z_n^i = (a^i_n,b^i_n)$, $i = 1, 2$, 
denote the at most two maximal intervals in $Y$ on which $\hF$ is monotonic and continuous
with $\tau|_{Z_n^i} = n$. 
Letting $\xi_n^i$ denote the inverse of $F|_{Z_n^i}$, for $\psi \in BV(Y)$ we follow
\eqref{eq:var 1},
\[
\begin{split}
\bigvee_Y \big( \Lp_{\Psi_{1,\ve}^{H}} \psi - \Lp_{\Psi_2^{H_\ve}} \psi \big)
& \le \sum_{n,i} \bigvee_Y \psi \circ \xi_n^i (e^{\Psi_{1,\ve} \circ \xi_n^i} - e^{\Psi_2 \circ \xi_n^i}) \\
& \; \; \; + \sum_{n,i} |\psi(a_n^i)| |e^{\Psi_{1,\ve} (a_n^i)} - e^{\Psi_2 (a_n^i)}| 
 + |\psi(b_n^i)| |e^{\Psi_{1,\ve} (b_n^i)} - e^{\Psi_2 (b_n^i)}| \\
& \le \sum_{n,i} \sup_{Z_n^i} |e^{\Psi_{1,\ve}} - e^{\Psi_2}| \Big( \bigvee_{Z_n^i} \psi + \sup_{Z_n^i} |\psi| \Big) 
  +  \sum_{n,i} \sup_{Z_n^i} |e^{\Psi_{1,\ve}} - e^{\Psi_2}| \, 2 \sup_{Z_n^i} |\psi|  \\
& \le 4 \| \psi \|_{BV} \sum_{n,i} \sup_{Z_n^i} |e^{\Psi_{1,\ve}} - e^{\Psi_2}|,
\end{split}
\]
where we have used the fact that $e^{\Psi_{1,\ve}} - e^{\Psi_2}$ is monotonic and does not change sign
on each $Z_n^i$ to bound
the variation by the supremum of the function.  Fixing $Z_n^i$, we estimate,
\[
\begin{split}
|e^{\Psi_{1,\ve}} - e^{\Psi_2}| & = e^{t\Phi - \tau p(t)} | 1- e^{(p(t) - p^{H_\ve}(t)) \tau - P_{t,\ve}}| \\
& \le C n^{-t(1+ \frac{1}{\gamma})} e^{-np(t)} (n(p(t) - p^{H_\ve}(t)) + P_{t,\ve}) e^{(p(t) - p^{H_\ve}(t))n + P_{t, \ve}} \\
& \le C n^{-t( 1 + \frac{1}{\gamma}) + 1} e^{-np^{H_\ve}(t)} ( p(t) - p^{H_\ve}(t) + P_{t, \ve}) ,
\end{split}
\]
where in the second line we have used the estimate $|1-e^x| \le x e^x$ for $x \ge 0$ and 
$|1-e^x| \le |x|$ for $x<0$.

Summing over $n$, we see that the sum is bounded uniformly in $\ve$ since $p^{H_\ve}(t) > 0$ 
for $t <t^{H_\ve}$
and since $t^{H_\ve} > 2\gamma/(1+\gamma)$ by assumption of \eqref{eq:tH}, 
we have $\sum_n n^{1-t(1 + \frac{1}{\gamma})} < \infty$ for $t \ge t^{H_\ve}$.  We have also used the
fact that there are at most two $Z_n^i$ per $n \in \N$.

To bound the difference in $L^1(m_t)$ norm for $\psi \in BV(Y)$, 
we use the fact that $m_t$ is conformal with respect to $\Psi_2$ to
write
\[
\begin{split}
\int_Y  |\Lp_{\Psi_{1,\ve}^{H}} \psi - \Lp_{\Psi_2^{H_\ve}} \psi| \, dm_t
& = \int_{\hY^1_\ve} |\psi| |1- e^{(p(t) - p^{H_\ve}(t)) \tau - P_{t, \ve}}| \\
& \le C \| \psi \|_{BV} \sum_{n,i} m_t(Z_n^i) |1 - e^{n(p(t) - p^{H_\ve}(t)) - P_{t,\ve}}|
\end{split}
\]
and note that this is the same estimate as above since due to conformality and the large images
assumption (H), $m_t(Z_n^i)$ is proportional to $e^{t\Phi - \tau p(t)}$.
\end{proof}

Our next lemma shows that in fact the bound obtained in the previous lemma
is continuous in $\ve$.

\begin{lemma}
Let $(H_\eps)_{\eps \le \eps_0}$ be a collection of holes centered at $z\in I$ with $H_\eps\to \{z\}$ as $\eps\to 0$.  Then for $t \in [0,1]$,
$$p^{H_\eps}(t)\to p(t), \; \; \;  t^{H_\ve} \to 1 \; \text{ and } \; P_{t, \eps} \to 0 \; \text{ as } \; \eps \to 0.$$
\label{lem: eps 0 press 0}
\end{lemma}

\begin{proof}
Note that for $t=1$, the statement of the lemma is trivial since
$p^{H_\ve}(1) = p(1) = 0$ and so $P_{1, \ve} = - \log \bar\Lambda_{1,\ve}$ is continuous in $\ve$
by Step 1.  We now focus on $t<1$.  Since the quantities of interest are clearly monotone in the size of the hole, we need only prove the lemma for Markov holes.

By Corollary~\ref{cor:variational}, we have $\log\lambda^{H_\ve}_t = p^{H_\ve}(t) - p(t)$.
By \eqref{eq:complement} and the comment following it, we have
$\lambda^{H_\ve}_t \to 1$ as $\ve \to 0$.  Thus $p^{H_\ve}(t) \to p(t)$ as $\ve \to 0$.

Now fix $t < 1$.  Since $p(t)>0$, by the previous paragraph we may choose $\ve>0$ sufficiently small
such that $p^{H_\ve}(t) > 0$.  By Proposition~\ref{prop:tH}, this implies $t^{H_\ve} > t$ and by
monotonicity, $t^{H_{\ve'}} > t$ for all $\ve' \le \ve$.  Since this is true for each $t<1$, we have
$t^{H_\ve} \to 1$ as $\ve \to 0$.

Finally, consider the rescaled transfer operator $e^{P_{t,\ve}} \Lp_{\Psi_{1, \ve}} = \Lp_{t\Phi - \tau p^{H_\ve}(t)}$ whose spectral radius on $BV(Y)$ is $e^{P_{t,\ve}}$.  Replacing $\Lp_{\Psi_{1,\ve}}$
by this rescaled operator in the statement and proof of Lemma~\ref{lem:BV close} yields,
\[
\| \Lp_{t\Phi - \tau p^{H_\ve}(t)} - \Lp_{t\Phi - \tau p(t)} \|_{BV} \le C (p(t) - p^{H_\ve}(t)) .
\]
Using now that $p^{H_\ve}(t) \to p(t)$ as $\ve \to 0$ and the fact that $\Lp_{t\Phi - \tau p(t)}$
has a spectral gap with leading eigenvalue 1, we conclude using standard perturbation
theory that the leading eigenvalue of $\Lp_{t\Phi - \tau p^{H_\ve}(t)}$ tends to 1 as $\ve \to 0$.
This implies $P_{t,\ve} \to 0$ as required.
\end{proof}

Since $P_{t, \ve} \to 0$ and $p^{H_\ve}(t) \to p(t)$ as $\ve \to 0$,  Lemmas~\ref{lem:BV close} and \ref{lem: eps 0 press 0} imply that as operators on $BV(Y)$, $\hLp_{\Psi_{1,\ve}^{H}}$ and 
$\hLp_{\Psi_2^{H_\ve}}$
are close so that their spectra and spectral projectors vary continuously by standard
perturbation theory (see \cite{kato}).  

Thus for $\ve$ sufficiently small, the largest eigenvalue of 
$\hLp_{\Psi_{1,\ve}^{H}}$, $\Lambda_{t, \ve}$, is close to $\bar\Lambda_{t, \ve}$, while the second largest 
eigenvalue, $\beta_{t,\ve}$ is as close as we like to $\bar\beta_{t, \ve}$ (if it lies outside the disk
of radius $\sigma^{1/2}$).  Since by Step 1, $\bar\Lambda_{t,\ve}$ and $\bar\beta_{t,\ve}$
are uniformly bounded away from one another for all $\ve \le \ve_0$, we may further shrink
$\ve_0$ if necessary
so that $\Lambda_{t, \ve}$ and $\beta_{t, \ve}$ are uniformly bounded away from one another
for all $\ve \le \ve_0$.   Thus 
$\hLp_{\Psi_{1,\ve}^{H}}$ has a spectral gap on $BV(Y)$ 
for all $\ve \le \ve_0$.

 This completes the proof of Proposition~\ref{prop:spectral gap}.
\end{proof}


\subsection{An invariant measure for $\hF_\ve$ on $\hY_\ve^\infty$}
\label{S invariant}

In this section, we fix $t \in [0,1]$ and assume that $\ve_0$ is small enough that for each
$H_\ve$ with $\ve \le \ve_0$, $\hLp_{\Psi_{1,\ve}^{H}} =: \hLp_\ve$ has a spectral gap by
Proposition~\ref{prop:spectral gap}.

Thus for $\ve \in (0, \ve_0]$ there exists a maximal eigenvalue $\Lambda_\ve < 1$ for $\hLp_{\ve}$
and unique $g_\ve \in BV(Y)$ on $\hY_\ve = Y \setminus \tH_\ve$
such that $\hLp_\ve g_\ve = \Lambda_\ve g_\ve$ and $g_\ve \tm_{t, H_\ve}$ defines
a conditionally invariant probability measure for $\hF$ with escape rate $- \log \Lambda_\ve$.  Moreover, 
there exists $C>0$ and $\rho<1$ such that for each
$\psi \in BV(Y)$ and $n \ge 0$, we have
\[
\| \Lambda_\ve^{-n} \hLp_\ve^n \psi - e_\ve(\psi) g_\ve \|_{BV} \le C\| \psi \|_{BV} \rho^n ,
\]
where $e_\ve(\psi)$ is 
determined by the spectral projector $\Pi_\ve$ of $\hLp_\ve$ onto the 
subspace spanned by $g_\ve$:  $e_\ve(\psi) = | \Pi_\ve \psi |_{L^1(\tm_{t,H})}$, due to the
normalization of $g_\ve$ we have chosen.
If $\psi$ is a probability density with respect to $m_t$, then $e_0(\psi) =1$.

We define an invariant measure on the survivor set $\hY_\ve^\infty$ using a well-known construction.
For $\psi \in BV$, define the functional,
\begin{equation}
\label{eq:inv def}
\nu_\ve(\psi) = \lim_{n \to \infty} \Lambda_\ve^{-n} \int_{\hY_\ve^n} \psi g_\ve \, d\tm_{t,H_\ve}
= \lim_{n \to \infty} \Lambda_\ve^{-n} \int_{\hY_\ve} \hLp_\ve^n (\psi g_\ve) \, d\tm_{t,H_\ve}
= e_\ve(\psi g_\ve) ,
\end{equation}
so that the limit is well-defined on $BV$.  Note that $\nu_\ve$ is linear, positive and
$\nu_\ve(\psi) \le |\psi|_\infty$ so that $\nu_\ve$ can be extended to a bounded, positive linear
functional on $C^0(Y)$.
Since, $\nu_\ve(1)=1$, by the Riesz representation theorem, $\nu_\ve$ corresponds 
to a unique Borel probability measure.  From its definition, it is clear that
$\nu_\ve$ is supported on the survivor set $\hY_\ve^\infty$.  

\begin{proposition}
\label{prop:invariant conv}
Let $( H_\ve )_{\ve_0}$ be a family of holes as in Proposition~\ref{prop:spectral gap} such that
that $\ve_0>0$ is small enough that $\hLp_\ve$ has a spectral gap for
each $\ve \le \ve_0$.  Let $\nu_\ve$ be the corresponding invariant measure on the
survivor set defined by \eqref{eq:inv def}.  Then 
\[
\nu_\ve(\psi) \to \nu_0(\psi) \; \; \mbox{as $\ve \to 0$ for each $\psi \in C^0(Y) \cup BV(Y)$},
\]
where $\nu_0$ is the unique invariant measure absolutely continuous with respect to $m_t$ for
$F$, the induced map without the hole.
\end{proposition}

\begin{proof}
Note that for $\ve > 0$, $\nu_\ve$ is singular with respect to $m_t$ for $\ve>0$, 
but that $d\nu_0 = g_0 dm_t$ where
$g_0 \in BV(Y)$ is the unique invariant probability density for $\Lp_{t\Phi - \tau p(t)}$.\footnote{Since
$H_0= \{ z \}$, $p^{H_0}(t) = p(t)$ so that the conformal measure for $\Psi_{1,0}^{H}$ is once 
again $m_t$.}  

Let $\bar{g}_\ve dm_t$, $\bar{g}_\ve \in BV(Y)$, denote the conditionally invariant probability
measure formed using the eigenvector $\bar{g}_\ve$ for $\hLp_{\Psi_2^{H_\ve}}$,
and let $\bar\nu_\ve$ denote the invariant measure on the survivor set $Y_\ve^\infty$
formed via the limit in \eqref{eq:inv def}, but using $\bar g_\ve dm_t$ in place of 
$g_\ve d\tm_{t, H_\ve}$. 

For $\psi \in BV(Y)$, we have
\begin{equation}
\label{eq:conv}
\left|\nu_\ve(\psi) - \nu_0(\psi)\right| \le |\nu_\ve(\psi )- \bar{\nu}_\ve(\psi)| + |\bar{\nu}_\ve(\psi) - \nu_0(\psi)| .
\end{equation}

To estimate the first term above, we write
\begin{equation}
\label{eq:conv split}
|\nu_\ve(\psi) - \bar{\nu}_\ve(\psi)|  \le \left|e_\ve(\psi g_\ve) - \bar e_\ve(\psi g_\ve)\right| 
+ \left| \bar e_\ve(\psi  g_\ve) - \bar e_\ve( \psi \bar{g}_\ve) \right| .
\end{equation}
Let $\Pi_\ve$ and $\bar{\Pi}_\ve$ denote
the projectors onto the eigenspaces corresponding to  the top eigenvalues of 
$\hLp_{\Psi_{1,\ve}^{H}}$ and $\hLp_{\Psi_2^{H_\ve}}$, respectively, so that
\[
\Pi_\ve(\psi g_\ve) = e_\ve(\psi g_\ve) g_\ve \; \; \; \mbox{and} \; \; \; 
\bar{\Pi}_\ve(\psi g_\ve) = \bar{e}_\ve(\psi g_\ve) \bar{g}_\ve .
\]
By Lemma~\ref{lem:BV close}, $\| g_\ve - \bar g_\ve \|_{BV} \le \zeta(\ve)$ for some function $\zeta$
such that $\zeta(\ve) \to 0$ as $\ve \to 0$.  Similarly, $\| \Pi_\ve - \bar \Pi_\ve \|_{BV} \le \zeta(\ve)$.
\[
\begin{split}
\| \Pi_\ve(\psi g_\ve)  & - \bar \Pi_\ve(\psi g_\ve) \|_{BV} 
 \ge \int | \Pi_\ve(\psi g_\ve) - \bar \Pi_\ve(\psi g_\ve) | \, dm_t  \\
& = \left| \int (e_\ve(\psi g_\ve) - \bar e_\ve(\psi g_\ve)) g_\ve \, dm_t 
+ \bar e_\ve(\psi g_\ve) \int (g_\ve - \bar g_\ve) \, dm_t \right|
\end{split}
\]
Since the term on the left of the inequality is of order $\zeta(\ve)$ and the second term on the right
 is of the same order,
then the same must be true of the first term on the right.
It follows then that $e_\ve(\cdot)$ and $\bar e_\ve(\cdot)$ can be
made arbitrarily close on $BV$ functions by choosing $\ve$ small.
Thus the first term in \eqref{eq:conv split} can be made arbitrarily small by choosing $\ve$ small.

For the second term of \eqref{eq:conv split}, we estimate,
\[
 \left| \bar e_\ve(\psi  g_\ve) - \bar e_\ve( \psi \bar{g}_\ve) \right| 
= \left| \int \bar \Pi_\ve (\psi g_\ve - \psi \bar g_\ve) \, dm_t \right|
\le C \| \psi \|_{BV} \| g_\ve - \bar g_\ve \|_{BV},
\]
which can again be made arbitrarily small.  This completes the estimate on the first term of
\eqref{eq:conv}.

To estimate the second term in \eqref{eq:conv}, we use Lemma~\ref{lem:small pert} so that the
spectra and spectral proctors of $\hLp_{\Psi_2^{H_\ve}}$ converge to those of $\Lp_{t\Phi - \tau p(t)}$
in the weaker $L^1(m_t)$ norm and not in $\| \cdot \|_{BV}$.
\[
\begin{split}
\left|\bar{\nu}_\ve(\psi) - \nu_0(\psi)\right| 
& \le \left| \bar e_\ve(\psi \bar g_\ve) - e_0(\psi \bar g_\ve) \right| 
+ \left| e_0(\psi \bar g_\ve) - e_0(\psi g_0) \right| \\
& = \left| \int \bar\Pi_\ve(\psi \bar{g}_\ve) \, dm_t - \int \Pi_0 (\psi \bar g_\ve) \, dm_t \right|
+ \left| \int \psi (\bar g_\ve - g_0) \, dm_t \right| \\
& \le \left| \bar\Pi_\ve (\psi \bar{g}_\ve)  - \Pi_0 (\psi \bar g_\ve) \right|_{L^1(m_t)} 
+ \| \psi \|_{BV} |\bar g_\ve - g_0|_{L^1(m_t)} \\
& \le ||| \bar \Pi_\ve - \Pi_0 ||| \, \| \psi \bar{g}_\ve  \|_{BV}  
+   \| \psi \|_{BV} |\bar g_\ve - g_0|_{L^1(m_t)}
\end{split}
\]
The first term above tends to zero as $\ve$ tends to 0
due to Lemma~\ref{lem:small pert}, while the second
term tends to zero by \cite[Corollary 1]{KelLiv99}.

Putting these estimates together with \eqref{eq:conv split} in \eqref{eq:conv} completes the
proof of convergence of $\nu_\ve$ to $\nu_0$ when integrated against functions in $BV(Y)$.
We extend this convergence to continuous functions
by approximation.  For $\psi_1 \in C^0(Y)$, let $\delta >0$ and
choose a step function $\psi_2 \in BV(Y)$ such that
$|\psi_1 -\psi_2 |_\infty < \delta$.  Then
\[
\begin{split}
|\nu_\ve(\psi_1) - \nu_0(\psi_1)| & \le |\nu_\ve(\psi_1) - \nu_\ve(\psi_2)| + |\nu_\ve(\psi_2) - \nu_0(\psi_2)| + |\nu_0(\psi_2) - \nu_0(\psi_1)| \\
& \le 2\delta + |\nu_\ve(\psi_2) - \nu_0(\psi_2)| ,
\end{split}
\]
and the last term tends to 0 as $\ve \to 0$ since $\psi_2 \in BV(Y)$.  Since $\delta>0$ was
arbitrary, this implies the required convergence on continuous functions.
\end{proof}


\subsection{An invariant measure for $\hf_\ve$ on $\hI_\ve^\infty$}
\label{T invariant}

In this section, we push the
invariant measure $\nu_\ve = \nu_{Y, \ve}$ on $\hY_\ve^\infty$ onto $\hI_\ve^\infty$ to obtain
an invariant measure $\nu_{H_\ve}$ for $\hf_\ve$ which inherits good properties from 
$\nu_{Y,\ve}$, thus completing parts (1) and (5) of Theorem~\ref{thm:stable}.

Define for any Borel set $A \subset I$,
\begin{equation}
\label{eq:project}
\nu_{H_\ve}(A) = \frac{1}{\int \tau \, d\nu_{Y, \ve}} \sum_{k=0}^\infty \sum_{i=0}^{\tau_k-1} \nu_{Y,\ve}(f^{-i}(A) \cap Y_k) , 
\end{equation}
where $\tau_k = \tau|_{Y_k} = k+1$.  This defines an $f$-invariant probability measure $\nu_{H_\ve}$ 
if $\int \tau \, d\nu_{Y,\ve} < \infty$.   The next lemma shows that in fact $\tau$ is uniformly integrable
with respect to $\nu_{Y, \ve}$ for all $\ve$ sufficiently small.

\begin{lemma}
\label{lem:finite}
There exists a constant $C_4 >0$ such that for each $k \ge 0$ and all $\ve$ sufficiently small,
$\nu_{Y,\ve}(Y_k) \le C_4 k^{-t(\frac{1}{\gamma}+1)}e^{-(k+1)p^{H_\ve}(t)}$.
\end{lemma}

\begin{proof}
Fix $k \ge 0$ and let $e^{(t\Phi - \tau p^{H_\ve}(t)-P_{t,\ve})(Y_k)}$ denote the maximum of
$e^{t\Phi - \tau p^{H_\ve}(t) - P_{t, \ve}}$ on $Y_k$.  
Since $F(Y_k) = [1/2,1]$, we have for $n \ge 1$,
\begin{equation}
\label{eq:image}
\begin{split}
\tm_{t,H_\ve}(\hY^n_\ve \cap Y_k)  & \le e^{(t\Phi - \tau p^{H_\ve}(t)- P_{t, \ve})} \tm_{t,H_\ve}(\hY_\ve^{n-1} \cap F(Y_k)) \\
& \le C_5 k^{-t(\frac{1}{\gamma}+1)}e^{-(k+1)p^{H_\ve}(t)}  \tm_{t,H_\ve}(\hY_\ve^{n-1})
\end{split}
\end{equation}
for some uniform constant $C_5$ depending on bounded distortion.
Note that $C_5$ is independent of $\ve$ (and $t$).  Now
\[
\Lambda_\ve^{-n+1} \tm_{t,H_\ve}(\hY_\ve^{n-1}) = \Lambda_\ve^{-n+1} \int \hLp_{\Psi_{1,\ve}^{H}}^{n-1} 1 \, d\tm_{t,H_\ve}
\xrightarrow[n \to \infty]{} e_\ve(1)> 0 .
\]
Therefore, there exists $n_1 = n_1(\ve)>0$ such that $\Lambda_\ve^{-n+1} \tm_{t,H_\ve}(\hY_\ve^{n-1}) 
\le 2 e_\ve(1)$
for all $n > n_1$.  Also, since $e_\ve(1) \to e_0(1) = 1$ as $\ve \to 0$, we have
$e_\ve(1) \le 2$ for all $\ve$ sufficiently small.

Putting these estimates together with \eqref{eq:image} yields,
\[
\begin{split}
\nu_\ve(Y_k) & = \lim_{n \to \infty} \Lambda_\ve^{-n} \int_{\hY^n_\ve} 1_{Y_k} \, g_\ve \, d\tm_{t,H_\ve}
\le \lim_{n \to \infty} \Lambda_\ve^{-n} |g_\ve|_\infty \tm_{t,H_\ve}(\hY_\ve^n \cap Y_k) \\
& \le \lim_{n \to \infty} |g_\ve|_\infty \Lambda_\ve^{-1} C_5 k^{-t(\frac{1}{\gamma}+1)}e^{-(k+1)p^{H_\ve}(t)}
\Lambda_\ve^{-n+1} \tm_{t,H_\ve}(\hY_\ve^{n-1}) \\
& \le 4 \Lambda_\ve^{-1} C_5 |g_\ve|_\infty k^{-t(\frac{1}{\gamma}+1)} e^{-(k+1)p^{H_\ve}(t)}. 
\end{split}
\]
Now since $g_\ve$ lies in a uniform ball in $BV(Y)$ for all $\ve$ sufficiently small,
the lemma is proved since the constants  
are bounded independently of $\ve$ and $k$.
\end{proof}

As a consequence of Lemma~\ref{lem:finite}, we have
\begin{equation}
\label{eq:tau}
\int_{\tau>k_0} \tau \, d\nu_{Y,\ve} = \sum_{k=k_0}^\infty \tau_k \nu_{Y,\ve}(Y_k) 
\le \sum_{k=k_0}^\infty (k+1) C_4 k^{-t(\frac{1}{\gamma}+1)}e^{-(k+1)p^{H_\ve}(t)}
\le \rho_t(k_0) ,
\end{equation}
where $\rho_t(k_0) \to 0$ as $k_0 \to \infty$.  In this last step, we have used the fact that 
$p^{H_\ve}(t) >0$ for $t< t^H$ as well as the assumption that $t^{H_\ve} > 2\gamma/(1+\gamma)$,
to conclude that the tail of the series tends to 0 as $k_0$ increases.

\begin{lemma}
\label{lem:tau}
$\int \tau \, d\nu_{Y,\ve} \to \int \tau \, d\nu_{Y,0}$ as $\ve \to 0$ .
\end{lemma}

\begin{proof}
Fix $\delta>0$.  By \ref{eq:tau}, we may choose $k_0$ so that 
$\int_{\tau>k_0} \tau \, d\nu_{Y,\ve} < \delta$
for all $\ve$ sufficiently small (including $\ve=0$).  Since $\tau$ is constant on each $Y_k$,
we have $\tau \cdot 1_{\{\tau \le k_0\}} \in BV(Y)$.  Thus
\[
\left|\int \tau \, d\nu_{Y,\ve} - \int \tau \, d\nu_{Y,0}\right| \le \left| \int \tau \cdot 1_{\{ \tau \le k_0 \}}  \, d\nu_{Y,\ve} 
- \int \tau \cdot 1_{\{\tau \le k_0 \}} \, d\nu_{Y,0} \right| + 2\delta
\]
and the difference of integrals on the right goes to zero as $\ve \to 0$ by 
Proposition~\ref{prop:invariant conv}.  Since $\delta>0$ was arbitrary, this completes the
proof of the lemma.
\end{proof}

For part (5) of Theorem~\ref{thm:stable}, we need to show that $\nu_{H_\ve} \to \mu_t$ as $\ve \to 0$, where $\mu_t$ is the equilibrium
state (absolutely continuous with respect to $m_t)$ for the unpunctured potential $t\phi - p(t)$.  
This is straightforward with 
Lemmas~\ref{lem:finite} and \ref{lem:tau} in hand since $\mu_t$ is simply the projection
of $\nu_{Y,0}$ to $I$ via the formula analogous to \eqref{eq:project}.  
Fix a function $\psi$ with bounded variation on $I$ and $\delta>0$.
Choose $k_0$ such that $\int_{\tau > k_0} \tau \, d\nu_{Y,\ve} < \delta$.
Then
\[
\begin{split}
|\nu_{H_\ve}(\psi) - \mu_t(\psi) | & \le \left|  \sum_{k=0}^{k_0} \sum_{i=0}^{\tau_k-1}
\frac{\nu_{Y,\ve}(\psi \circ f^i \cdot 1_{Y_k})}{\int \tau \, d\nu_{Y,\ve}} - \frac{\nu_{Y,0}(\psi \circ f^i \cdot 1_{Y_k})}{\int \tau \, d\nu_{Y,0}}  \right| \\
& \; \; \; \; + \| \psi \|_{BV} \sum_{k>k_0} \sum_{i=0}^{\tau_k-1}
\frac{\nu_{Y,\ve}(Y_k)}{\int \tau \, d\nu_{Y,\ve}}  +  \| \psi \|_{BV} \sum_{k>k_0} \sum_{i=0}^{\tau_k-1}
\frac{\nu_{Y,0}(Y_k)}{\int \tau \, d\nu_{Y,0}} .
\end{split}
\]
Now the terms on the second line are bounded by $2\delta \| \psi \|_{BV}$ and each of the finitely many 
differences on the right hand side of the first line tend to 0 with $\ve$ by
Lemma~\ref{lem:tau} and Proposition~\ref{prop:invariant conv}.
Since $\delta>0$ was arbitrary, this proves the required convergence of
$\nu_{H_\ve}$ to $\mu_t$ when integrated against functions of bounded variation.  
To complete the proof of items (1) and (5) of Theorem~\ref{thm:stable}, we extend 
this convergence to continuous
functions by approximation as in the proof of Proposition~\ref{prop:invariant conv}.


\subsubsection{Interpretation of $\nu_{H_\eps}$}

In this section, we prove items (2) and (3) of Theorem~\ref{thm:stable}.  
We restrict to a sequence of holes $(H_\ve)_{\ve \le \ve_0}$ satisfying {\bf (H)} and such that
$\ve_0$ is sufficiently small that $\hLp_{t\Phi^{H_{\ve_0}} - \tau p^{H_{\ve_0}}(t) - P_{t, \ve}}$
has a spectral gap by Proposition~\ref{prop:spectral gap}.
 
Given a Markov
hole $H = H_\ve$, 
set $$\psi_t^H:=t\phi^H-p^H(t)- P(t\Phi^H-\tau p^H(t))\cdot 1_Y.$$
We first note that 
$$\Psi_t^H = S_\tau \psi_t^H =t\Phi^H-\tau p^H(t)- P(t\Phi^H-\tau p^H(t)),$$ 
so $P(\Psi_t^H)=0$. 

It follows from \cite{BruDemMel10} that the measure $\nu_{Y,\ve}$
constructed in Section~\ref{S invariant}
is a Gibbs measure with Gibbs constant 0 for the potential $t\Phi^H - \tau p^{H}(t) - P_{t,\ve} - \log \Lambda_{t,\ve}$.  So by Theorem~\ref{thm:Gibbs}, 
\[
0=P\left(t\Phi^H - \tau p^{H}(t) - P_{t,\ve} - \log \Lambda_{t,\ve}\right)
=P(t\Phi^H - \tau p^{H}(t)) - P(t\Phi-\tau p^H(t)) - \log \Lambda_{t,\ve}.
\]
Hence
$$ \log \Lambda_{t,\ve}=P(t\Phi^H - \tau p^{H}(t)) - P(t\Phi-\tau p^H(t))$$
and we conclude that  $\nu_{Y,\ve}$ is a Gibbs measure for  $\Psi_t^H$.  Then by Theorem~\ref{thm:Gibbs}, we know that this is an equilibrium state for $\Psi_t^H$ provided the integral of $\Psi_t^H$ is finite.  This latter fact follows from Lemma~\ref{lem:finite} and since whenever $p^{H}(t)=0$, then $t>\frac{\gamma}{1+\gamma}$ by Proposition~\ref{prop:tH}. 

Since 
\[
0=P(\Psi_t^{H})=h_{\nu_{Y, \eps}}+\int \Psi_t^{H}~d\nu_{Y, \eps}
=h_{\nu_{Y, \eps}}+\int t\Phi^{H}-\tau p^{H}(t)-P(t\Phi^{H}-\tau p^{H_\ve}(t))~d\nu_{Y, \eps},
\]
Abramov's formula implies 
$$h_{\nu_{H}}+\int \psi_t^{H}~d\nu_{H}=0,$$
i.e.,
\begin{equation}
 h_{\nu_{H}}+\int t\phi^{H}~d\nu_{H}=P(t\Phi^{H}-\tau p^{H}(t))\nu_{H}(Y)+p^{H}(t).
 \label{eq:FE nuH}
\end{equation}
In particular, $\nu_{H}$ is an equilibrium state for $\psi_t^{H}$, which proves (2)
of Theorem~\ref{thm:stable} for Markov holes. 

We will next show that the free energy given by \eqref{eq:FE nuH} varies continuously in $\eps$.  
Using this and the fact that Markov holes are dense in our sequence $( H_\ve )_{\ve \le \ve_0}$, 
we will be able to
conclude that in fact \eqref{eq:FE nuH} holds for non-Markov holes as well.

For a given $t\in (0,1)$,  note that setting
$$\eps_t:=\inf \{\eps \le \eps_0 :P(t\Phi^{H_\eps}) > 0\},$$ we have $P(t\Phi^{H_{\eps_t}})=0$
whenever $\ve_t$ is finite.  Moreover, set $\eps_1=0$.

We will address two cases, noting that for some values of $t \le t^{H_{\ve_0}}$, Case 2 will be empty.

{\em Case 1:  $\eps\le \eps_t$.}  In this case  $p^{H_\eps}(t)=h_{\nu_{H_\eps}}+\int t\phi^{H_\ve}~d\nu_{H_\eps}$, and we can show the continuity of this quantity via the Implicit Function Theorem and the fact that in this case $s=p^{H_\eps}(t)$ is the unique solution to $P(t\Phi^{H_\ve}-\tau s)=0$.

{\em Case 2: $\eps>\eps_t$.}  In this case $p^{H_\eps}(t)=0$ and $h_{\nu_{H_\eps}}+\int t\phi^{H_\ve}~d\nu_{H_\eps}= P(t\Phi^{H_\ve})\nu_{H_\eps}(Y)$, so we need to prove continuity of induced pressure and $\nu_{H_\eps}(Y)$ in $\eps$.

\begin{lemma} 
Fix $t\in (0, 1]$ and $s\in \R$.  Let $( H_\eps )_{\eps\in [0, \eps_0]}$ 
be a collection of holes centered at $z\in I$ with $\eps<\eps'$ implying 
$H_\eps\subset H_{\eps'}$.  
If $P(t\Phi- \tau s)<\infty$, then $\eps\mapsto P(t\Phi^{H_\eps}-\tau s)$ 
is continuous on $[0, \eps_0]$.
\label{lem:eps cts}
\end{lemma}

\begin{proof}
We will set $s=0$ since the proof for any other value is analogous.
Clearly $\eps\mapsto P(t\Phi^{H_\eps})$ is monotone decreasing.  
Since we can approximate any hole $H_\eps$ by elements of the Markov partition,  it is sufficient to consider monotone sequences $(H_{\eps_n})_n$ of Markov holes.  We will prove that for $H_\eps$, for any $\eta>0$ there exists $n\in \N$ such that $P(t\Phi^{H_\eps\sm C_n})-P(t\Phi^{H_\eps})<\eta$ where $C_n$ is any $n$-cylinder.  This  then implies the lemma. 

By assumption on $\Phi$, we have $P(t\Phi) < \infty$.  Also the spectral radius of
$\Lp_{t \Phi}$ on $BV(Y)$ is $e^{P(t\Phi)}$.  On the one hand,
\[
| \Lp_{t \Phi} 1 |_{L^1(m_t)} \le \| \Lp_{t\Phi} \|_{BV} \| 1 \|_{BV} < \infty,
\]
while on the other, $\int_Y \Lp_{t \Phi} 1 \, dm_t = \int_Y e^{\tau p(t)} \, dm_t$, since
$m_t$ is $(t\Phi - \tau p(t))$-conformal.  Thus $e^{\tau p(t)} \in L^1(m_t)$.

We claim that $\hLp_{t\Phi^H_{\ve}}$ is a perturbation of $\Lp_{t\Phi}$.  This follows 
precisely as in the proof of Lemma~\ref{lem:small pert} since
\[
| \hLp_{t\Phi} \psi - \hLp_{t\Phi^{H_\ve}} \psi |_{L^1(m_t)} \le
| \psi |_{\infty} \int_{\tH_\ve \cup F^{-1} \tH_\ve} e^{\tau p(t)} \, dm_t ,
\]
and this expression tends to 0 with $\ve$ since $e^{\tau p(t)} \in L^1(m_t)$.

Since $P(t\Phi) >0$, it follows from the above that in the perturbative regime
$\ve \le \ve_0$, we have $P(t\Phi^{H_\eps})>-t\log 2$.   
Let $\delta:=P(t\Phi^{H_\eps})+t\log 2>0$.  Since $DF\ge 2$, we have  $t\Phi\le -t\log 2$ for $t\ge 0$, so $$t\Phi-P(t\Phi^{H_\eps})\le -\delta.$$

So for large $n$, subtracting an $n$-cylinder $C_n$ from an existing hole adds at most 
$$
e^{t S_n \Phi(x)}\le e^{-\delta n + nP(t \Phi^{H_\ve})}
$$
to $Z_n(t\Phi^{H_\eps})$ for $x \in C_n$.  Briefly denote $Z_n=Z_n(t\Phi^{H_\eps})$.  Subtracting our $n$-cylinder changes
$\frac1n\log Z_n$ to at most $\frac1n\log(Z_n+e^{ n(P(t\Phi^{H_\eps})-\delta)})=\frac1n\log Z_n +\frac1n\log\left(1+\frac{ e^{n\left(P(t\Phi^{H_\eps})-\delta\right)}}{Z_n}\right)$.
Since $e^{-nP(t\Phi^{H_\eps})}Z_n$ has subexponential growth, this alters $\frac1n\log Z_n$ by at most $\zeta(n)$ for $\zeta(n) =Ce^{-\frac\delta2 n}$ for $C>0$ independent of $n$.
 
Clearly, $\displaystyle \lim_{k\to \infty} \tfrac1{mk}\log Z_{mk}(\Psi)=P(\Psi)$ for any H\"older continuous potential $\Psi$. 
Let  $\hat Z_n=Z_n(t\Phi^{H_\eps\sm C_n})$.  Then for $K$ a distortion constant, 
and $k\in \N$, $\hat Z_{nk}\le K\hat Z_n^k$.  Thus
$$\frac1{nk}\log \hat Z_{nk}\le \frac K{nk}+\frac1n\log \hat Z_n\le  \frac K{nk} +\frac1{n}\log Z_n +\zeta(n).$$
So letting $k\to \infty$, we see that for $n$ large, $P(t\Phi^{H_\eps})$ and $P(t\Phi^{H_\eps\sm C_n})$ differ by  $O(e^{-\frac\delta2 n})$.
\end{proof}

We use the lemma to prove Cases 1 and 2.

{\em Cae 1.}  Noting that for fixed $t\in \R$ and $\eps\ge 0$, $s\mapsto P(t\Phi^{H_\eps}-\tau s)$ 
is strictly decreasing, Lemma~\ref{lem:eps cts} and the Implicit Function Theorem (in its continuous version) imply that $\eps\mapsto p^{H_\eps}(t)$ is continuous for $\eps\le \eps_t$.

{\em Case 2.} In this case $p^{H_\eps}(t)=0$ so the free energy of $\nu_{H_\ve}$ is given by
$P(t\Phi^{H_\ve}) \nu_{H_\ve}(Y)$, according to \eqref{eq:FE nuH}.
Lemma~\ref{lem:eps cts} gives continuity of $P(t\Phi^{H_\ve})$, while
continuity of $\nu_{H_\ve}(Y) = \int \tau d\nu_{Y,\ve}$ follows from Lemma~\ref{lem:tau}
for $\ve \le \ve_0$.


\subsection{Analyticity of the free energy for $t > t^H$}
\label{ssec:analytic}

In this section we focus on the pressure (or free energy) of the measures $\nu_{H_\ve} = \nu_{H_\ve, t}$ constructed
previously.
Fixing $\ve$ sufficiently small, we drop the subscript $\ve$ and denote $H_\ve$
simply by $H$.  We will show that the pressure
\[
h_{\nu_{H,t}} + \int t \phi^H \, d\nu_{H,t},
\]
is an analytic function of $t$ for $t > t^H$, thus proving item (4) of Theorem~\ref{thm:stable} and
completing the proof of that theorem.

For brevity, we will denote $\hLp_{t \Phi^H}$ by $\hLp_t$ in this section (recall that $p^H(t) = 0$
for $t \ge t^H$).
Fix $s = \frac{t^H + 1}{2}$ and for $\psi \in BV(Y)$, write 
\[
\hLp_t \psi(x)  = \hLp_{s+ (t-s)} \psi(x) = \sum_{y \in F^{-1}(x)} \psi(y) e^{s\Phi^H(y)} e^{(t-s) \Phi(y)} 
= \sum_{n = 0}^\infty \sum_{y \in F^{-1}(x)} \psi(y) e^{s\Phi^H(y)} \frac{(t-s)^n}{n!} \Phi(y)^n,
\]
and the interchange of sums is justified once we show that the series of operators
$\sum_{n =0}^\infty L_{s,n}$, where $L_{s,n}\psi = \hLp_s (\psi e^{s \Phi^H} \frac{(t-s)^n}{n!} \Phi^n)$, converges in the 
operator norm on $BV(Y)$.  This will follow from our next lemma.

\begin{lemma}
\label{lem:analytic}
There exists $C_H > 0$, depending only on $H$ and $f$, but not $t$, such that
$\| L_{s,n} \|_{BV} \le C |t-s|^n$.  Thus
$\hLp_t$ is an analytic perturbation of $\hLp_s$ for $t \in (t^H,1)$.  Moreover,
the perturbation is continuous on the closure $[t^H, 1]$.

Similarly, the operator for the closed system $\Lp_t = \Lp_{t\Phi}$ is an analytic perturbation of $\Lp_{s\Phi}$ for
$t \in (t^H,1)$ and continuous on the closure. 
\end{lemma}

\begin{proof}
We will show the estimates for $L_{s,n}$ using the punctured potential $\Phi^H$.  
The corresponding estimates for $\Phi$ are nearly identical and are omitted.

Let $K_{j, i} = [a_{j,i}, b_{j,i}]$ denote the images of intervals of monotonicity for $\hF$ for which
$\tau = j$ and
let $\xi_j$ denote the inverse branch of $F$ on $K_{j,i}$.  Note that by assumption on $H$, 
there are a finite and uniformly bounded number
of $i$ for each $j$, and if $H = \emptyset$, then there is only one $i$ per $j$, thus the inverse
$\xi_j$ depends only on $j$ and not on $i$.  Now for
$\psi \in BV(Y)$, we estimate following \eqref{eq:var split} and \eqref{eq:var 1},
\[
\begin{split}
\bigvee_Y & L_{s,n} \psi  \le \sum_{j,i} \bigvee_{K_{j,i}} \Big(\psi e^{s\Phi} \frac{(t-s)^n}{n!}
\Phi^n \big) \circ \xi_j \\
& \qquad + \sum_{j,i} \Big|\psi e^{s\Phi} \frac{(t-s)^n}{n!}
\Phi^n \big| \circ \xi_j(a_{j,i}) + \Big|\psi e^{s\Phi} \frac{(t-s)^n}{n!}
\Phi^n\Big| \circ \xi_j(b_{j,i}) \\
& \le \frac{|t-s|^n}{n!} \left[ \sum_{j,i} \bigvee_{\xi_j(K_{j,i})} \psi \cdot \sup_{\xi_j(K_{j,i})} e^{s\Phi} |\Phi|^n
+ \bigvee_{\xi_j(K_{j,i})} e^{s\Phi} \Phi^n \cdot \sup_{\xi_j(K_{j,i})} |\psi| 
 + \sum_{j,i} 2 \sup_{\xi_j(K_{j,i})} |\psi| e^{s\Phi} |\Phi|^n \right] \\
& \le \frac{|t-s|^n}{n!} \sum_{j,i} \left( \bigvee_{\xi_j(K_{j,i})} \psi + 3 \sup_{\xi_j(K_{j,i})} |\psi| \right) \sup_{\xi_j(K_{j,i})} e^{s\Phi} |\Phi|^n,
\end{split}
\]
where in the last line we have used the fact that $e^{s\Phi} |\Phi|^n$ is monotonic on each interval
$\xi_j(K_{j,i})$ to replace the variation of the function by its supremum.
Since $\sup_{\xi_j(K_{j,i})} \psi \le C \| \psi \|_{BV}$, we estimate,
\begin{equation}
\label{eq:L bound}
\bigvee_Y L_{s,n} \psi \le C \| \psi \|_{BV} \frac{|t-s|^n}{n!} \sum_{j,i} \sup_{\xi_j(K_{j,i})} e^{s\Phi} |\Phi|^n .
\end{equation}
Standard distortion estimates (see (D1) and (D2) of Section~\ref{apply tower}) imply that
on each $\xi_j(K_{j,i})$, $\Phi = - \log Df^j \sim \log (j+1)^{(1+\frac{1}{\gamma})}$.  Thus,
\[
\sup_{\xi_j(K_{j,i})}e^{s\Phi} |\Phi|^n \le C (j+1)^{-s(1+\frac{1}{\gamma})} (\log (j+1))^n .  
\]
Now by \eqref{eq:tH}, we have $s > t^H > \frac{2\gamma}{1+\gamma}$ so that
$s(1+\frac{1}{\gamma}) > 2$.  Thus the sum in \eqref{eq:L bound} is bounded by
\[
\sum_{j,i} \sup_{\xi_j(K_{j,i})} e^{s\Phi} |\Phi|^n
\le C \sum_{j \ge 2} j^{-s(1+\frac{1}{\gamma})} (\log j)^n \le C' \int_1^\infty x^{-2} (\log x)^n \, dx 
\le C' n!,
\]
where we have integrated by parts $n$ times and used the fact that the number of intervals
$K_{j,i}$ is uniformly bounded for each $j$.  Putting this estimate together with \eqref{eq:L bound}
yields
\[
\sum_{n =0}^\infty \| L_{s,n} \|_{BV} \le C \sum_{n=0}^\infty |t-s|^n < \infty,
\]
since $|t-s| < 1$.  This proves the claimed bound on $\| L_{s,n} \|_{BV}$ as well as the
analyticity of $\hLp_t$.
\end{proof}

Using equation \eqref{eq:VP nue} and the fact that $p^H(t) = 0$ for $t \ge t^H$, we have
\begin{equation}
\label{eq:free}
h_{\nu_{H,t}} + \int t \phi^H \, d\nu_{H,t}
= \frac{\log \Lambda_{H, t} + P_t}{\int \tau \, d\nu_{Y, t}},
\end{equation}
where $\Lambda_{H,t}$ is the largest eigenvalue of $\hLp_{t\Phi^H - P_t}$ and $P_t = P(t\Phi)$
is the largest eigenvalue of $\Lp_{t\Phi}$ (both as operators on $BV(Y)$).  Lemma~\ref{lem:analytic}
implies that both $\Lambda_{H,t}$ and $P_t$ vary analytically in $t$ for $t \in (t^H,1)$, and are continuous 
on the closure of this interval.  All that remains to consider is the denominator of the fraction
on the right hand side of the above expression for the free energy of $\nu_{H,t}$.

Let $\tg^H_t$ be the unique probability density in $BV(Y)$ satisfying 
$\hLp_{t\Phi^H - P_t} \tg^H_t = \Lambda_{H,t} \tg^H_t$.
Now since $t\Phi^H - P_t$ is a contracting potential of bounded variation 
(see the proof of Lemma~\ref{lem:conformal}), it follows from \cite{LSV1} that
$\hLp_{t\Phi^H - P_t}$ admits a conformal measure $\eta_t^H$ with eigenvalue
$\Lambda_{H,t}$, supported on $\hY^\infty$, such that 
$\nu_{Y,t} = \tg_t^H \eta_t^H$.  (Alternatively one can derive the existence of $\eta^H_t$
as in Section~\ref{apply tower} in the proof of Theorem~\ref{thm:rate gap}.)

Since $\hLp_{t\Phi^H - P_t}$ enjoys a spectral gap by Proposition~\ref{prop:spectral gap},
we have the following spectral decomposition for all $\psi \in BV(Y)$,
\[
\hLp_{t\Phi^H - P_t} \psi  = \Lambda_{H,t} \Pi_t \psi + \mathcal{R}_t \psi,
\]
where the spectral radius of $\mathcal{R}_t$ is strictly smaller than $\Lambda_{H,t}$,
$\Pi_t \mathcal{R}_t = \mathcal{R}_t \Pi_t = 0$ and $\Pi_t^2 = \Pi_t$.  Moreover,
\begin{equation}
\label{eq:proj}
\Pi_t \psi = \tg^H_t \int \psi \, d\eta^H_t,
\end{equation}
and $\Pi_t$ is analytic as an operator on $BV(Y)$.  It follows that $\int \psi \, d\eta^H_t$ also
varies analytically for $\psi \in BV(Y)$.

While $\tau \notin BV(Y)$, we do have $\hLp_{t\Phi^H - P_t} \tau \in BV(Y)$.  Thus, 
$\Pi_t (\hLp_{t\Phi^H - P_t} \tau)$ is analytic in $t$.  It follows from \eqref{eq:proj} and
 the analyticity of $\tg^H_t$, that
$\int \hLp_{t\Phi^H - P_t} \tau \, d\eta^H_t$ is analytic in $t$.  But by the conformality of $\eta^H_t$,
\[
\int \hLp_{t\Phi^H - P_t} \tau \, d\eta^H_t = \Lambda_{H,t} \int \tau \, d\eta^H_t,
\] 
and the analyticity of $\Lambda_{H,t}$ allows us to conclude that $\int \tau \, d\eta^H_t$ is
analytic for $t \in (t^H,1)$.  Combining this with \eqref{eq:free} proves that the free energy of
$\nu_{H,t}$ varies analytically in $t$.


\section{The swallowing cases}
\label{sec:swallowing}

In this section we describe possible behaviors of escape through a swallowing hole
and explore ways in which transitivity on the  survivor set can fail.
For simplicity, we always assume that $H$ is an interval.
Although the following list is not exhaustive (for example, a hole can be a union
of intervals and divide the open system into any number of transitive components),
further
generalizations will be a combination of the situations described below, with some components
leaking into others, and the component with the slowest rate of escape dominating the rest.

\subsection{Case 1: The hole contains 0}

This is the simplest case, when $H=(0, a)$.   Then the system is uniformly hyperbolic and the 
classical analysis for uniformly expanding maps holds:  the transfer operator acting
on an appropriate space of functions (H\"older continuous in the Markov case, and functions
of bounded variation in the non-Markov case) will have spectral gap, the largest eigenvalue
will correspond to the escape rate and the corresponding eigenfunction will define a
conditionally invariant measure absolutely continuous with respect to the conformal
measure $m_t$.  See \cite{DemYoung06} for references.

\subsection{Case 2: $H=(a,1]$ for $a \le 1/2$.}

In this case, only a single branch will remain for the open system and the dynamics are
trivial: intervals map progressively to the right until they enter the hole.  The escape rate
here is clearly $p(t)$ since $\hI^n = [0, \hf^{-n}(a)]$ and the conformality
of $m_t$ implies that $m_t(J_{k-1}) = e^{tS_k \phi - k p(t)}$ since $f^k(J_{k-1}) = I$ for each $k \ge 1$, so that
\[
\lim_{n \to \infty} \frac 1n \log m_t(\hI^n) = \lim_{n \to \infty} \frac 1n \log m_t(\cup_{k \ge n_a} J_k) = -p(t),
\]
where $n_a$ is the index such that $\hf^{-n}(a) \in J_{n_a}$.

For $t<1$, Theorem~\ref{thm:limit point}
gives the existence of physically relevant accim with escape rate $p(t)$.  For $t=1$, the only
limiting distribution is $\delta_0$ (\cite{DemFer13}).

\subsection{Case 3: Hole to the left of $1/2$, not capturing 0}
If $H = (a, b)$ is a collection of adjacent 1-cylinders of $\P_1$
strictly between 1/2 and 0 then the system 
is divided into two components, $L = [0,a]$ and $R = [b,1]$.  Note that $L$ maps to itself and
to $H$, but not to $R$, while $R$ maps to itself, to $L$ and to $H$.

The dynamics restricted to $R$ is uniformly hyperbolic and 
so the classical results hold regarding escape rate and spectral gap for the associated transfer
operator.
We denote the escape rate out of $R$ by $\log\lambda_t^r$ and the associated
absolutely continuous invariant measure by $\mu_t^r$.  On the other hand,
the dynamics on $L$ are completely dominated by the tail as described in Case 2 above.

There are two
possibilities for how measures evolve in such a system.
If $-\log\lambda_t^r<p(t)$, then escape is slower from $R$ than from $L$ (note that
this cannot happen when $t=1$).  Since $R$ maps to $L$,
the overall escape rate matches that of $R$, $- \log \lambda_t^r$, any limiting
distribution obtained by pushing forward and renormalizing $m_t$, 
i.e. as a limit of $\hf^n_* m_t/|\hf^n_* m_t|$ will be 
fully supported on $I \setminus H$.  In fact, this limiting distribution must be a multiple of
$\mu_t^r$ on $R$.

The second possibility is that $-\log \lambda_t^r > p(t)$.  In this case, the escape rate from
$L$ is slower than the escape rate from $R$, but since $L$ does not map to $R$, any
limiting distribution obtained by pushing forward and renormalizing $m_t$ must be
identically 0 on $R$.  Any such limit point will have a density on $L$ with escape rate
$p(t)$ for $t<1$.  For $t=1$, $\delta_0$ will be the only limit point.

Note that both possibilities will occur as $t$ varies in $[0,1]$.  Due to the uniformly hyperbolic
behavior on $R$, for Markov holes we have the variational equation for the open system,
$- \log \lambda_t^r = p(t) - p^H_r(t)$, where $p^H_r(t)$ is the punctured pressure for the open
system on $R$.  So the two cases above can be equivalently characterized as 
$p^H_r(t) > 0$ or $p^H_r(t) < 0$.  If $t=0$, we have $p^H_r(0)$ equal to the topological entropy
of the survivor set on $R$, which is positive since $R$ contains a horseshoe; if $t=1$, then
$p^H_r(1) < 0$ since $p(1) = 0$ and the rate of escape from $R$ with respect to Lebesgue measure
is exponential.  

\subsection{Case 4: $H = (1/2, a)$, for some $1/2 < a < 1$.}

Since $H$ contains a right neighborhood of $1/2$, the set $[b,1]$ is invariant for the
open system, where $b = f(a) > 0$.  As in Case 3, the system splits into two components,
$L = [0,b)$ which is single branched and dominated by the tail, and $R= [b,1]$, which
is uniformly hyperbolic. 
The difference is that now the leakage is from $L$ to $R$ and not from $R$ to $L$.  

Considering again the two possibilities, if $-\log\lambda_t^r<p(t)$
then as above, the overall escape rate matches $-\log \lambda_t^r$, but now the
limiting distribution obtained from $\hf^n_* m_t/|\hf^n_* m_t|$ is simply
$\mu_t^r$ on $R$ and 0 on $L$. 

On the other hand, if $-\log \lambda_t^r > p(t)$, then the overall escape rate will be $p(t)$,
and densities will evolve on $L$ as in Case 2 above.  
For $t < 1$, limit points of $\hf^n_* m_t/|\hf^n_* m_t|$
will be fully supported on $I \setminus H$ and the density on $R$ will be a pushforward 
average of
the limiting density on $L$.  For $t=1$, $\delta_0$ will be the only limit point.

\subsection{Case 5: $H = (3/4,7/8)$.}

This is anomalous in the sense that if the hole were $H = (3/4, 1]$, then $H$ would be
non-swallowing, so we see that a subset of a non-swallowing hole can be
swallowing.

In this case, the system once again divides into two transitive components,
$L = [0,3/4)$ and $R = [7/8,1]$,  with $L$ mapping to $R$, but $R$ not mapping to $L$.

The dynamics on $L$ acts as if $H$ were simply $(3/4,1]$ and so the theorems of
Section~\ref{main results} apply to this system.  The dynamics on $R$ are simply those of
the doubling map with a single surviving branch, so the classical results for uniformly
hyperbolic systems apply
to this restricted system.  In fact, the escape rate from $R$
is necessarily $p(t) + t\log 2$.   The escape rate from $L$ is 
$- \log \lambda_t^\ell = p(t) - p_\ell^H(t)$ by Corollary~\ref{cor:variational},
where $p_\ell^H(t)$ is the punctured pressure for the open system restricted to $L$.
So for all $t \in [0,1]$, we have that the escape rate from $L$ is strictly slower than the escape
rate from $R$.  Thus, for $t<1$ any limit points of $\hf^n_* m_t/ |\hf^n_* m_t|$ are 
fully supported on $I \setminus H$ and are determined by the limiting behavior of the
open system on $L$.  As usual, for $t=1$ the sole limit point is $\delta_0$.


\end{document}